\documentclass[11pt, oneside]{amsart}

\usepackage{amsmath,amssymb,cite,mathrsfs,tikz-cd}
\usepackage[all]{xy}
\usepackage{soul}
\usepackage{typearea} %%Makes better use of the page, IMHO
\usepackage{hyperref} %%Clickable reference in PDF
\usepackage{color} %% For colorful text

\usepackage{pdfcomment} %% For better comments

\usepackage{epstopdf} 
\usepackage{booktabs}

\newtheorem{teo}{Theorem}[section]
\newtheorem{thm}[teo]{Theorem}
\newtheorem{prop}[teo]{Proposition}
\newtheorem{lemma}[teo]{Lemma}
\newtheorem{cor}[teo]{Corollary}
\newtheorem{constr}[teo]{Construction}
\newtheorem{conj}[teo]{Conjecture}

\newtheorem{defn}[teo]{Definition}

\newtheorem{rmk}[teo]{Remark}

\newtheorem{def-prop}[teo]{Definition-Proposition}

\newtheoremstyle{named}{}{}{\itshape}{}{\bfseries}{.}{.5em}{\thmnote{#3 }#1}
\theoremstyle{named}

\makeatletter
\newcommand{\neutralize}[1]{\expandafter\let\csname c@#1\endcsname\count@}
\makeatother

\newenvironment{conjbis}[1]
  {\renewcommand{\theteo}{\ref*{#1}$'$}%
   \neutralize{teo}\phantomsection
   \begin{conj}}
  {\end{conj}}

  \newenvironment{conjbisbis}[1]
  {\renewcommand{\theteo}{\ref*{#1}$''$}%
   \neutralize{teo}\phantomsection
   \begin{conj}}
  {\end{conj}}

\numberwithin{equation}{section}

  \newcommand{\A}{\mathbb{A}}
  \newcommand{\C}{\mathbb{C}}

  \renewcommand{\P}{\mathbb{P}}
  \newcommand{\Q}{\mathbb{Q}}
  \newcommand{\R}{\mathbb{R}}

  \newcommand{\Z}{\mathbb{Z}}

  \renewcommand{\cong}{\simeq}
  \renewcommand{\bar}{\overline}

\newcommand{\IQbar}{\overline{\Q}}

  \providecommand{\frac}[1]{\operatorname{Frac}(#1)}
  \newcommand{\spec}{\operatorname{Spec}}

  \renewcommand{\ker}{\operatorname{Ker}}

  \newcommand{\codim}{\operatorname{codim}}

  \renewcommand{\deg}{\operatorname{deg}}

\newcommand{\cA}{\mathcal{A}}

\newcommand{\cL}{\mathcal{L}}
\newcommand{\cM}{\mathcal{M}}
\newcommand{\cO}{\mathcal{O}}

\newcommand{\pullbackcorner}[1][dr]{\save*!/#1-1.7pc/#1:(-1.5,1.5)@^{|-}\restore}

\newcommand\supervisor[1]{\def\@supervisor{#1}}

\newcounter{elno}

\renewcommand{\cong}{\simeq}

\begin{document}
\title[]{Recent developments of the uniform Mordell--Lang conjecture}
\author{Ziyang Gao}

\address{CNRS, IMJ-PRG, 4 place Jussieu, 75005 Paris, France}
\email{ziyang.gao@imj-prg.fr}
%\address{Department of Mathematics and Computer Science, University of Basel, Spiegelgasse 1, 4051 Basel, Switzerland}
%\email{philipp.habegger@unibas.ch}


\maketitle

\begin{abstract} This expository survey is based on my online talk at the ICCM 2020. It aims to sketch key steps of the recent proof of the uniform Mordell--Lang conjecture for curves embedded into Jacobians (a question of Mazur). %This conjecture 
%The proof is by combining Dimitrov--Gao--Habegger's \cite{DGHUnifML} and K\"{u}hne's \cite{KuehneUnifMM}. 
The full version of this conjecture is proved by combining Dimitrov--Gao--Habegger \cite{DGHUnifML} and K\"{u}hne \cite{KuehneUnifMM}. We include in this survey a detailed proof on how to combine these two results, which was implicitly done in \cite{DGHBog} but not explicitly written in existing literature. 
At the end of the survey we state some future aspects.
\end{abstract}

\tableofcontents

%% Section 1
\section{Introduction}
Let $F$ be a field of characteristic $0$. A smooth curve $C$ defined
over $F$ is a geometrically irreducible, smooth, projective curve
defined over $F$. Let $\mathrm{Jac}(C)$ be the Jacobian of $C$.

The goal of this survey is to report the recent development of the following theorem, known as the \textit{Uniform Mordell--Lang Conjecture} for curves embedded into Jacobians. It is a question posed by Mazur \cite[top of pp.234]{mazur1986arithmetic}.

\begin{thm}[Dimitrov--Gao--Habegger + K\"{u}hne]\label{ConjMazur}
Let $g \ge 2$ be an integer. Then there exists a constant $c(g) \ge 1$ with the following property. Let $C$ be a smooth curve of genus $g$ defined over $F$, let $P_0 \in C(F)$, and let $\Gamma$ be a subgroup of $\mathrm{Jac}(C)(F)$ of finite rank $\rho$. Then
\begin{equation}\label{EqBoundMazur}
\#(C(F)-P_0)\cap\Gamma \le c(g)^{1+\rho}
\end{equation}
where $C - P_0$ is viewed as a curve in $\mathrm{Jac}(C)$ via the Abel--Jacobi map based at $P_0$.
\end{thm}

A specialization argument using Masser's result \cite{masser1989specializations} reduces this theorem to $F = \overline{\Q}$; see \cite[Lem.3.1]{DGHBog}. Then Theorem~\ref{ConjMazur} is proved by a combination of the recent works of Dimitrov--Gao--Habegger \cite{DGHUnifML} and K\"{u}hne \cite{KuehneUnifMM}. 
More precisely, Dimitrov--Gao--Habegger's \cite[Thm.1.2]{DGHUnifML} proves Theorem~\ref{ConjMazur} for curves $C$ whose modular height is larger than a number $\delta = \delta(g)$ depending only on the genus $g$, and it can be complemented by K\"{u}hne's \cite[Thm.3]{KuehneUnifMM}
%roughly speaking K\"{u}hne's \cite[Thm.3]{KuehneUnifMM} complements \cite[Thm.1.2]{DGHUnifML}
 because \cite[Thm.3]{KuehneUnifMM} can handle curves with small modular height.% and thus complements \cite[Thm.1.2]{DGHUnifML}.% to obtain the full Theorem~\ref{ConjMazur}.

The way to combine these results to obtain Theorem~\ref{ConjMazur} is not immediate; it was implicitly done in \cite[$\mathsection$2.3 and 2.4]{DGHBog} but did not appear explicitly in literature. In this survey, we include this argument in  $\mathsection$\ref{SectionNGPProof}.

\medskip
There are already some excellent surveys on the topic of the Mordell--Lang Conjecture, for example \cite{HindrySurveyML} and \cite{Mazur:00}, where aspects on function fields can also be found. The current survey focuses on the \textit{uniformity} aspect.

\medskip
Here is a first digest on the conclusion of Theorem~\ref{ConjMazur} and its consequences, including two particularly interesting cases (rational points and algebraic torsion points). In what follows $g \ge 2$.
\begin{enumerate}
\item\label{EnumerateRatPt} \textbf{Rational points.} A particularly important case of Theorem~\ref{ConjMazur} is when $F$ is a number field and $\Gamma = \mathrm{Jac}(C)(F)$. In this case, the Mordell--Weil Theorem says that $\mathrm{Jac}(C)(F)$ is a finitely generated abelian group. 
Thus \eqref{EqBoundMazur} becomes a bound on the number of rational points $\#C(F) \le c(g)^{1+\mathrm{rk}\mathrm{Jac}(C)(F)}$. This improves \cite[Thm.1.1]{DGHUnifML}, which proves $\#C(F) \le c(g,[F:\Q])^{1+\mathrm{rk}\mathrm{Jac}(C)(F)}$. However, $\#C(F)$ must depend on $[F:\Q]$ in some way; in the stronger bound this dependence is encoded in $\mathrm{rk}\mathrm{Jac}(C)(F)$.

In the case of rational points, the most ambitious bound is that $\#C(F)$ is bounded above solely in terms of $g$ and $[F:\Q]$. Caporaso--Harris--Mazur and Pacelli \cite{CaHaMa, Pacelli} proved this bound \textit{assuming a widely open conjecture of Lang}.\footnote{When the number field $F$ is fixed, \cite{CaHaMa, CaHaMaUpdate} proved more: Assuming the widely open \textit{Strong Lang Conjecture}, the cardinality $\#C(F)$ is bounded above solely in terms of $g$ except for finitely many $F$-isomorphic classes of curves $C$ of genus $g \ge 2$.} Techniques developed by Abramovich in \cite{Abramovich:CHM} were used in Pacelli's work.

\item \textbf{Arbitrary finite rank subgroup.} If we pass from rational points to an arbitrary $\Gamma$ and proceed with quasi-orthogonality (Vojta's method), then the bound \eqref{EqBoundMazur} is optimal. Indeed, $\#(C(F)-P_0) \cap \Gamma$ must depend on $g$ and $\rho =\mathrm{rk}\Gamma$. Moreover, the exponent $1+\rho$ is optimal: While it is clear that the exponent should be at least $\rho = \mathrm{rk}\Gamma$ for a general $\Gamma$, we need the extra value $1$ to handle torsion points; see the next case.

\item\label{EnumerateTorPt} \textbf{Algebraic torsion points.} Another particularly interesting case of Theorem~\ref{ConjMazur} is when  $F = \C$ and $\Gamma = \mathrm{Jac}(C)_{\mathrm{tor}}$. In this case \eqref{EqBoundMazur} becomes $\#(C(\C)-P_0)\cap  \mathrm{Jac}(C)_{\mathrm{tor}} \le c(g)$, the \textit{Uniform Manin--Mumford Conjecture} for curves in their Jacobians. In this case, \cite[Thm.3]{KuehneUnifMM} suffices to conclude.

\cite[Thm.3]{KuehneUnifMM} is sometimes known as the \textit{Uniform Bogomolov Conjecture} for curves embedded into Jacobians and is of independent interest. It can be deduced from the \textit{Relative Bogomolov Conjecture} \cite[Conj.1.1]{DGHBog} which is still open. We will have a discussion on this in $\mathsection$\ref{SubsectionRelBog}. 
In this survey, the Uniform Bogomolov Conjecture is merged to be part of the \textit{New Gap Principle}, Theorem~\ref{ThmNGP}; the latter is the major new input which, based on Vojta's proof of the Mordell Conjecture and classical results of many others, leads to Theorem~\ref{ConjMazur}; see $\mathsection$\ref{SubsectionVojtaIntro} and $\mathsection$\ref{SubsectionNGPIntro}.
\end{enumerate}

\medskip

Let us step back and give a historical point of view. The problem is divided into several grades.

\begin{itemize}
\item[-] \textbf{Finiteness.} Faltings \cite{faltings1983endlichkeitssatze} proved the celebrated \textit{Mordell conjecture}, which claims that a smooth curve of genus $g \ge 2$ defined over a number field has only finitely many rational points. %In the notation of Theorem~\ref{ConjMazur}, this result is precisely the \textit{finiteness} of $\#(C(\IQbar)-P_0)\cap \Gamma$ for $\Gamma = \mathrm{Jac}(C)(F)$ and $F$ a number field. 
This is precisely the \textit{finiteness} of $C(F)$, the rational point problem mentioned in \eqref{EnumerateRatPt} above. 
A new proof was later on given by Vojta \cite{Vojta:siegelcompact}, which was simplified by Faltings \cite{Faltings:DAAV} and further simplified by Bombieri \cite{bombieri1990mordell}.  Notice that up to replacing $F$ by a finite extension, this implies the finiteness of $(C(\IQbar)-P_0)\cap \Gamma$ for $\Gamma$ an arbitrary finitely generated subgroup. Raynaud \cite{RaynaudML} explained how to pass from finitely generated subgroups to finite rank subgroups.

As for algebraic torsion points as mentioned in \eqref{EnumerateTorPt} above, Raynaud \cite{Raynaud:MM} proved the Manin--Mumford conjecture, claiming the finiteness of $(C(\C)-P_0)\cap \mathrm{Jac}(C)_{\mathrm{tor}}$.

Faltings \cite{Faltings:DAAV} also further generalized Vojta's proof to allow high dimensional subvarieties of an abelian variety, and Hindry \cite{Hindry:Lang} proved how to pass from finitely generated subgroups to finite rank subgroups in this more general situation. Thus the \textit{Mordell--Lang Conjecture} for abelian varieties was proved by  \cite{Faltings:DAAV} and \cite{Hindry:Lang}.% The finiteness of $(C(\IQbar)-P_0)\cap \Gamma$ as in Theorem~\ref{ConjMazur} was thus proved by.

\item[-] \textbf{Bounds.} Bombieri's proof \cite{bombieri1990mordell} was the first to give effective bounds for the number of rational points. Silverman \cite{Silverman:twists} proved a bound on the number of rational points when $C$ ranges over twists of a given smooth curve. The Bogomolov conjecture, proved by Ullmo \cite{Ullmo} and S.~Zhang \cite{ZhangEquidist}, allows to bound $\#(C(\IQbar)-P_0)\cap \Gamma$ for arbitrary $\Gamma$. The bound thus obtained depends on $C$ and is not explicit.

An explicit upper bound of $\#(C(\IQbar)-P_0)\cap\Gamma$ was later on proved by R\'{e}mond \cite{Remond:Decompte}. Apart from $g$ and $\mathrm{rk}(\Gamma)$, R\'{e}mond's bound depends also on a suitable height of $\mathrm{Jac}(C)$ and the degree of the definition field of $C$. Setting $P_0 \in C(F)$ and $\Gamma = \mathrm{Jac}(C)(F)$ then leads  to a bound of the number for the rational point problem mentioned in \eqref{EnumerateRatPt} above. %, this bound becomes $\#C(F) \le c(g,[F:\Q], h(\mathrm{Jac}(C)))^{1+\mathrm{rk}\mathrm{Jac}(C)(F)}$.
Based on this result, a more explicit bound for the number of rational points was obtained for a particular kind of curves \cite{Remond2010}. R\'{e}mond's bound holds true for high dimensional subvarieties of abelian varieties.

\item[-] \textbf{Uniform bounds.} Let us turn to previous results towards Theorem~\ref{ConjMazur}. In the direction of rational points, \textit{i.e.} the bound $\#C(F) \le c(g,[F:\Q])^{1+\mathrm{rk}\mathrm{Jac}(C)(F)}$ for $F$ a number field mentioned in \eqref{EnumerateRatPt} above. Based on the method of Vojta, David--Philippon~\cite{DaPh:07} proved this bound  if $\mathrm{Jac}(C)$ is contained 
in a power of an
elliptic curve, and David--Nakamaye--Philippon proved this bound for some families of curves \cite{DaNaPh:07}. 
More recently,  Alpoge~\cite{AlpogeRatPt} \cite[Chap.5]{AlpogeThesis} proved that
the average number of rational points on a curve of genus $2$ with a
marked Weierstrass point is bounded. Pazuki \cite{PazukiBound, PazukiHabilitation} showed that a suitable version of the far-reaching Lang--Silverman conjecture implies the desired bound; some unconditional results are obtained in some cases \cite[Cor.1.10]{PazukiBound}.
 The Chabauty--Coleman
approach~\cite{Chabauty, Coleman:effCha} yields estimates under an additional
hypothesis on the rank of Mordell--Weil group. For example, if
$\mathrm{Jac}(C)(F)$ has rank at most $g-3$, Stoll~\cite{Stoll:Uniform} showed
that $\#C(F)$ is bounded solely in terms of $[F:\Q]$ and $g$ if $C$ is
hyperelliptic; Katz--Rabinoff--Zureick-Brown~\cite{KatzRabinoffZB}
later, under the same rank hypothesis, removed the hyperelliptic
hypothesis. %In spirit of the Manin-Demjanenko method \cite[$\mathsection$5.2]{SerreLectureMW}, Checcoli--Veneziano--Viada~\cite{CVV:17} obtained an effective height bound
%under a restriction on 
% the Mordell--Weil rank. 

In the direction of torsion points, \textit{i.e.} $F = \C$ and $\Gamma = \mathrm{Jac}(C)_{\mathrm{tor}}$ mentioned in \eqref{EnumerateTorPt}, the desired bound $\#(C(\C)-P_0)\cap \mathrm{Jac}(C)_{\mathrm{tor}} \le c(g)$ was proved by DeMarco--Krieger--Ye \cite{DeMarcoKriegerYeUniManinMumford} for any genus $2$ curve admitting a degree-two map to an elliptic curve when the Abel--Jacobi map is based at a Weierstrass point. Katz--Rabinoff--Zureick-Brown~\cite{KatzRabinoffZB} proved a weaker bound (in the form of \cite[Thm.1.4]{DGHUnifML}) assuming that $C$ has good reduction at a small prime. \textit{Over function fields}\footnote{Namely, $F$ is an algebraic closure of $k(B)$, where $k$ is an algebraically closed field and $B$ is a smooth curve defined over $k$.} and if $C$ is not isotrivial, Looper--Silverman--Wilms \cite{WilmsUnifMM} proved an explicit bound $c(g) = 112g^2 + 240g + 380$; Wilms's result remains true over positive characteristic.

Stoll~\cite{Stoll:Uniform} showed that a far-reaching conjecture of
Pink~\cite{Pink} on unlikely intersections implies
Theorem~\ref{ConjMazur}.

\item[-] \textbf{Effective Mordell.} This is not directly related to the topic of the current survey. As a question it is  fundamental but currently out of reach.

For the rational point problem as mentioned in \eqref{EnumerateRatPt}, the \textit{effective Mordell conjecture} is to find an explicit bound for the height of $P \in C(F)$ which is linear in terms of a suitable height of $C$; see \cite[Conj.F.4.3.2]{DG2000}. Little is known for this conjecture. In spirit of the Manin-Demjanenko method \cite[$\mathsection$5.2]{SerreLectureMW}, Checcoli, Veneziano, and Viada~\cite{CVV:17, CVV:19, VV:20} have some results on this. There are also $p$-adic approaches (Chabauty--Coleman--Kim, Lawrence--Venkatesh) to this question, for which we refer to the survey \cite{pAdicMordellSurvey}.
%Apart from the effective Mordell conjecture, the Chabauty--Coleman method can also be used to determine the set $C(F)$ in some cases. We refer to the survey \cite{} for this.
%A fundamental yet out of reach question is to determine $C(F)$ for each curve $C$ defined over a number field $F$. We will not go into too much detail for this as it is not directly related to the topic of the current survey.
\end{itemize}

\subsection{Key new ingredients}\label{SubsectionKeyNewIngre}
The proof of Theorem~\ref{ConjMazur} is based on Vojta's approach to prove the Mordell conjecture \cite{Vojta:siegelcompact}. A key new notion to prove Theorem~\ref{ConjMazur} is the \textit{non-degenerate subvarieties} of any given abelian scheme over $\overline{\Q}$; see $\mathsection$\ref{SectionNonDeg}. This notion was introduced by Habegger in \cite{Hab:Special}, and played an important role in the proof of the Geometric Bogomolov Conjecture over characteristic $0$ by Gao--Habegger and Cantat--Gao--Habegger--Xie \cite{GaoHab, CGHX:18}.

In the course of the proof, the following aspects on non-degenerate subvarieties have been developed.
\begin{enumerate}
\item[(i)] The geometric criterion of non-degenerate subvarieties and some related constructions.
\item[(ii)] A height inequality on any given non-degenerate subvariety.
\item[(iii)] An equidistribution result on any given non-degenerate subvariety.
\end{enumerate}
Part (i) was done by Gao in \cite{GaoBettiRank}, part (ii) was done by Dimitrov--Gao--Habegger in \cite{DGHUnifML}, and part (iii) was done by K\"{u}hne in \cite{KuehneUnifMM}. For $1$-parameter families of abelian varieties, (i) and (ii) were proved in \cite{Hab:Special} for fibered power of elliptic surfaces and in \cite{GaoHab} in its full generality.

Dimitrov--Gao--Habegger's \cite[Thm.1.2]{DGHUnifML} uses (i) and (ii). The blueprint was laid down in \cite{DGH1p}, where we used \cite{GaoHab} to prove \cite[Thm.1.2]{DGHUnifML} for $1$-parameter families.

K\"{u}hne's \cite[Thm.3]{KuehneUnifMM} uses (i) and (iii), and implicitly part (ii) as it was used in K\"{u}hne's proof of the equidistribution result. 

More recently, Yuan--Zhang extended the  definition of non-degenerate subvarieties to polarized dynamical systems \cite[$\mathsection$6.2.2]{YuanZhangEqui}. They proved a more general height inequality and a more general equidistribution theorem \cite[Thm.6.5 and Thm.6.7]{YuanZhangEqui}. Their proof uses deep theory of adelic line bundles, arithmetic intersection theory and arithmetic volumes. Notice that in the case of abelian schemes, this leads to new proofs of (ii) and (iii) above.

\subsection{Quick summary of Vojta's method}\label{SubsectionVojtaIntro}
Before moving on, let us take a step back to briefly recall Vojta's method. Let $\mathbb{A}_{g,1}$ be the coarse moduli space of principally polarized abelian varieties of dimension $g$. Fix an immersion $\iota \colon \mathbb{A}_{g,1} \rightarrow \P^m_{\overline{\Q}}$. Let $h \colon \P^m_{\overline{\Q}}\rightarrow \R$ be the absolute logarithmic Weil height. In what follows, we will identify $\mathbb{A}_{g,1}$ with its image under $\iota$.

Let $\hat h\colon \mathrm{Jac}(C)(\IQbar)\rightarrow [0,\infty)$ denote the
N\'eron--Tate height attached to a symmetric and ample line bundle on $\mathrm{Jac}(C)$. 
We divide $C(\IQbar) \cap \Gamma$ into two parts:
\begin{itemize}
\item Small points $\left\{P \in C(\IQbar) \cap \Gamma
: \hat{h}(P) \le B(C)  \right\}$;
\item Large points $\left\{P \in C(\IQbar) \cap \Gamma : \hat{h}(P) > B(C) \right\}$
\end{itemize}
where $B(C)$ is allowed to depend on a suitable height of $C$. 
Denote by $[\mathrm{Jac}(C)]$ the point in $\P^m(\IQbar)$ induced by $\mathrm{Jac}(C)$ and $\iota$. 
It
turns out that
we can take $B(C)=c_0 \max\{1,h([\mathrm{Jac}(C)])\}$
for some $c_0 = c_0(g) > 0$. The constant $c_0$ is chosen in a
way that accommodates  both the \textit{Mumford inequality} and
the \textit{Vojta inequality}. Combining these
two inequalities yields an upper bound on the number of large points
by $c_1(g)^{1+\rho}$, see for example Vojta's~\cite[Thm.6.1]{Vojta:siegelcompact} in the important case where $\Gamma$ is the
group of points of $\mathrm{Jac}(C)$ rational over
 a number field or more generally in the work of
 David--Philippon~\cite{DPvarabII,DaPh:07} and R\'emond~\cite{Remond:Decompte}. Moreover, in the case for rational points, de
Diego~\cite{deDiego:97} proved that the number of large points %for a family of smooth curves 
is at most $c(g)
7^{\rho}$, where $c(g)>0$ depends only on $g$; the value $7$ had already appeared in Bombieri's work \cite{bombieri1990mordell}. 
Recently, Alpoge \cite{AlpogeRatPt} \cite[Thm.6.1.1]{AlpogeThesis} improved  $7$
to $1.872$ and, for $g$ large enough, even to $1.311$. 

David--Philippon~\cite{DPvarabII,DaPh:07} also showed that an appropriate lower bound on the essential minimum of subvarieties of $\mathrm{Jac}(C)$ yields a bound on the number of small points.

%A key point to prove Theorem~\ref{ConjMazur} is a height inequality established in \cite{}, as well as the criterion . In the work of Dimitrov--Gao--Habegger, this inequality 

\subsection{A New Gap Principle}\label{SubsectionNGPIntro}
As said above, the combination of \cite{DGHUnifML} and \cite{KuehneUnifMM} to imply Theorem~\ref{ConjMazur} is not immediate. This is done via proving the following New Gap Principle. We refer to Theorem~\ref{ThmNGP} for the precise statement.

Roughly speaking, we find positive constants $c_1$ and $c_2$ that depend only on $g$  such that each $P \in C(\overline{\Q})$ satisfies
 \begin{equation}
 \#\left\{Q \in C(\overline{\Q}) : \hat{h}(Q-P) \le c_1\max\{1, h([\mathrm{Jac}(C)])\} \right\} < c_2. 
\end{equation}

Up to some finite set of uniformly bounded cardinality, this New Gap Principle is precisely \cite[Prop.7.1]{DGHUnifML} provided that $h([\mathrm{Jac}(C)]) \ge \delta$ for some $\delta = \delta(g)$. It was explained in \cite[Prop.2.3 and Prop.2.5]{DGHBog} how this extra condition on $h([\mathrm{Jac}(C)])$ can be removed by assuming the  \textit{Relative Bogomolov Conjecture}. Following a similar proof, we show in $\mathsection$\ref{SectionNGPProof}  that this extra condition on $h([\mathrm{Jac}(C)])$ can also be removed by using \cite[Thm.3]{KuehneUnifMM}, which itself can be deduced from the Relative Bogomolov Conjecture.

Here is a sketch. The proof of \cite[Prop.7.1]{DGHUnifML} shows that the bound above holds true (for any curve) with $c_1 \max\{1, h([\mathrm{Jac}(C)])\}$ replaced by $c_1 \max\{1, h([\mathrm{Jac}(C)])\} - c_3$, for some $c_3 = c_3(g)$. Hence what remains to be done is to remove this constant term $c_3$. This is exactly what \cite[Thm.3]{KuehneUnifMM} ($\#\{Q \in C(\overline{\Q}) : \hat{h}(Q-P) \le c_3 \} < c_2$ up to adjusting $c_3$ and $c_2$ appropriately)  accounts for.%; this removes the constant term in Dimitrov--Gao--Habegger.

\subsection{Structure of the survey}
In $\mathsection$\ref{SectionHtMachine}, we give a quick recall to the Height Machine. In $\mathsection$\ref{SectionVojta}, we briefly go through the key ingredients of Vojta's approach to prove the Mordell conjecture. In particular, we will summarize the classical results on bounding the number of large points, by Mumford's and Vojta's inequality; in the end we state the classical results in the relative setting. 

In $\mathsection$\ref{SectionNGP}, we give our setup involving several universal families, and state the New Gap Principle. In $\mathsection$\ref{SectionBetti}, we recall the Betti map and Betti form, which are fundamental tools to study non-degeneracy.

In $\mathsection$\ref{SectionNonDeg}--\ref{SectionEqdistr}, we explain the three key new ingredients listed in $\mathsection$\ref{SubsectionKeyNewIngre}, each occupying a section. We will state the main results and focus on presenting how they are applied. 
In $\mathsection$\ref{SectionNonDeg}, we give the definition of non-degenerate subvarieties in two equivalent ways and explain how to construct non-degenerate subvarieties from given varieties; this construction is  important in applications. In $\mathsection$\ref{SectionHtIneqEqdistr}, we state the height inequality and give an example on how it is used in Diophantine Geometry. This example is in line with \cite[Prop.7.1]{DGHUnifML}; a minor improvement is that it provides more explicit constants. In $\mathsection$\ref{SectionEqdistr}, we state the equidistribution result, and give a detailed proof on how it is used to prove \cite[Thm.3]{KuehneUnifMM}.

We will give a detailed proof of the New Gap Principle in $\mathsection$\ref{SectionNGPProof} using the height inequality and the equidistribution result from the previous section. The proof is in line with \cite[Prop.2.3]{DGHBog}. Then we shortly explain how to conclude for Theorem~\ref{ConjMazur}.

We will discuss some related open problems in $\mathsection$\ref{SectionOtherAspects}. In $\mathsection$\ref{SubsectionRelBog}, we state the \textit{Relative Bogomolov Conjecture} and explain how it implies \cite[Thm.3]{KuehneUnifMM}. In $\mathsection$\ref{SubsectionHighDimUnifML}, we discuss briefly the \textit{Uniform Mordell--Lang Conjecture} for high dimensional subvarieties of abelian varieties. We give several equivalent formulations of this conjecture and prove their equivalence. We also formulate (without proof) the generalized \textit{New Gap Principle}.% and  that it implies the Uniform Mordell--Lang Conjecture in general based on classical results of R\'{e}mond.% In the end we show that the Uniform Mordell--Lang Conjecture can be deduced from this generalized New Gap Principle.

\subsection*{Acknowledgements}
I would like to thank my collaborators Philipp Habegger and Vesselin Dimitrov on this project, and I would like to thank Yves Andr\'{e}, Serge Cantat, Pietro Corvaja, Junyi Xie, and Umberto Zannier for collaboration on related problems. I would like to thank Lars K\"{u}hne for sending me his preprints \cite{KuehneUnifMM, KuehneRBC}. I would like to thank Dan Abramovich, Marc Hindry, and Barry Mazur for their encouragement, comments, and suggestions on the conjectures about the high dimensional subvarieties discussed in $\mathsection$\ref{SubsectionHighDimUnifML}. 
I would like to thank Camille Amoyal, Laura DeMarco, Gabriel Dill, Philipp Habegger, Marc Hindry, Lars K\"{u}hne, Myrto Mavraki, Barry Mazur, Fabien Pazuki, Yunqing Tang, Xinyi Yuan, and Umberto Zannier for their valuable comments on a previous version of the manuscript. I would like to thank Gabriel Dill for providing me the references \cite{RaynaudML} and for an argument to fix a gap of Lemma~\ref{LemmaConjHighDimMLStrWeakEqui} in a previous version. I would like to thank Ga\"{e}l R\'{e}mond for providing me the two examples at the end of the survey; they helped me achieve the current formulation for Conjecture~\ref{ConjNGPHighDim}. 
This project has received funding from the European Research Council (ERC) under the European Union’s Horizon 2020 research and innovation programme (grant agreement n$^{\circ}$ 945714).

%% Section 2
\section{The Height Machine}\label{SectionHtMachine}
In this section, we recollect some basic facts on the Height Machine and the canonical height functions on an abelian variety. There are many standard textbooks on this, for example \cite{BG} and \cite{DG2000}.

All varieties, line bundles and morphisms in this section are assumed to be defined over $\IQbar$.

\subsection{Naive height function on projective spaces}
We refer to \cite[Chap.1]{BG} and \cite[B.1 and B.2]{DG2000}.

We start with the simplest case. Let $x \in \P^1(\Q)$. There is a unique way to write $x$ as $[a:b]$ with $a, b \in \Z$ such that we are in one of the following two cases:
\begin{itemize}
\item $a=0$, $b=1$ or $a=1$, $b=0$;
\item $a > 0$ and $b \not=0 $ are coprime.
\end{itemize}
Then the \textit{height} of $x$ is defined to be $0$ in the first case and $\log\max\{|a|,|b|\}$ in the second case, with $|\cdot|$ being the standard absolute value.

Now let us generalize this definition to $\P^n(K)$ for any integer $n \ge 1$ and any number field $K$.

A \textit{place} of a number field $K$ is an absolute value
$|\cdot|_v \colon K\rightarrow [0,\infty)$ whose restriction to $\Q$ is
either the standard absolute value or a $p$-adic absolute value for
some prime $p$ with $|p|=p^{-1}$. Let $K_v$ be the completion of $K$ at $v$ with respect to $|\cdot|_v$. 
Set $d_v = [K_v:\R]$ in the former and
$d_v = [K_v:\Q_p]$ in the latter case.
The \textit{absolute logarithmic Weil height}, or just \textit{height}, of a
point $x = [x_0:\ldots:x_n]\in \mathbb{P}^n(K)$ with
$x_0,\ldots,x_n\in K$ is 
\begin{equation}
h(x) = \frac{1}{[K:\Q]}\sum_{v} d_v\log\max\{|x_0|_v,\ldots,|x_n|_v\}
\end{equation}
where the sum runs over all places $v$ of $K$.
The value $h(x)$ is independent of the choice of projective
coordinates by the Product Formula, and for $x \in \P^1(\Q)$ this $h(x)$ coincides with the height defined in the previous paragraph. 
Moreover, the height
does not change when replacing $K$ by another number
field that contains the coordinates of $x$.
Therefore, $h(\cdot)$ is a well-defined function
\begin{equation}\label{EqNaiveHt}
h \colon \P^n(\IQbar) \rightarrow [0,\infty).
\end{equation}
We call this function the \textit{naive height function} on $\P_{\IQbar}^n$.

\subsection{Height Machine}
We refer to \cite[Chap.2]{BG} and \cite[B.3]{DG2000}.

Let $X$ be an irreducible projective variety defined over $\IQbar$. Denote by $\R^{X(\IQbar)}$ the set of functions $X(\IQbar) \rightarrow \R$, and by $O(1)$ the subset of bounded functions.

The \textit{Height Machine} 
associates to each line bundle $L \in \mathrm{Pic}(X)$ a unique class of functions $\R^{X(\IQbar)}/O(1)$, \textit{i.e.} a map
\begin{equation}\label{EqHtMachine}
\mathbf{h}_X \colon \mathrm{Pic}(X) \rightarrow \R^{X(\IQbar)}/O(1), \quad L \mapsto \mathbf{h}_{X,L}.
\end{equation}
Let $h_{X,L} \colon X(\IQbar) \rightarrow \R$ a representative of the class $\mathbf{h}_{X,L}$; it is called a \textit{height function associated with $(X,L)$}. 

One can construct $h_{X,L}$ as follows. In each case below, $h_{X,L}$ depends on some extra data and hence is not unique. However, it can be shown that any two choices differ by a bounded functions on $X(\IQbar)$, and thus the class of $h_{X,L}$ is well-defined.
\begin{enumerate}
\item[(i)] If $L$ is very ample, then the global sections of $L$ give rise to a closed immersion $\iota \colon X \rightarrow \P^n$ for some $n$. Set $h_{X,L} = h\circ \iota$, with $h$ the naive height function on $\P^n$ from \eqref{EqNaiveHt}.
\item[(ii)] If $L$ is ample, then $L^{\otimes m}$ is very ample for some $m \gg 1$. Set $h_{X,L} = (1/m) h_{X,L^{\otimes m}}$.
\item[(iii)] For an arbitrary $L$, there exist ample line bundles $L_1$ and $L_2$ on $X$ such that $L \cong L_1 \otimes L_2^{\otimes -1}$. Set $h_{X,L} = h_{X,L_1} - h_{X,L_2}$.
\end{enumerate}

Here are some basic properties of the Height Machine. These properties, or more precisely properties (i)-(iii), also uniquely determine \eqref{EqHtMachine}.
\begin{prop}
We have
\begin{enumerate}
\item[(i)] (Normalization) Let $h$ be the naive height function from \eqref{EqNaiveHt}. Then for all $x \in \P^n(\IQbar)$, we have
\[
h_{\P^n,\cO(1)}(x) = h(x) + O(1).
\]
\item[(ii)] (Functoriality) Let $\phi \colon X \rightarrow Y$ be a morphism of irreducible projective varieties and let $L$ be a line bundle on $Y$. Then for all $x \in X(\IQbar)$, we have
\[
h_{X,\phi^*L}(x) = h_{Y, L}(\phi(x)) + O(1).
\]
\item[(iii)] (Additivity) Let $L$ and $M$ be two line bundles on $X$. Then for all $x \in X(\IQbar)$, we have
\[
h_{X, L\otimes M}(x) = h_{X,L}(x) + h_{X,M}(x) + O(1).
\]
\item[(iv)] (Positivity) If $s \in H^0(X,L)$ is a global section, then for all $x \in (X\setminus \mathrm{div}(s))(\IQbar)$ we have
\[
h_{X,L}(x) \ge O(1).
\]
\item[(v)] (Northcott property) Assume $L$ is ample. Let $K_0$ be a number field on which $X$ is defined. Then for any $d \ge 1$ and any constant $B$, the set
\[
\{x \in X(K) : [K:K_0] \le d, ~ h_{X,L}(x) \le B\}
\]
is a finite set.
\end{enumerate}
\end{prop}
The $O(1)$'s that appear in the proposition depend on the varieties, line bundles, morphisms, and the choices of the representatives in the classes of height functions. But they are independent of the points on the varieties.

In applications, we often do not have projective varieties, but only quasi-projective varieties. For example, $f \colon X \rightarrow Y$ a morphism between quasi-projective varieties. Then $f$ can be viewed as a rational map $\xymatrix{ X \ar@{-->}[r] & Y}$. In this case, we have the following result of Silverman.

\begin{thm}\label{ThmHtTriangularIneq}
Let $\xymatrix{f\colon  X \ar@{-->}[r] & Y}$ be a generically finite rational map between projective varieties. Let $L$ be an ample line bundle on $X$ and $M$ be an ample line bundle on $Y$. Then
\begin{enumerate}
\item[(i)] there exist constants $c_1>0$ and $c_2$ such that $h_{Y,M}(f(x)) \le c_1 h_{X,L}(x) + c_2$ for all $x \in X(\IQbar)$ such that $f(x)$ is well-defined;
\item[(ii)] there exist constants $c_1' > 0$, $c_2'$ and a Zariski open dense subset $U \subseteq X$ such that $h_{Y,M}(f(x)) \ge c_1' h_{X,L}(x) - c_2'$ for all $x \in U(\IQbar)$.
\end{enumerate}
\end{thm} 

While part (i) \cite[Lem.4]{Silverman:heightest11} is an easy application of the \textit{triangular inequality}, part (ii) \cite[Thm.1]{Silverman:heightest11} is highly non-trivial.

\subsection{N\'{e}ron--Tate height function on abelian varieties}
We refer to \cite[Chap.9]{BG} and \cite[B.5]{DG2000}.

In this subsection, we turn to abelian varieties. Let $A$ be an abelian variety and $L$ be a line bundle on $A$. Assume furthermore that $L$ is symmetric, \textit{i.e.} $L \cong [-1]^*L$.

The Tate Limit Process provides a distinguished representative in the class of height functions associated with $(A,L)$ provided by the Height Machine \eqref{EqHtMachine}. Indeed, let $h_{A,L}$ be a representative of this class, and set
\begin{equation}\label{EqNTHeight}
\hat{h}_{A,L}(x) := \lim_{N\rightarrow \infty} \frac{h_{A,L}([2^N]x)}{4^N}.
\end{equation}
The function $\hat{h}_{A,L}$ is called the \textit{canonical height} or \textit{N\'{e}ron--Tate height} on $A$ with respect to $L$. It satisfies, and is uniquely determined by, the following properties.\footnote{In particular, $\hat{h}_{A,L}$ does not depend on the choice of the representative $h_{A,L}$ in \eqref{EqNTHeight}.}
\begin{prop}
We have, for all $x \in A(\IQbar)$,
\begin{enumerate}
\item[(i)] $\hat{h}_{A,L}(x) = h_{A,L}(x) + O(1)$;
\item[(ii)] $\hat{h}_{A,L}([N]x) = N^2 \hat{h}_{A,L}(x)$ for all $N \in \Z$.
\end{enumerate}
\end{prop}
Note that (i) implies that $\hat{h}_{A,L}$ is in the same class of height functions as $h_{A,L}$. 
The bounded function $O(1)$ in (i) depends on $A$, $L$ and the choice of the representative $h_{A,L}$ in the class of height functions. 

In practice, we often work with symmetric ample line bundles. We have the following theorem.
\begin{thm}
Assume $L$ is ample. Then
\begin{enumerate}
\item[(i)] $\hat{h}_{A,L}(x) \ge 0$ for all $x \in A(\IQbar)$;
\item[(ii)] $\hat{h}_{A,L}(x) = 0$  if and only if $x \in A(\IQbar)_{\mathrm{tor}}$;
\item[(iii)] $\hat{h}_{A,L}$ extends $\R$-linearly to a positive definite quadratic form $A(\IQbar)\otimes_{\Q}\R \rightarrow \R$, which by abuse of notation is still denoted by $\hat{h}_{A,L}$.
\end{enumerate}
\end{thm}
In the context where the abelian variety is clear, we often abbreviate $\hat{h}_{A,L}$ by $\hat{h}_L$.

We close this section by discussing the relative setting. Let $S$ be an irreducible variety and let 
$\pi \colon \cA\rightarrow S$ be an abelian scheme of relative dimension $g \ge 1$. 
Let $\cL$ be a relatively ample line bundle on $\cA/S$ such that $[-1]^*\cL \cong \cL$. In particular over each $s \in S(\IQbar)$, the line bundle $\cL_s := \cL|_{\cA_s}$ on $\cA_s :=  \pi^{-1}(s)$ is ample and symmetric. The \textit{fiberwise N\'{e}ron--Tate height} with respect to $\cL$ is defined to be
\begin{equation}\label{EqFiberwiseNTHeight}
\hat{h}_{\cA,\cL} \colon \cA(\IQbar) \rightarrow [0,\infty) , \quad x \mapsto \hat{h}_{\cA_{\pi(x)}, \cL_{\pi(x)}}(x).
\end{equation}
In the rest of the paper, we often abbreviate it as $\hat{h}_{\cL}$.

We close this section with the following theorem of Silverman--Tate; see \cite[Thm.A]{Silverman} and \cite[Thm.A.1]{DGHUnifML}. Let $\cM$ be an ample line bundle on $\overline{S}$, a compactification of $S$. Then the Height Machine provides a height function $h_{\overline{S},\cM} \colon S(\IQbar) \rightarrow \R$.
\begin{thm}
There exists a constant $c = c(\cA/S, \cL, \cM) > 0$ such that %for all $$ we have
\[
|\hat{h}_{\cL}(x) - h_{\cA,\cL}(x)| \le c \max\{1,h_{\overline{S},\cM}(\pi(x))\} \quad \text{ for all }x \in \cA(\IQbar).
\]
\end{thm}

%% Section 3
\section{Vojta's method}\label{SectionVojta}
In this section we give an overview of Vojta's approach to prove  the Mordell Conjecture.% and its generalization when the group of rational points on the Jacobian is replaced by any subgroup of finite rank.% The end of this section 

%We rely on R\'emond's \cite{Remond:Vojtasup,Remond:Decompte} quantitative results. Work of Pazuki
%\cite{Pazuki:uniform} also
%involves completely explicit constants.

Let $A$ be an abelian variety defined over $\IQbar$ equipped
with a very ample and symmetrical line bundle $L$. Then $L$ gives rise to a normalized height function $\hat{h}_L \colon A(\IQbar) \rightarrow [0,\infty)$ as constructed in \eqref{EqNTHeight}.
%We may suppose that
% an  immersion   attached to the line bundle
% realizes $A$ as a  projectively
%normal subvariety of $\P^n$. 

%Let $h$ denote the absolute logarithmic Weil height on
%$\P^n(\IQbar)$.
%We write $h_1$ for an upper bound for the absolute logarithmic projective height of bihomogeneous polynomials that describe the addition morphism on $A$ as a subvariety of $\P^n$, see Section 5~\cite{Remond:Vojtasup} for a precise definition. Tate's Limit Process provides us with a N\'eron--Tate height $\hat h \colon A(\IQbar)\rightarrow\R$. It is well-known there exists a constant $c $, depending on the data introduced above, such that $|h(P)- \hat h(P)|\le c $ for all $P\in A(\IQbar)$. 

For $P,Q\in A(\IQbar)$  we set
$\langle P,Q\rangle = (\hat h_L(P+Q)-\hat h_L(P)-\hat h_L(Q))/2$
and often abbreviate $|P| = \hat h_L(P)^{1/2}$. The notation $|P|$ is
justified by the fact that it induces a norm after tensoring with the reals.

\subsection{Mordell conjecture}
The following fundamental inequalities are the keys to prove the finiteness of rational points on curves of genus at least $2$. They are called the \textit{Mumford inequality} (or \textit{Mumford's Gap Principle}) and the \textit{Vojta inequality}. We state them together.

\begin{thm}\label{ThmMumfordVojtaIneq}
Let $g \ge 2$ and $C$ be a smooth curve of genus at least $2$ defined over $\IQbar$. Let $P_0 \in C(\IQbar)$, and $j \colon C \rightarrow \mathrm{Jac}(C)$ be the Abel--Jacobi embedding via $P_0$. 

There exists a constant $R = R(C, P_0) > 0$ such that the following properties hold true. Consider all distinct points $P, Q\in C(\IQbar)$ such that $|j(Q)| \ge |j(P)| \ge R$ and
\begin{equation}\label{EqMumfordVojtaIneqAngle}
\langle j(P), j(Q) \rangle \ge \frac{3}{4}|j(P)| |j(Q)|,
\end{equation}
then we have
\begin{enumerate}
\item[(i)] (Mumford Inequality) $|j(Q)| \ge 2|j(P)|$.
\item[(ii)] (Vojta Inequality) there exists a constant $\kappa = \kappa(g) > 0$ such that $|j(Q)| \le \kappa |j(P)|$.
\end{enumerate}
\end{thm}

Notice that these two inequalities hold true for all \textit{algebraic points}, not only rational points, on the curve $C$.

\medskip

Let us have a digest of the inequalities. 

We start with the assumptions of the properties. The hypothesis $|j(Q)| \ge |j(P)|$ can be assumed to hold true up to exchanging $P$ and $Q$. The assumption \eqref{EqMumfordVojtaIneqAngle} should be understood to be saying that the angle between $j(P)$ and $j(Q)$ is bounded above by a constant $\cos^{-1}(3/4)$. More precisely, if we fix a subgroup $\Gamma$ of $\mathrm{Jac}(C)(\IQbar)$ of finite rank and consider only $P, Q \in \Gamma$, then $j(P), j(Q) \in \Gamma \otimes_{\Q}\R$ and $(\Gamma\otimes_{\Q}\R, |\cdot |)$ is a normed Euclidean space of finite dimension, and $\langle j(P), j(Q) \rangle / |j(P)| |j(Q)|$ is precisely the angle between $j(P)$ and $j(Q)$. Observe that it is possible to divide $\Gamma\otimes_{\Q}\R$ into $7^{\mathrm{rk}\Gamma}$ cones $\Lambda$ such that each two points in the same cone satisfies \eqref{EqMumfordVojtaIneqAngle}.

Now we turn to the conclusions. Part (i) says that each two distinct points in a same cone $\Lambda$ are ``far'' from each other, while part (ii) says that they cannot be ``too far'' either, unless at least one of these two points has small norm. Now if there is a sequence of distinct points $P_0,P_1,\ldots,P_m$ in $\Lambda$ such that $|j(P_m)| \ge \cdots \ge |j(P_1)| \ge |j(P_0)| \ge R$, then $|j(P_m)| \ge 2|j(P_{m-1})| \cdots \ge 2^m |j(P_0)|$ by (i) and $|j(P_m)| \le \kappa |j(P_0)|$ by (ii). Thus $m \le \log\kappa/\log 2$. As there are $7^{\mathrm{rk}\Gamma}$ cones, we obtain
\begin{equation}\label{EqLargePointOriginalBound}
\#\{P \in \Gamma  :  |j(P)| \ge R \} \le (\log\kappa/\log 2 + 1) 7^{\mathrm{rk}\Gamma}.
\end{equation}

\medskip

Notice that \eqref{EqLargePointOriginalBound} suffices to prove the Mordell conjecture. Assume $C$ is defined over a number field $K$. Take $P_0 \in C(K)$ and the Abel--Jacobi embedding $j \colon C \rightarrow \mathrm{Jac}(C)$ via $P_0$. By the Mordell--Weil theorem, $\Gamma:=\mathrm{Jac}(C)(K)$ is a finitely generated group. Thus the set $\Gamma_{\mathrm{tor}}$ of torsion points in $\Gamma$ is a finite set. So to prove the finiteness of $C(K) \cong \Gamma \cap j(C)(\IQbar)$ we may identify $\Gamma$ with its image in $\Gamma \otimes_{\Z} \R$. Consider the Euclidean space $(\Gamma \otimes_{\Z}\R, |\cdot|)$. %The image of $\Gamma \rightarrow \Gamma\otimes_{\Z}\R$ is $\Gamma/\Gamma_{\mathrm{tor}}$.
 By \eqref{EqLargePointOriginalBound}, to prove $\#C(K) < \infty$ it suffices to prove the finiteness of $C(K)_{\mathrm{small}} := \{P \in C(K) : |j(P)| < R\}$, or equivalently the finiteness of $j(C(K)_{\mathrm{small}})$. 
But after modulo the finite set $\Gamma_{\mathrm{tor}}$, $j(C(K)_{\mathrm{small}})$ is a subset of $\{z \in \Gamma : |z| < R\}$ which  
%Recall that $|j(P)| = 0$ if and only if $j(P)$ is a torsion point (see \eqref{}). So by the finiteness of rational torsion points, it suffices to prove the finiteness of $\{z \in \Gamma : |z| < R\}$. But this set 
consists of lattice points of bounded norm and hence is immediately a finite set. Hence we are done.

\subsection{Relative setting}
Mumford's and Vojta's inequality (Theorem~\ref{ThmMumfordVojtaIneq}) can be realized in families. The first explicitly written result in this direction is de Diego \cite[Thm.2 and below]{deDiego:97}. The version we state here can be obtained as a consequence of R\'{e}mond's quantitative versions of the Mumford and the Vojta inequalities, \cite[Thm.3.2]{Remond:Decompte} and \cite[Thm.1.2]{Remond:Vojtasup}.

All varieties and morphisms below are assumed to be defined over $\IQbar$.

Let $S$ be an irreducible variety and let 
$\pi \colon \cA\rightarrow S$ be an abelian scheme of relative dimension $g \ge 1$. 
Let $\cL$ be a relatively ample line bundle on $\cA/S$ such that $[-1]^*\cL \cong \cL$. 
 We write $\hat h_{\cL} \colon \cA(\IQbar)\rightarrow [0,\infty)$ for the fiberwise N\'eron--Tate height \eqref{EqFiberwiseNTHeight}. 

 Moreover, let $\cM$ be an ample line bundle over a compactification $\bar{S}$ of $S$. Then we obtain a function  $h_{\bar{S},\cM} \colon \bar{S}(\IQbar)\rightarrow \R$ which is a representative of the height provided by the Height Machine \eqref{EqHtMachine}.

If $\mathfrak{C}$ is an irreducible closed subvariety of $\cA$ and $s\in
S(\IQbar)$, then we write
$\mathfrak{C}_s$ for $\pi|_{\mathfrak{C}}^{-1}(s)$. 

\begin{thm}
    \label{ThmLargePointsFamily}
  Let $\mathfrak{C} \subset \cA$ be an irreducible closed subvariety that
  dominates $S$ and such that $\mathfrak{C}\rightarrow S$
  is a flat family of curves of genus at least $2$. 
  Then there exists a constant $c=c(\pi,\cL,\cM ; \mathfrak{C})\ge 1$ with the following property.
  Suppose $s\in S(\IQbar)$ and  
 $\Gamma$ is a subgroup of $\cA_s(\IQbar)$ of finite rank
  $\rho \ge 0$, then % If $C$ is an irreducible component of $\mathfrak{C}_s$
%  that is not a coset in $\cA_s$ (namely $\mathfrak{C}_s$ is not a translate of an abelian subvariety of $\cA_s$), then 
  \begin{equation}\label{EqLargePointBoundFamily}
   \# \left\{ P \in \mathfrak{C}_s(\IQbar)\cap \Gamma : \hat h_{\cL}(P) > c
   \max\{1,h_{\bar{S},\cM}(s)\} \right\} \le c^{\rho}. 
  \end{equation}
\end{thm}

It is possible to prove Theorem~\ref{ThmLargePointsFamily} by adapting appropriately the arguments in \cite{deDiego:97}. Alternatively, Theorem~\ref{ThmLargePointsFamily} can be proved more directly and with more explicit constants as a consequence of R\'{e}mond's quantitative versions of the Mumford and the Vojta inequalities, \cite[Thm.3.2]{Remond:Decompte} and \cite[Thm.1.2]{Remond:Vojtasup}, with the (Arithmetic) B\'{e}zout Theorem; see \cite[proof of Prop.8.1]{DGHUnifML} for more details.% The advantage is that the constants in \eqref{EqLargePointBoundFamily} are then explicit.

%% Section 4
\section{Basic setup and Statement of the New Gap Principle}\label{SectionNGP}
Fix an integer $g \ge 2$ and an integer $\ell \ge 3$. By \textit{level-$\ell$-structure} we mean symplectic level-$\ell$-structure.

\subsection{Universal families}\label{SubsectionUnivFamily}
It is natural to work with families to prove uniform bounds. In this subsection we introduce the various universal families which will be used.
\begin{enumerate}
\item[(i)] The universal curve $\mathfrak{C}_g \rightarrow \mathbb{M}_g$. Here $\mathbb{M}_g$ is the fine moduli space of smooth projective curves of genus $g$ with level-$\ell$-structure, and each fiber over $s \in \mathbb{M}_g(\C)$ is isomorphic to the curve parametrized by $s$. It is known that $\mathbb{M}_g$ is an irreducible regular quasi-projective variety of dimension $3g-3$. It is an irreducible variety defined over $\IQbar$. We refer to  \cite[(5.14)]{DM:irreducibility}, or \cite[Thm.1.8]{OortSteenbrink}.
\item[(ii)] The universal abelian variety $\pi \colon \mathfrak{A}_g \rightarrow \mathbb{A}_g$. Here $\mathbb{A}_g$ is the fine moduli space of principally polarized abelian varieties of dimension $g$ with level-$\ell$-structure, and each fiber over $s \in \mathbb{A}_g(\C)$ is isomorphic to the abelian variety parametrized by $s$. It is known that
$\mathbb{A}_g$ is an irreducible regular quasi-projective variety of dimension $g(g+1)/2$. It is an irreducible variety defined over $\IQbar$. We refer to \cite[Thm.7.9 and below]{MFK:GIT94} or \cite[Thm.1.9]{OortSteenbrink}.
\end{enumerate}

%Let $\mathbb{M}_g$ be the fine moduli space of smooth projective curves of genus $g$  with level-$\ell$-structure, \textit{cf.}  \cite[Chapter~XVI, Thm.2.11 (or above Prop.2.8)]{ACG:Curve},  \cite[(5.14)]{DM:irreducibility}, or \cite[Thm.1.8]{OortSteenbrink}.  It is known that $\mathbb{M}_g$ is an irreducible regular quasi-projective variety defined over $\IQbar$, and $\dim \mathbb{M}_g = 3g-3$.  There exists a universal curve $\mathfrak{C}_g$ over $\mathbb{M}_g$; it is smooth and proper over $\mathbb{M}_g$ with fibers that are smooth curves of genus $g$.% Moreover, it is equipped with level $\ell$-structure. 

The two universal families can be related in the following way.  Let $\mathrm{Jac}(\mathfrak{C}_g/\mathbb{M}_g)$ be the relative Jacobian of $\mathfrak{C}_g \rightarrow \mathbb{M}_g$. It is an abelian scheme
equipped with a natural principal polarization and with level-$\ell$-structure; see
\cite[Prop.6.9]{MFK:GIT94}. Attaching the Jacobian to a smooth curve induces the Torelli morphism
$\tau\colon \mathbb{M}_g \rightarrow \mathbb{A}_g$. The famous Torelli
theorem states that, absent level structure, the Torelli morphism is
injective on $\C$-points.
In our setting,  $\tau$ is a quasi-finite morphism \textit{cf.} \cite[Lem.1.11]{OortSteenbrink}. 
As $\mathbb{A}_g$ is a fine moduli space we have the following Cartesian diagram
\begin{equation}\label{EqUnivJac}
\xymatrix{
\mathrm{Jac}(\mathfrak{C}_g/\mathbb{M}_g) \ar[r] \ar[d]  \pullbackcorner & \mathfrak{A}_g \ar[d]^{\pi} \\
\mathbb{M}_g \ar[r]_{\tau} & \mathbb{A}_g
}
\end{equation}
%We shall use the following notation. Let $S \rightarrow \mathbb{M}_g$ be a morphism. We will denote the base changes by $\mathfrak{C}_S := \mathfrak{C}_g \times_{\mathbb{M}_g} S$ and $\mathfrak{A}_S := \mathfrak{A}_g \times_{\mathbb{A}_g} S$ (for $S \rightarrow \mathbb{M}_g \xrightarrow{\tau}\mathbb{A}_g$).

\subsection{The Faltings--Zhang map}\label{SubsectionFaltingsZhang}
The \textit{New Gap Principle} from $\mathsection$\ref{SubsectionNGPIntro} concerns the differences of the points on each curve $C$ taken in its Jacobian $\mathrm{Jac}(C)$. This operation can be made precise by setting the subvariety $C-C$ of $\mathrm{Jac}(C)$ to be the image of $C \times C \rightarrow \mathrm{Jac}(C) = \mathrm{Pic}^0(C), (P,Q) \mapsto [Q-P]$. By abuse of notation we denote by $(P,Q) \mapsto Q-P$.\footnote{Notice that for any Abel--Jacobi embedding $j \colon C \rightarrow \mathrm{Jac}(C)$, we have $C-C = j(C) - j(C)$, where the difference on the right hand side is taken as the group operation on the abelian variety. This is because doing the difference cancels out the base point of the Abel--Jacobi embedding.}

We need to realize this difference in families. Let $\mathrm{Pic}(\mathfrak{C}_g/\mathbb{M}_g)$ be the relative Picard scheme; it is a group scheme over $\mathbb{M}_g$ and can be decomposed as the union of open and closed subschemes $\mathrm{Pic}^p(\mathfrak{C}_g/\mathbb{M}_g)$ for all $p \in \Z$, where $p$ indicates the degree of a line bundle. The difference group law $\mathrm{Pic}(\mathfrak{C}_g/\mathbb{M}_g) \times_{\mathbb{M}_g} \mathrm{Pic}(\mathfrak{C}_g/\mathbb{M}_g) \rightarrow \mathrm{Pic}(\mathfrak{C}_g/\mathbb{M}_g)$, when restricted to $\mathrm{Pic}^1(\mathfrak{C}_g/\mathbb{M}_g) \times_{\mathbb{M}_g} \mathrm{Pic}^1(\mathfrak{C}_g/\mathbb{M}_g)$, induces an $\mathbb{M}_g$-morphism
\[
\mathrm{Pic}^1(\mathfrak{C}_g/\mathbb{M}_g) \times_{\mathbb{M}_g} \mathrm{Pic}^1(\mathfrak{C}_g/\mathbb{M}_g) \rightarrow \mathrm{Pic}^0(\mathfrak{C}_g/\mathbb{M}_g) = \mathrm{Jac}(\mathfrak{C}_g/\mathbb{M}_g).
\]
 From \cite[proof of Prop.6.9]{MFK:GIT94} we get a $\mathbb{M}_g$-morphism $\mathfrak{C}_g \rightarrow \mathrm{Pic}^1(\mathfrak{C}_g/\mathbb{M}_g)$. Thus the $\mathbb{M}_g$-morphism above induces an $\mathbb{M}_g$-morphism
 \begin{equation}\label{EqFaltingZhang1}
 \mathscr{D}_1 \colon \mathfrak{C}_g \times_{\mathbb{M}_g} \mathfrak{C}_g \rightarrow \mathrm{Jac}(\mathfrak{C}_g/\mathbb{M}_g).
 \end{equation}
The restriction of $\mathscr{D}_1$ to each fiber is precisely $(P,Q) \mapsto Q-P$. We thus denote by $\mathfrak{C}_g - \mathfrak{C}_g$ the image of $\mathscr{D}_1$.%$\mathscr{D}_1( \mathfrak{C}_g \times_{\mathbb{M}_g} \mathfrak{C}_g )$ in $\mathfrak{A}_g$ under the top morphism in \eqref{EqUnivJac}.

This construction can be generalized to more factors. Let $M \ge 1$ be an integer. Let $\mathfrak{C}_g^{[M]}$ and $\mathrm{Jac}(\mathfrak{C}_g/\mathbb{M}_g)^{[M]}$ denote the respective $M$-th fibered powers over $\mathbb{M}_g$. %, and let $\mathfrak{A}_g^{[M]}$ be the $M$-th fibered power over $\mathbb{A}_g$. 
 Then we get an $\mathbb{M}_g$-morphism
 \begin{equation}\label{EqFaltingZhang}
 \mathscr{D}_M \colon \mathfrak{C}_g^{[M+1]} \rightarrow \mathrm{Jac}(\mathfrak{C}_g/\mathbb{M}_g)^{[M]} ,
 \end{equation}
such that over each fiber it is $(P_0,P_1,\ldots,P_M) \mapsto (P_1-P_0, \ldots, P_M-P_0)$.

%Finally, for any morphism $S \rightarrow \mathbb{M}_g$, by abuse of notation we will denote by
 %\begin{equation}\label{EqFaltingZhangFinal}
% \mathscr{D}_M \colon \mathfrak{C}_S^{[M+1]}:=\mathfrak{C}_g^{[M+1]}\times_{\mathbb{M}_g}S \rightarrow \mathfrak{A}_g^{[M]}
% \end{equation}
%the composite $\mathfrak{C}_S^{[M+1]} \rightarrow \mathfrak{C}_g^{[M+1]} \xrightarrow{\eqref{EqFaltingZhang}}   \mathrm{Jac}(\mathfrak{C}_g/\mathbb{M}_g)^{[M]} \rightarrow \mathfrak{A}_g^{[M]}$.

\subsection{Height functions}\label{SubsectionHtUnivFamily}
To give a precise statement of the New Gap Principle, we need to fix the height functions. All line bundles below are assumed to be defined over $\IQbar$.

Let $\mathfrak{L}$ be a line bundle on $\mathrm{Jac}(\mathfrak{C}_g/\mathbb{M}_g)$ ample over $\mathbb{M}_g$ such that $[-1]^*\mathfrak{L} \cong \mathfrak{L}$; see \cite[Thm.XI~1.4]{LNM119}. This defines a fiberwise N\'{e}ron--Tate height \eqref{EqFiberwiseNTHeight}
\begin{equation}
\hat{h}_{\mathfrak{L}} \colon \mathrm{Jac}(\mathfrak{C}_g/\mathbb{M}_g)(\IQbar) \rightarrow [0,\infty).
\end{equation}
We also fix an ample line bundle $\mathfrak{M}$ on $\overline{\mathbb{M}_g}$, where $\overline{\mathbb{M}_g}$ is a compactification of $\mathbb{M}_g$. The Height Machine \eqref{EqHtMachine} provides an equivalence class of height function of which we fix a representative
\begin{equation}
h_{\overline{\mathbb{M}_g},\mathfrak{M}} \colon \overline{\mathbb{M}_g}(\IQbar) \rightarrow \R.
\end{equation}

\subsection{The New Gap Principle}
We are now ready to give the precise statement of the new Gap Principle. By definition of the moduli space and the universal curve, each smooth curve $C$ of genus $g \ge 2$ defined over $\IQbar$ is isomorphic to $\mathfrak{C}_s$, the fiber of $\mathfrak{C}_g \rightarrow \mathbb{M}_g$ over $s$, for some $s \in \mathbb{M}_g(\IQbar)$. Use the height functions from $\mathsection$\ref{SubsectionHtUnivFamily}.
\begin{thm}[Dimitrov--Gao--Habegger + K\"{u}hne]\label{ThmNGP}
There exist positive constants $c_1, c_2$ depending only on $g$ (apart from $\mathfrak{L}$ and $\mathfrak{M}$) with the following property. For each $s \in \mathbb{M}_g(\IQbar)$ and each $P \in \mathfrak{C}_s(\IQbar)$, %there exists a subset $\Xi_s \subset \mathfrak{C}_s(\IQbar)$ with $\#\Xi_s \le c_2$ such that each $P \in \mathfrak{C}_s(\IQbar)$ satisfies the following alternative:
we have 
%\begin{enumerate}
%\item[(i)] Either $P \in \Xi_s$;
%\item[(ii)] or $
\begin{equation}\label{EqNGPThm}
\#\left\{Q \in \mathfrak{C}_s(\IQbar) : \hat{h}_{\mathfrak{L}}(Q-P) \le c_1\max\{1,h_{\overline{\mathbb{M}_g},\mathfrak{M}}(s)\} \right\} < c_2.
\end{equation}
%\end{enumerate}
\end{thm}
In the statement \eqref{EqNGPThm}, the height $h_{\overline{\mathbb{M}_g},\mathfrak{M}}(s)$ can be replaced by  any modular height of $[\mathrm{Jac}(C)] \in \mathbb{A}_{g,1}(\IQbar)$; see \cite[proof~of~Thm.1.2 and above]{DGHUnifML}. Here $\mathbb{A}_{g,1}$ is the coarse moduli space of principally polarized abelian varieties of dimension $g$. The key point is that the Torelli map $\tau \colon \mathbb{M}_g \rightarrow \mathbb{A}_g$ is quasi-finite and the triangular inequality Theorem~\ref{ThmHtTriangularIneq}.(i). By a fundamental work of Faltings \cite[$\mathsection$3 including the proof of Lemma~3]{faltings1983endlichkeitssatze} to compare the modular height with the \textit{Faltings height} of any given abelian variety, 
$h_{\overline{\mathbb{M}_g},\mathfrak{M}}(s)$ can furthermore be replaced by 
the Faltings height $h_{\mathrm{Fal}}(\mathrm{Jac}(C))$.

The proof of Theorem~\ref{ThmNGP} is a combination of \cite[Prop.7.1 (and its proof)]{DGHUnifML} and \cite[Thm.3]{KuehneUnifMM}. Roughly speaking, the former result handles curves of large height, and the latter result handles curves of small height. 
 More precisely, an adjustment of the proof of \cite[Prop.7.1]{DGHUnifML} proves \eqref{EqNGPThm} with $c_1\max\{1,h_{\overline{\mathbb{M}_g},\mathfrak{M}}(s)\}$ replaced by $c_1\max\{1,h_{\overline{\mathbb{M}_g},\mathfrak{M}}(s)\} - c_3$, and hence what remains to be done is to remove the constant term $c_3$. Then \cite[Thm.3]{KuehneUnifMM} proves \eqref{EqNGPThm} with $c_1\max\{1,h_{\overline{\mathbb{M}_g},\mathfrak{M}}(s)\}$ replaced by some $c_3' > 0$, which is exactly what is needed to remove the $c_3$.

%By \cite[Prop.4.4.10(ii) and Prop.4.1.4]{EGAII}, we then have a closed immersion $\mathfrak{A}_g \rightarrow \P^n_{\IQbar}\times \mathbb{A}_g$ over $\mathbb{A}_g$ arising from $\mathfrak{L}\otimes\pi^*\mathfrak{M}^{\otimes p}$, where $\mathfrak{M}$ is an ample line bundle on $\mathbb{A}_g$, for some integer $p \ge 1$.

\small

\subsection{Polarization type}
Let $d_1|\cdots |d_g$ be positive integers, and set $D: = \mathrm{diag}(d_1,\ldots,d_g)$.

In this subsection, we introduce a new moduli space and the universal family, \textit{cf.} \cite[$\mathsection$1.2 and 1.3]{GenestierNgo}.

Let $\A_{g,\ell,D}$ be the moduli space of abelian varieties polarized of type $D$ (so of dimension $g$) with level-$\ell$-structure. If $\ell \ge 3$, then $\A_{g,\ell,D}$ is a fine moduli space, and hence admits a universal family $\mathfrak{A}_{g,\ell,D} \rightarrow \A_{g,\ell,D}$.

The universal covering in the category of complex spaces for $\A_{g,\ell,D}$ is given by $\mathfrak{H}_g \rightarrow \A_{g,\ell,D}^{\mathrm{an}}$, where $\mathfrak{H}_g = \{Z \in \mathrm{Mat}_{g\times g}(\C): Z = Z^{\!^{\intercal}} , ~ \mathrm{Im}(Z) > 0\}$ is the Siegel upper half space. Let $\mathrm{Sp}_{2g,D}$ be the $\Q$-group defined by
\begin{equation}\label{EqGroupSpD}
\mathrm{Sp}_{2g,D}(\Q) = \left\{g \in \mathrm{SL}_{2g}(\Q) : g \begin{bmatrix}	 0 & D \\ -D & 0	\end{bmatrix} g^{\!^{\intercal}} = \begin{bmatrix}	 0 & D \\ -D & 0	\end{bmatrix}\right\}.
\end{equation}
Then $\mathrm{Sp}_{2g,D}(\R)$ acts transitively on $\mathfrak{H}_g$ as described in \cite[$\mathsection$1.2]{GenestierNgo}, and the uniformization above induces $\A_{g,\ell,D}^{\mathrm{an}} \cong \mathrm{Sp}_{2g,D}(1+\ell\Z) \backslash \mathfrak{H}_g$ with $\mathrm{Sp}_{2g,D}(1+\ell\Z) = \ker(\mathrm{Sp}_{2g,D}(\Z) \rightarrow \mathrm{Sp}_{2g,D}(\Z/\ell\Z))$.

In the context, we often abbreviate $\mathbb{A}_{g,\ell,D}$ by $\mathbb{A}_{g,D}$, and $\mathfrak{A}_{g,\ell,D}$ by $\mathfrak{A}_{g,D}$.

Now let $S$ be an irreducible variety over $\C$ and $\pi \colon \cA \rightarrow S$ be an abelian scheme of relative dimension $g \ge 1$. By \cite[$\mathsection$2.1]{GenestierNgo}, $\cA \rightarrow S$ is polarizable of type $D$ for some diagonal matrix $D$ as above. Then up to taking a finite cover of $S$ and taking the appropriate base change of $\cA \rightarrow S$, there exists a Cartesian diagram
\begin{equation}\label{EqModularMap}
\xymatrix{
\cA \ar[r]^-{\iota} \ar[d]_{\pi} \pullbackcorner & \mathfrak{A}_{g,D} \ar[d] \\
S \ar[r]^-{\iota_S} & \mathbb{A}_{g,D}.}
\end{equation}
The morphism $\iota$ is called the \textit{modular map}.

\normalsize

%% Section 5
\section{Betti map and Betti form}\label{SectionBetti}
This section introduces two fundamental tools in the course of proving Theorem~\ref{ConjMazur}, the \textit{Betti map} and the \textit{Betti form}.

In this section, let $S$ be an irreducible variety over $\C$ and $\pi \colon \cA \rightarrow S$ be an abelian scheme of relative dimension $g \ge 1$. By \cite[$\mathsection$2.1]{GenestierNgo}, there exist positive integers $d_1|\cdots |d_g$ such that $\cA \rightarrow S$ is polarizable of type $D: = \mathrm{diag}(d_1,\ldots,d_g)$.

\subsection{Betti map}
The Betti map is a useful tool in Diophantine Geometry. It was already used in early works of Corvaja, Masser and Zannier on the Relative Manin--Mumford Conjecture; see $\mathsection$\ref{SubsectionRelBog}. The name ``Betti map'' was proposed by Bertrand.

\medskip

The idea to define the Betti map is simple: one identifies each closed fiber $\cA_s$ with the real torus $\mathbb{T}^{2g}$  under the period matrices. Here is a brief construction. For any $s \in S(\C)$, there exists an open neighborhood $\Delta
\subseteq S^{\mathrm{an}}$ of $s$ which we may assume is simply-connected.  Then one can define the \textit{Betti map} 
\begin{equation}\label{EqBettiMap}
b_{\Delta} \colon \cA_{\Delta} = \pi^{-1}(\Delta) \rightarrow \mathbb{T}^{2g},
\end{equation}
as follows.  As $\Delta$ is simply-connected, one defines a basis $\omega_1(s),\ldots,\omega_{2g}(s)$ of the period lattice of each fiber $s \in \Delta$ as holomorphic functions of $s$. Now each fiber $\cA_s = \pi_S^{-1}(s)$ can be identified with the complex torus $\C^g/\Z \omega_1(s)\oplus \cdots \oplus \Z\omega_{2g}(s)$, and each point $x \in \cA_s(\C)$ can be expressed as the class of $\sum_{i=1}^{2g}b_i(x) \omega_i(s)$ for real numbers $b_1(x),\ldots,b_{2g}(x)$. Then $b_{\Delta}(x)$ is defined to be the class of the $2g$-tuple $(b_1(x),\ldots,b_{2g}(x)) \in \R^{2g}$ modulo $\Z^{2g}$. We thus obtain \eqref{EqBettiMap}. The map $b_{\Delta}$ is not unique, but it is unique up to $\mathrm{GL}_{2g}(\Z) \cong \mathrm{Aut}(\mathbb{T}^{2g})$. In fact, later on we will see that $b_{\Delta}$ is in fact unique up to $\mathrm{Sp}_{2g,D}(\Z)$ with the group $\mathrm{Sp}_{2g,D}$ defined in \eqref{EqGroupSpD} if the basis is well-chosen.

\medskip

The following \textit{Betti rank} is of particular importance; see \cite{ACZBetti}.%, as suggested by \cite{ACZBetti, GaoHab, CGHX:18}.% The question of dete explicitly asked as a question in  \cite{}.
\begin{defn}
Let $X$ be an irreducible subvariety of $\cA$ and let $x \in X^{\mathrm{sm}}(\C)$. The Betti rank of $X$ at $x$ is defined to be
\begin{equation}\label{EqBettiRank}
\mathrm{rank}_{\mathrm{Betti}}(X,x) := \mathrm{rank}_{\R}(\mathrm{d}b_{\Delta}|_{X^{\mathrm{sm,an}}})_x
\end{equation}
where $\Delta$ is an open neighborhood of $\pi(x)$ in $S^{\mathrm{an}}$ and $b_{\Delta}$ is the Betti map.
\end{defn}
%Because $b_{\Delta}$ is unique up to $\mathrm{GL}_{2g}(\Z)$, t
The right hand side of \eqref{EqBettiRank} does not depend on the choice of $\Delta$ or $b_{\Delta}$.

\medskip

More concrete constructions of the Betti map can be found in \cite{ACZBetti} via $1$-motives, in \cite{CGHX:18} by means of Arithmetic Dynamics, and in \cite{GaoBettiRank} using the universal abelian varieties. An ad hoc construction when $\dim S = 1$ can be found in \cite{GaoHab}. In the course of the constructions, the following proposition can be proved.
\begin{prop}\label{PropBettiMap}
The Betti map $b_{\Delta}$ satisfies the following properties.
\begin{enumerate}
\item[(i)] For each $t \in \mathbb{T}^{2g}$, we have that $b_{\Delta}^{-1}(t)$ is complex analytic.
\item[(ii)] For each $s \in \Delta$, the restriction $b_{\Delta}|_{\cA_s}$ is a group isomorphism.
\item[(iii)] The map $(b_{\Delta},\pi) \colon \cA_{\Delta} \rightarrow \mathbb{T}^{2g} \times \Delta$ is a real analytic isomorphism.
\end{enumerate}
\end{prop}

We hereby take the construction from \cite[$\mathsection$3 and 4]{GaoBettiRank}, and briefly sketch for the case $\mathfrak{A}_{g,D} \rightarrow \mathbb{A}_{g,D}$. The general case follows easily from it by composing with the modular map $\iota$ from \eqref{EqModularMap}.

The universal covering $\mathfrak{H}_g \rightarrow \mathbb{A}_{g,D}^{\mathrm{an}}$,
where $\mathfrak{H}_g = \{Z \in \mathrm{Mat}_{g\times g}(\C): Z = Z^{\!^{\intercal}} , ~ \mathrm{Im}(Z) > 0\}$ is the Siegel upper half space, gives a
polarized family of abelian varieties
$\cA_{\mathfrak{H}_g} \rightarrow \mathfrak{H}_g$ fitting into the diagram
\[
\xymatrix{
\cA_{\mathfrak{H}_g} := \mathfrak{A}_{g,D}^{\mathrm{an}} \times_{\mathbb{A}_{g,D}^{\mathrm{an}}}\mathfrak{H}_g \ar[r]^-{u_B} \ar[d] & \mathfrak{A}_{g,D}^{\mathrm{an}} \ar[d]^{\pi^{\mathrm{univ}}} \\
\mathfrak{H}_g \ar[r] & \mathbb{A}_{g,D}^{\mathrm{an}}.
}
\]
For the universal covering $u \colon \C^g \times \mathfrak{H}_g
\rightarrow \cA_{\mathfrak{H}_g}$ and for each $Z \in
\mathfrak{H}_g$, the kernel of $u|_{\mathbb{C}^g \times \{Z\}}$ is
$D\mathbb{Z}^g + Z \mathbb{Z}^g$. Thus the map $\C^g \times
\mathfrak{H}_g \rightarrow \R^g \times \R^g \times
\mathfrak{H}_g \rightarrow \R^{2g}$, where the first map is the
inverse of $(a,b,Z) \mapsto (Da + Z  b,Z)$ and the second map is the
natural projection, descends to a \textit{real} analytic map
\[
b^{\mathrm{univ}} \colon \cA_{\mathfrak{H}_g} \rightarrow \mathbb T^{2g}.
\]
 Now for each $s_0 \in \mathbb A_{g,D}(\C)$, there exists a contractible,
 relatively compact, open neighborhood $\Delta$ of $s_0$ in $\mathbb A_{g,D}^{\mathrm{an}}$ such that $\mathfrak{A}_{g,D,\Delta}:=(\pi^{\mathrm{univ}})^{-1}(\Delta)$ can be identified with $\cA_{\mathfrak{H}_g,\Delta'}$ for some open subset $\Delta'$ of $\mathfrak{H}_g$. The composite $b_{\Delta} \colon \mathfrak{A}_{g,D,\Delta} \cong \cA_{\mathfrak{H}_g,\Delta'} \rightarrow \mathbb T^{2g}$ is  real analytic and satisfies the three properties for the Betti map. Thus $b_{\Delta}$ is the desired Betti map in this case. Note that for a fixed (small enough) $\Delta$, there are infinitely choices of $\Delta'$; but for $\Delta$ small enough, if $\Delta_1'$ and $\Delta_2'$ are two such choices, then $\Delta_2' = \alpha \cdot \Delta_1'$ for some $\alpha \in \mathrm{Sp}_{2g,D}(\Z) \subset \mathrm{SL}_{2g}(\Z)$.% Thus we have proved Proposition~\ref{PropBettiMap} for $\mathfrak{A}_g \rightarrow \mathbb A_g$.

The last sentence of the previous paragraph implies the following property. Fix a Betti map $b_{\Delta} \colon \mathfrak{A}_{g,D,\Delta} \rightarrow \mathbb{T}^{2g}$, then any other Betti map $\mathfrak{A}_{g,D,\Delta} \rightarrow \mathbb{T}^{2g}$ is $\alpha \cdot b_{\Delta}$ for some $\alpha \in  \mathrm{Sp}_{2g,D}(\Z)$.
 
 \subsection{Betti form}\label{SubsectionBettiForm}
 The \textit{Betti form} is a closed semi-positive smooth $(1,1)$-form $\omega$ on $\cA^{\mathrm{an}}$ with the property $[N]^*\omega = N^2 \omega$ such that the following property holds true: For any subvariety $X$ of $\cA$ and any $x \in X^{\mathrm{sm}}(\C)$, we have
 \begin{equation}\label{EqBettiRankBettiForm}
 \mathrm{rank}_{\mathrm{Betti}}(X,x) = 2\dim X \Leftrightarrow (\omega|_X^{\wedge \dim X})_x \not= 0.
 \end{equation}
% If the geometric generic fiber of $\cA/S$ is a simple abelian variety, then $\omega$ is unique up to constant.

There are several ways to construct the Betti form $\omega$. In \cite[$\mathsection$2.2 and 2.3]{DGHUnifML} by using a formula given by Mok \cite[pp.374]{Mok11Form}, and in \cite[$\mathsection$2]{CGHX:18} by means of Arithmetic Dynamics.
 
 We hereby state a third construction via the Betti map. It is closely related to the construction in \cite{DGHUnifML}.
 \begin{constr}\label{ConstrBettiForm}
 Use $(a,b) = (a_1,b_1; \ldots ; a_g, b_g)$ to denote the coordinates of $\mathbb{T}^{2g}$. Let $\Delta$ be a simply-connected open subset of $S^{\mathrm{an}}$ and $b_{\Delta} \colon \cA_{\Delta} \rightarrow \mathbb{T}^{2g}$ be the Betti map from \eqref{EqBettiMap}. Define the $2$-form on $\cA_{\Delta}$
 \begin{equation}\label{EqBettiMapBettiFormFormula}
\omega_{\Delta} = b_{\Delta}^{-1}\left( 2(D\mathrm{d}a)^{\!^{\intercal}}  \wedge \mathrm{d}b \right) = b_{\Delta}^{-1}\left(2\sum_{j=1}^g d_j \mathrm{d}a_j \wedge \mathrm{d}b_j \right).
\end{equation}
Observe that $\omega_{\Delta}$ is well-defined, because two different choices of $b_{\Delta}$ differ from an element in $\mathrm{Sp}_{2g,D}(\Z)$ (see above $\mathsection$\ref{SubsectionBettiForm}) and $2(D\mathrm{d}a)^{\!^{\intercal}}  \wedge \mathrm{d}b$ is preserved by $\mathrm{Sp}_{2g,D}(\R)$.

Moreover, it is not hard to check that these $\omega_{\Delta}$ glue together to a $2$-form $\omega$ on $\cA^{\mathrm{an}}$. 

This $\omega$ is the desired Betti form.
 \end{constr}
 
For the $\omega$ constructed above, the facts that $\omega$ is smooth and  $[N]^*\omega = N^2 \omega$ are not hard to check. To check that $\omega$ is a $(1,1)$-form and is semi-positive, one can do an explicit computation by the change of coordinates $(b_{\Delta},\pi) \colon \cA_{\Delta} \rightarrow \mathbb{T}^{2g} \times \Delta$ from Proposition~\ref{PropBettiMap}.(iii). In fact, by a similar computation executed in \cite[$\mathsection$2.2]{DGHUnifML}, one can prove the following statement. For the uniformization $\mathbf{u} \colon \C^g \times \mathfrak{H}_g  \rightarrow \mathfrak{A}_{g,D}^{\mathrm{an}}$ and the Betti form $\omega$ on $\mathfrak{A}_{g,D}^{\mathrm{an}}$, we have
\begin{equation}
\mathbf{u}^*\omega = \sqrt{-1} \partial\bar{\partial}\left( 2 (\mathrm{Im}w)^{\!^{\intercal}} (\mathrm{Im}Z)^{-1} (\mathrm{Im}w) \right)
\end{equation}
where we use $(w,Z)$ to denote the coordinates on $\C^g \times \mathfrak{H}_g$. The symmetric real matrix representing $\mathbf{u}^*\omega$ is
\begin{equation}
\begin{bmatrix}
1 &  - (\mathrm{Im}w)^{\!^{\intercal}} (\mathrm{Im}Z)^{-1} \\
- (\mathrm{Im}Z)^{-1} (\mathrm{Im}w) &  (\mathrm{Im}Z)^{-1} (\mathrm{Im}w) (\mathrm{Im}w)^{\!^{\intercal}} (\mathrm{Im}Z)^{-1}
\end{bmatrix} \otimes (\mathrm{Im}Z)^{-1}.
\end{equation}
Now \eqref{EqBettiRankBettiForm} a consequence of \eqref{EqBettiMapBettiFormFormula}; see \cite[Lem.10]{KuehneUnifMM}.

\begin{rmk}\label{RmkBettiFormFiberProduct}
For each integer $M \ge 1$, set $\cA^{[M]} = \cA\times_S \ldots \times_S \cA$ ($M$-copies). Then  $p_1^*\omega + \cdots + p_M^*\omega$ is a choice of the Betti form on $(\cA^{[M]})^{\mathrm{an}}$, with each $p_i\colon \cA^{[M]} \rightarrow \cA$ the projection to the $i$-th factor.
\end{rmk}

We close this section by pointing out a more geometric property of the Betti form, which is a geometric motivation behind \cite{DGHUnifML} and \cite{KuehneUnifMM}.% I thank Ngaiming Mok for pointing it out to me.
 
 Assume $\ell \ge 3$ is even. There exists a tautological relatively ample line bundle $\mathfrak{L}_{g,D}$ on $\mathfrak{A}_{g,D} / \mathbb{A}_{g,D}$, namely for each $s \in \mathbb{A}_{g,D}(\C)$, $((\mathfrak{A}_{g,D})_s, (\mathfrak{L}_{g,D})_s)$ is the polarized abelian variety parametrized by $s$. Moreover $[-1]^*\mathfrak{L}_{g,D} = \mathfrak{L}_{g,D}$. We refer to \cite[Prop.10.8 and 10.9]{PinkThesis}.

\begin{prop}\label{PropBettiFormChernClass}
The cohomology class of the Betti form $\omega$ on $\mathfrak{A}_{g,D}^{\mathrm{an}}$ coincides with the first Chern class $c_1(\mathfrak{L}_{g,D})$ of $\mathfrak{L}_{g,D}$.
\end{prop}
This proposition can be deduced  from \cite[pp.374]{Mok11Form} or \cite[Lem.2.4]{CGHX:18}.

%% Section 6
\section{Non-degenerate subvarieties}\label{SectionNonDeg}
This section is based on \cite{GaoBettiRank}. In this section, let $S$ be an irreducible variety over $\C$ and $\pi \colon \cA \rightarrow S$ be an abelian scheme of relative dimension $g \ge 1$. Let $\omega$ be the Betti form on $\cA$ from Construction~\ref{ConstrBettiForm}.

\begin{defn}\label{DefnNonDeg}
An irreducible subvariety $X$ of $\cA$ is said to be non-degenerate if one of the following equivalent conditions holds true:
\begin{enumerate}
\item[(i)] $\mathrm{rank}_{\mathrm{Betti}}(X,x) = 2 \dim X$ for some $x \in X^{\mathrm{sm}}(\C)$;
\item[(ii)] $(\omega|_X^{\wedge \dim X})_x \not= 0$ for some $x \in X^{\mathrm{sm}}(\C)$.
\end{enumerate}
\end{defn}
The conditions (i) and (ii) are equivalent by \eqref{EqBettiRankBettiForm}. By (ii) and Proposition~\ref{PropBettiFormChernClass}, non-degeneracy should be understood to be some \textit{bigness} condition of an appropriate line bundle.\footnote{In the particular case where $X$ is a \textit{projective} subvariety of $\mathfrak{A}_{g,D}$, $X$ is non-degenerate if and only if $\mathfrak{L}_{g,D}|_X$ is a big line bundle.}

\subsection{A first discussion}
In this section, we abbreviate $\mathfrak{A}_{g,D} \rightarrow \mathbb{A}_{g,D}$ by $\mathfrak{A}_g \rightarrow \mathbb{A}_g$, with the polarization type $D$ clear according to the context.

Consider the Cartesian diagram from \eqref{EqModularMap}, with the modular map $\iota$,
\begin{equation}\label{EqModularMap2}
\xymatrix{
\cA \ar[r]^-{\iota} \ar[d]_{\pi} \pullbackcorner & \mathfrak{A}_g \ar[d] \\
S \ar[r]^-{\iota_S} & \mathbb{A}_g.}
\end{equation}

The Betti map $b_{\Delta}$ from \eqref{EqBettiMap} factors through $\iota$. Thus $\mathrm{rank}_{\mathrm{Betti}}(X,x) \le 2\dim \iota(X)$ trivially holds true. So from (i) of Definition~\ref{DefnNonDeg}, $\iota|_X$ must be generically finite if $X$ is non-degenerate. On the other hand, the target of $b_{\Delta}$ is $\mathbb{T}^{2g}$. So $\mathrm{rank}_{\mathrm{Betti}}(X,x) \le 2g$ trivially holds true. So from (i) of Definition~\ref{DefnNonDeg}, $\dim X \le g$ if $X$ is non-degenerate.
To sum it up, the trivial bounds yield
\begin{equation}\label{EqNonDegNaive}
X \text{ non-degenerate} \Rightarrow \iota|_X\text{ is generically finite and }\dim X \le g.
\end{equation}
Thus, $\mathfrak{C}_g  - \mathfrak{C}_g$ defined below \eqref{EqFaltingZhang1} is a degenerate subvariety of $\mathrm{Jac}(\mathfrak{C}_g/\mathbb{M}_g)$, because its dimension is greater than $g$.

The converse of \eqref{EqNonDegNaive} is in general false; see \cite[Thm.1.4(ii)]{GaoBettiRank} for an example. But \textit{the converse of \eqref{EqNonDegNaive} is true if the geometric generic fiber of $\cA \rightarrow S$ is a simple abelian variety}; see \cite[Thm.1.4(i)(a)]{GaoBettiRank}.

Another useful observation is the following lemma.
\begin{lemma}\label{LemmaNonDegFiberProd}
Let $X$ and $Y$ be irreducible subvarieties of $\cA$ such that $\pi|_X$ and $\pi|_Y$ are both dominant.  
Assume that $X$ is non-degenerate. Then $X \times_S Y$ is a non-degenerate subvariety of $\cA\times_S \cA$.
\end{lemma}
\begin{proof}
By generic smoothness, we may assume that $S$ is smooth, and both $X^{\mathrm{sm}} \rightarrow S$ and $Y^{\mathrm{sm}}\rightarrow S$ are smooth morphisms.

We have $\dim X\times_S Y = \dim X + \dim Y - \dim S$. Since $X$ is non-degenerate, there exists $x \in X^{\mathrm{sm}}(\C)$ such that $\mathrm{rank}_{\mathrm{Betti}}(X,x) = 2 \dim X$. Let $s = \pi(x) \in S(\C)$.

As the Betti map is a group isomorphism when restricted to $\cA_s = \pi^{-1}(s)$ (Proposition~\ref{PropBettiMap}.(ii)), we have that $\mathrm{rank}_{\mathrm{Betti}}(Y,y) \ge  2 \dim Y_s = 2(\dim Y - \dim S)$ for a generic $y \in Y_s^{\mathrm{sm}}(\C)$. Thus $y \in Y^{\mathrm{sm}}(\C)$ as $S$ is smooth and $Y^{\mathrm{sm}}\rightarrow S$ is a smooth morphism.

Now $(x,y) \in (X\times_S Y)^{\mathrm{sm}}(\C)$ and $\mathrm{rank}_{\mathrm{Betti}}(X\times_S Y, (x,y)) = 2(\dim X + \dim Y - \dim S)$. Hence we are done.
\end{proof}

The proof of the lemma above also has the following consequence. Let $M \ge 1$ be an integer. For notation, let $X^{[M]} = X\times_S \ldots\times_S X$ ($M$-copies) for any subvariety $X$ of $\cA$, and $\omega_M$ be the Betti form on $\cA^{[M]}:=\cA\times_S\cdots\times_S\cA$ ($M$-copies).
\begin{lemma}\label{LemmaNonDegPointProduct}
Assume $\pi|_{X^{\mathrm{sm}}}$ is smooth and $x \in X^{\mathrm{sm}}(\C)$ satisfies $(\omega|_X^{\wedge \dim X})_x \not= 0$. 
Then we have $(\omega_M|_{X^{[M]}}^{\wedge \dim X^{[M]}})_{(x,\ldots,x)} \not = 0$. 
\end{lemma}
\begin{proof}
We have $\mathrm{rank}_{\mathrm{Betti}}(X,x) = 2\dim X$ by \eqref{EqBettiRankBettiForm}. Let $s = \pi(x)$. By assumption, $(x,\ldots,x) \in (X^{[m]})^{\mathrm{sm}}(\C)$. Thus $\mathrm{rank}_{\mathrm{Betti}}(X^{[m]},(x,\ldots,x)) = 2 \dim X + 2(m-1)\dim X_s = 2 \dim X^{[m]}$. Hence we are done by \eqref{EqBettiRankBettiForm}.
\end{proof}

\subsection{A construction of non-degenerate subvarieties}
In applications, especially \cite{DGHUnifML} and \cite{KuehneUnifMM}, it is necessary to have some reasonable non-degenerate subvariety to start with. The following result \cite[Thm.1.2']{GaoBettiRank}, and more generally \cite[Thm.1.3]{GaoBettiRank}, play a crucial role.

\begin{thm}\label{ThmNonDegCA}
Let $S \rightarrow \mathbb{M}_g$ be a generically finite morphism. Let $\mathscr{D}_M$ be as from \eqref{EqFaltingZhang}.
%\begin{enumerate}
%\item[(i)] 
Then $\mathscr{D}_M(\mathfrak{C}_g^{[M+1]}) \times_{\mathbb{M}_g} S$ is a non-degenerate subvariety of $\mathrm{Jac}(\mathfrak{C}_g / \mathbb{M}_g)^{[M]} \times_{\mathbb{M}_g} S$ for $M \ge \dim S+1$;
%\item[(ii)] if $\mathfrak{C}_S \rightarrow S$ admits a multisection $\epsilon$ and $\mathfrak{C}_S$ is viewed as a subvariety of $\mathfrak{A}_g$ based at $\epsilon$, then $\mathfrak{C}_S^{[M]}$ is a non-degenerate subvariety of $\mathfrak{A}_g^{[M]}$ for $M \ge \dim S$.
%\end{enumerate}
\end{thm}

Theorem~\ref{ThmNonDegCA} is a particular case of the more general \cite[Thm.10.1]{GaoBettiRank}, which we state now. We expect \cite[Thm.10.1]{GaoBettiRank} (with $t=0$) to have more applications, for example for the uniform Mordell--Lang conjecture for higher dimensional subvarieties of abelian varieties.

Let $\pi \colon \cA \rightarrow S$ be an abelian scheme as at the beginning of this section, and $\iota\colon \cA \rightarrow \mathfrak{A}_g$ be the modular map from \eqref{EqModularMap}. For each integer $M \ge 1$, set $\cA^{[M]}:=\cA\times_S\cdots\times_S\cA$ ($M$-copies). Define the \textit{Faltings--Zhang} map
\begin{equation}\label{EqFaltingsZhangGeneral}
\mathscr{D}_M^{\cA} \colon \cA^{[M+1]} \rightarrow \cA^{[M]}
\end{equation}
to be the $S$-morphism fiberwise defined by $(P_0,P_1,\ldots,P_M) \mapsto (P_1-P_0,\ldots,P_M-P_0)$. 

For each $M \ge 1$, let $\iota^{[M]} \colon \cA^{[M]} \rightarrow \mathfrak{A}_{Mg}$ be the modular map. As the convention of \cite{GaoBettiRank} is somewhat different from standard notation, we state the result under the formulation of \cite[Thm.4.4.4]{GaoHDR}.
\begin{thm}\label{ThmNonDegXA}
%$\mathsection$\ref{SectionNonDeg}. 
 Let $X$ be an irreducible subvariety of $\cA$ such that $\pi|_X$ is dominant to $S$.  
Assume that $X_{\bar{\eta}}$ (the geometric generic fiber of $X \rightarrow S$) is irreducible.\footnote{This assumption is harmless because it can always be achieved in the following way. There exists a quasi-finite \'{e}tale morphism $S' \rightarrow S$ such that some component $X'$ of  $X \times_S S'$ satisfies that $X'_{\bar{\eta}}$ is irreducible. But $X'$ dominates $X$ under the natural projection $X\times_S S' \rightarrow X$. In applications, we apply this theorem to $X' \subseteq \cA\times_S S' \rightarrow S'$.}
 
Assume furthermore
 \begin{enumerate}
 \item[(a)] $\dim X > \dim S$.
 \item[(b)] $X_s$ is generates $\cA_s$ for each $s \in S(\C)$.
 \item[(c)] On the geometric generic fiber $\cA_{\bar{\eta}}$ of $\cA\rightarrow S$, the stabilizer of $X_{\bar{\eta}}$, which we denote by $\mathrm{Stab}_{\cA_{\bar{\eta}}}(X_{\bar{\eta}})$, is finite.
 \end{enumerate}
Then as subvarieties of $\cA^{[M]}$, we have that
\begin{enumerate}
\item[(i)] $X^{[M]}$ is non-degenerate if $M \ge \dim S$ and $\iota^{[M]}|_{X^{[M]}}$ is generically finite.
\item[(ii)] $\mathscr{D}_M^{\cA}(X^{[M+1]})$ is non-degenerate if $M \ge \dim X$ and $\iota^{[M]}|_{\mathscr{D}_M^{\cA}(X^{[M+1]})}$ is generically finite.
\end{enumerate}
Here $X^{[M]} = X\times_S \cdots \times_S X$ ($M$-copies) for each integer $M \ge 1$.
\end{thm}

In practice, to verify the extra generic finiteness required in (i) and (ii), one can sometimes use the following observations.
For (i),  $\iota^{[M]}|_{X^{[M]}}$ is generically finite if $\iota|_X$ is generically finite. For (ii),  $\iota^{[M]}|_{\mathscr{D}_M^{\cA}(X^{[M+1]})}$ is generically finite if $\iota$ (and not $\iota|_X$) is quasi-finite.

%Notice that $X^{[M]}$ may not be irreducible, because $X_{\bar{\eta}}$ may not be irreducible. But it is not hard to show that $X^{[M]}$ is equidimensional, and so is $\mathscr{D}_M^{\cA}(X^{[M+1]})$. The conclusion for (i) and (ii) means that every irreducible component is non-degenerate.

%\footnote{This is only a technical assumption to guarantee that $X^{[M]}$ is irreducible. This assumption can always be attained up to replacing $S$ by a finite covering.}

%\medskip

Although hypothesis (b) implies hypothesis (a), but we still list hypothesis (a) here to emphasize that this construction does not work if $X \rightarrow S$ is a multi-section.

In fact, \cite[Thm.10.1]{GaoBettiRank} is stronger than Theorem~\ref{ThmNonDegXA}. It says that Theorem~\ref{ThmNonDegXA} still holds true with hypothesis (c) replaced by the weaker hypothesis
 \begin{enumerate}
  \item[(c')] On the geometric generic fiber $\cA_{\bar{\eta}}$ of $\cA\rightarrow S$, the neutral component of $\mathrm{Stab}_{\cA_{\bar{\eta}}}(X_{\bar{\eta}})$ is contained in the $\bar{\C(\eta)}/\C$-trace of $\cA_{\bar{\eta}}$, where $\bar{\C(\eta)}$ is an algebraic closure of the function field of $S$.
\end{enumerate}

%The assumption ``\textit{$\iota|_X$ is generically finite}'' is sometimes not easy to attain, especially when working with higher dimensional subvarieties of abelian varieties. The more general \cite[Thm.10.1]{GaoBettiRank} (applied to $t=0$) is helpful for this purpose: in practice, it is often possible to find an $M \gg 1$ such that $\iota^{[M]}|_{X^{[M]}}$ is generically finite for the modular map $\iota^{[M]} \colon \cA^{[M]} \rightarrow \mathfrak{A}_{Mg}$.
%\begin{thm}
%Let $X$ be a subvariety of $\cA$ such that $\pi|_X$ is dominant to $S$. Assume that $X_{\bar{\eta}}$ (the geometric generic fiber of $X \rightarrow S$) is irreducible. Assume that $X$ satisfies the hypotheses (a), (b) and (c) in Theorem~\ref{ThmNonDegXA} (or (a), (b) and (c')). Then we have
%\begin{enumerate}
%\item[(i)] $X^{[M]}$ is non-degenerate if $M \ge \dim S$ and $\iota^{[M]}|_{X^{[M]}}$ is generically finite.
%\item[(ii)] $\mathscr{D}_M^{\cA}(X^{[M+1]})$ is non-degenerate if $M \ge \dim X$ and $\iota^{[M]}|_{\mathscr{D}_M^{\cA}(X^{[M+1]})}$ is generically finite.
%\end{enumerate}
%\end{thm}

For the general criterion of non-degeneracy, we refer to \cite[Thm.1.1]{GaoBettiRank} and \cite[Thm.4.3.1]{GaoHDR}. In some particular cases, the criterion can be simplified; see \textit{e.g.} \cite[(1.4)]{GaoBettiRank}.

\subsection{The degeneracy locus}
In this subsection we state another fundamental result about non-degeneracy. It claims that being non-degenerate is in fact an algebraic property.
%\begin{equation}\label{EqDegLocusGeom}
%\{x \in X^{\mathrm{sm}}(\C) : \mathrm{rank}_{\mathrm{Betti}}(X,x) < 2\dim X\} = \{x \in X^{\mathrm{sm}}(\C) : (\omega|_X^{\wedge \dim X})_x = 0\}.
%\end{equation}
%From the definition, \eqref{EqDegLocusGeom} is a real analytic subset of $X^{\mathrm{an}}$. However, the following \cite[Thm.1.1.(ii)]{GaoBettiRank} asserts that it is algebraic.

\begin{thm}\label{ThmDegLocusZarClosed}
To each $X$, one can associate an intrinsically defined Zariski closed subset $X^{\mathrm{deg}}$ of $X$ such that the following property holds true: 
%the set in \eqref{EqDegLocusGeom} is precisely $(X^{\mathrm{deg}} \cap X^{\mathrm{sm}})(\C)$
$X$ is non-degenerate if and only if $X \not= X^{\mathrm{deg}}$. 

Moreover if $X$ is defined over an algebraically closed field $F$, so is $X^{\mathrm{deg}}$.
\end{thm}
This formulation of Theorem~\ref{ThmDegLocusZarClosed} is taken from \cite[Thm.4.3.1 and Prop.4.2.4]{GaoHDR}. 
The result follows essentially from \cite[Thm.1.7 and Thm.1.8]{GaoBettiRank} and their proofs.

%One should distinguish $X^{\mathrm{deg}}$ and the following set
%\begin{equation}\label{EqDegLocusGeom}
%X^{\mathrm{lin-deg}}:= \{x \in X^{\mathrm{sm}}(\C) : \mathrm{rank}_{\mathrm{Betti}}(X,x) < 2\dim X\} = \{x \in X^{\mathrm{sm}}(\C) : (\omega|_X^{\wedge \dim X})_x = 0\}.
%\end{equation}
%From the definition, $X^{\mathrm{lin-deg}}$ is a closed real analytic subset of $X^{\mathrm{an}}$. 

%The equidistribution result in \cite{KuehneUnifMM}  is proved for functions supported on $X \setminus X^{\mathrm{lin-deg}}$. Moreover, 
To be able to compute the constant $c(g)$ from Theorem~\ref{ConjMazur}, one needs a  better understanding of $X^{\mathrm{deg}}$. We refer to \cite[$\mathsection$1.2]{GaoBettiRank} and \cite[$\mathsection$4.2]{GaoHDR} for the definition and some further discussions on $X^{\mathrm{deg}}$.

We close the main part of this section by outlining the main steps of to study the non-degeneracy in \cite{GaoBettiRank}. Both Theorem~\ref{ThmNonDegXA} and Theorem~\ref{ThmDegLocusZarClosed} are proved following this guideline, where functional transcendence and the unlikely intersection theory are heavily used. 
The major step is to establish a criterion, \textit{in simple geometric terms}, for an irreducible subvariety $X$ of the universal abelian variety $\mathfrak{A}_g$ to be degenerate. Roughly speaking, the proof of the desired criterion is divided into two steps. Step~1 transfers the degeneracy property to an \textit{unlikely intersection problem} in $\mathfrak{A}_g$ by invoking the \textit{mixed Ax--Schanuel theorem} for $\mathfrak{A}_g$ \cite[Thm.1.1]{GaoMixedAS}. More precisely we show that $X$ is degenerate if and only if $X$ is the union of subvarieties satisfying an appropriate unlikely intersection property. Step~2 solves this unlikely intersection problem, and the key point is to use \cite[Thm.1.4]{GaoMixedAS} to prove that the union mentioned above is a finite union. In this step the notion of \textit{weakly optimal subvarieties} introduced by Habegger--Pila \cite{HabeggerPilaENS} is involved.

\small

\subsection{$1$-parameter case}
When $\dim S = 1$, the criterion of non-degeneracy and the degeneracy locus are easier to describe.

\begin{defn}\label{DefinitionSpecialSubvariety}
An irreducible closed subvariety $Y$ of $\cA$ is 
called a \textbf{generically special} subvariety of $\cA$, or just
generically special, if
it dominates $S$ and if
its geometric generic fiber $Y\times_S \spec{\overline{\C(S)}}$ 
is a finite
 union of $%\mathrm{Tr}_{\overline{k(S)}/k}
 (Z\otimes_\C \overline{\C(S)})+B$, where $Z$ is a closed
irreducible subvariety of $A^{\overline{\C(S)}/\C}$ (the $\overline{\C(S)}/\C$-trace of $A$) and $B$ is a torsion
coset in $A\otimes_{\C(S)} \overline{\C(S)}$.
\end{defn}

We then have the following results from \cite[Thm.5.1, Prop.1.3]{GaoHab}.
\begin{thm}
Assume $\dim S = 1$. Let $X$ be  an irreducible closed subvariety of $\cA$ which is dominant to $S$. Then
\begin{enumerate}
\item[(i)] $X$ is degenerate if and only if $X$ is generically special;
\item[(ii)] we have
\[
X^{\mathrm{deg}} = \bigcup_{\substack{Y \subseteq X \\ Y\text{ is a generically
special}\\ \text{subvariety of $\cA$}}} Y.
\]
The union is a finite union.
\end{enumerate}
\end{thm}

\normalsize

%\subsection{Geometric description of $X^{\mathrm{deg}}$}
%The goal of this subsection is to give the geometric description of $X^{\mathrm{deg}}$. It is needed if one wants compute the constant $c(g)$ in Theorem~\ref{ConjMazur} with this method.

%We start with the case $\dim S = 1$. 

%% Section 7
\section{The height inequality and its application}\label{SectionHtIneqEqdistr}
This section is based on \cite{DGHUnifML}. 
Let $S$ be a quasi-projective irreducible variety and let $\pi \colon \cA \rightarrow S$ be an abelian scheme of relative dimension $g$, both over $\IQbar$.

Let $\cL$ be a relatively ample line bundle on $\cA/S$ with $[-1]^*\cL \cong \cL$, and let $\cM$ be a line bundle over a compactification $\bar{S}$ of $S$. All these data are assumed to be defined over $\IQbar$. Then we have a fiberwise N\'{e}ron--Tate height function $\hat{h}_{\cL} \colon \cA(\IQbar) \rightarrow [0,\infty)$ as in \eqref{EqFiberwiseNTHeight}, and a height function $h_{\bar{S},\cM} \colon S(\IQbar) \rightarrow \R$ provided by the Height Machine \eqref{EqHtMachine}.

\subsection{Statement of the height inequality}
For any irreducible subvariety $X$ of $\cA$, set $X^* = X \setminus X^{\mathrm{deg}}$ with $X^{\mathrm{deg}}$ the Zariski closed subset of $X$ from Theorem~\ref{ThmDegLocusZarClosed}. Then $X^* \not= \emptyset$ if and only if $X$ is non-degenerate.

Here is the height inequality of Dimitrov--Gao--Habegger from \cite{DGHUnifML}. When $\dim S = 1$ it is proved in \cite{GaoHab}.

\begin{thm}\label{ThmHtIneq}
Let $X$ be an irreducible subvariety of $\cA$ defined over $\IQbar$. Let $X^* = X \setminus X^{\mathrm{deg}}$ be the Zariski open subset of $X$ as defined above; it is defined over $\IQbar$.

There exist constants $c > 0$ and $c'$, depending only on $X$ and the data of the height functions, such that 
\begin{equation}\label{EqHtIneq}
\hat{h}_{\cL}(P) \ge c h_{\bar{S},\cM}(\pi(P)) - c' \quad \text{ for all }P \in X^*(\IQbar).
\end{equation}
\end{thm}
This theorem is non-trivial only if $X$ is non-degenerate (otherwise $X^*= \emptyset$). The version stated here is a minor improvement of \cite[Thm.1.6 and Thm.B.1]{DGHUnifML}. It follows from a simple Noetherian induction from \cite[Thm.B.1]{DGHUnifML}, the Zariski closedness of $X^{\mathrm{deg}}$ and the geometric description of $X^{\mathrm{deg}}$; \textit{cf.} \cite[Thm.4.4.2]{GaoHDR}. Another minor improvement is that $\cM$ is not required to be ample on $\bar{S}$ as in \cite{DGHUnifML}, as this extra requirement can easily be dropped by the Height Machine.

In practice, to apply Theorem~\ref{ThmHtIneq}, one needs to have some non-degenerate subvarieties to start with. For this purpose, apart from directly applying the criterion of non-degeneracy \cite[Thm.1.1]{GaoBettiRank}, the construction in Theorem~\ref{ThmNonDegXA} is a useful tool.

We point out that the constants in \eqref{EqHtIneq} are effective; see \cite[Rmk.5.1]{DGHUnifML} for comments on $c$ (which is denoted by $c_1'$ in \textit{loc.cit.}).

\subsection{Application to the New Gap Principle}
%Use the notation from $\mathsection$\ref{SectionNGP}.

In this subsection, we use Theorem~\ref{ThmHtIneq} and Theorem~\ref{ThmNonDegXA} to prove a proposition in the flavor of the New Gap Principle Theorem~\ref{ThmNGP}. In fact, applying the proposition to an appropriate family yields \cite[Prop.7.1]{DGHUnifML}, which is a weaker version of the New Gap Principle; see the end of $\mathsection$\ref{SubsectionProofNGP}.

% Theorem~\ref{ThmNGP}.% The proof is a finer treatment of \cite[Prop.7.1]{DGHUnifML}.% Apart from rendering the constants more explicit, only the last paragraph is slightly different from the proof of \cite[Prop.7.1]{DGHUnifML}.

This proof, in line with \cite[Prop.7.1]{DGHUnifML}, is a good example for how the height inequality Theorem~\ref{ThmHtIneq} is applied to Diophantine problems. Moreover, the framework of the proof will also be used in Proposition~\ref{PropUnifBog} (Step~4) and Lemma~\ref{LemmaRelBogUnifBog}.

We also render the constants from \cite[Prop.7.1]{DGHUnifML} more explicit, by applying the refined height inequality and by making Lemma~\ref{LemmaNogaAlon} (which is \cite[Lem.6.3]{DGHUnifML}) explicit.

Let $\cA/S$, $\cL$ and $\cM$ be as the beginning of this section. Write $\iota\colon \cA \rightarrow \mathfrak{A}_g$ for the modular map \eqref{EqModularMap2}.% For each $s \in S(\IQbar)$, write $\cA_s$ the fiber of $\cA \rightarrow S$ over $s$.

\begin{prop}\label{PropNGPLargeCurve}
Let $\mathfrak{C} \subseteq \cA$ be an irreducible subvariety satisfying the following properties. 
Each fiber $\mathfrak{C}_s$ of $\mathfrak{C} \rightarrow S$ is an irreducible curve which generates $\cA_s$ and is not a translate of an elliptic curve, and $\iota|_{\mathfrak{C}\times_S S'}$ is generically finite for all subvarieties $S' \subseteq S$.

Then there exist constants $c_1',c_2'$ and $c_3'$ such that for each $s \in S(\IQbar)$, we have
\begin{equation}
\#\left\{x \in \mathfrak{C}_s(\IQbar) : \hat{h}_{\cL}(x) \le c_1' \max\{1, h_{\bar{S},\cM}(s)\} - c_3' \right\} < c_2'.
\end{equation}
\end{prop}

\begin{proof}
We prove this proposition by induction on $\dim S$. The proof for the base step $\dim S =0$ is contained in the induction step.

Fix $M \ge \dim S$. The properties of $\mathfrak{C}$ allows us to apply Theorem~\ref{ThmNonDegXA}(i) applied to $\mathfrak{C} \subseteq \cA \rightarrow S$. So $\mathfrak{C}^{[M]}$ is a non-degenerate subvariety of $\cA^{[M]}$. Set $X: = \mathfrak{C}^{[M]}$. Let $X^{\mathrm{deg}}$ be the degeneracy locus of $X$ from Theorem~\ref{ThmDegLocusZarClosed}; it is Zariski closed in $X$ and is defined over $\IQbar$. Moreover $X^{\mathrm{deg}} \not= X$ as $X$ is non-degenerate.

Let $X^* = X\setminus X^{\mathrm{deg}}$; it is Zariski open dense in $X$. Applying the height inequality, Theorem~\ref{ThmHtIneq}, to $X$ and $\cA \rightarrow S$, we get
\begin{equation}\label{EqHtIneqApp}
\hat{h}_{\cL}(x_1) + \cdots + \hat{h}_{\cL}(x_M) \ge c h_{\bar{S},\cM}(s) - c'
\end{equation}
for all $s \in S(\IQbar)$ and $(x_1, \ldots, x_M) \in X^*(\IQbar)$ in the fiber above $s$. 

As $X^*$ is Zariski open dense in $X$, each irreducible component of $\bar{S \setminus \pi(X^*)}$ has dimension $\le \dim S - 1$. By induction hypothesis, it suffices to prove the proposition with $S$ replaced by $S \setminus \bar{S \setminus \pi(X^*)}$. Therefore we may and do assume $\pi(X^*) = S$. Thus for each $s \in S(\IQbar)$, the fiber of $X^*$ over $s$ is non-empty.

Use $X_s$, $X^*_s$ and $X^{\mathrm{deg}}_s$ to denote the corresponding fibers over $s$. Then the last sentence of the previous paragraph says $X^*_s \not= \emptyset$ for each $s \in S(\IQbar)$. Equivalently,
\begin{equation}\label{EqNGPDGHHyp}
X^{\mathrm{deg}}_s  \not= X_s = \mathfrak{C}_s^M.% = \mathscr{D}_M(\mathfrak{C}_s^{[M+1]}) \subseteq (\mathfrak{A}_g)_{\tau(s)}.
\end{equation}

This allows us to apply Lemma~\ref{LemmaNogaAlon} to $V = \cA_s$, $L = \cL_s$, $C = \mathfrak{C}_s$ and $Z = X^{\mathrm{deg}}_s$. Thus setting
\begin{equation}
c_2' := \max_{s \in S(\IQbar)}\deg_{\cL_s}(\mathfrak{C}_s)^{M(M+1)/2}\deg_{\cL_s}(\cA_s)^{M(M-1)/2}\deg_{\cL_s^{\boxtimes M}}(X^{\mathrm{deg}}_s) + 1,
\end{equation}
the following holds true. \footnote{In a flat family, all fibers have the same degree. Thus $c_2'$ exists, possibly by a Noetherian induction.}  If a subset $\Sigma \subseteq \mathfrak{C}_s(\IQbar)$ has cardinality $\ge c_2'$, then $\Sigma^M \not\subseteq X^{\mathrm{deg}}_s$.
 
We work with $\Sigma =  \{x \in \mathfrak{C}_s(\IQbar) : \hat{h}_{\cL}(x) \le c_1'\max\{1,h_{\bar{S},\cM}(s)\} - c_3'\}$, with
\begin{equation}
c_1' = c/2M \quad \text{ and } \quad c_3' = (c+c')/M,
\end{equation} 
where $c$ and $c'$ come from the height inequality \eqref{EqHtIneqApp}.

We claim that $\#\Sigma < c_2'$. Assume otherwise, then $\Sigma^M \not\subseteq X^{\mathrm{deg}}_s$, and thus there exist $x_1,\ldots,x_M \in \Sigma$ such that $(x_1, \ldots, x_M) \not\in X^{\mathrm{deg}}_s$. Hence \eqref{EqHtIneqApp} holds true, and we thus obtain
\[
c h_{\bar{S},\cM}(s) - c' \le M c_1'\max\{1,h_{\bar{S},\cM}(s)\} - M c_3' = \frac{1}{2}c \max\{1,h_{\bar{S},\cM}(s)\} - (c + c').
\]
As $c \max\{1,h_{\bar{S},\cM}(s)\} \le c (1+h_{\bar{S},\cM}(s))$, the inequality above implies
\[
c \max\{1,h_{\bar{S},\cM}(s)\} - c - c' \le \frac{1}{2}c \max\{1,h_{\bar{S},\cM}(s)\} - (c + c').
\]
But this last inequality cannot hold true. So we get a contradiction, and hence $\#\Sigma < c_2'$. This is precisely the desired bound, and hence we are done.
\end{proof}

The following lemma as well as the proof presented here is nothing but \cite[Lem.6.3]{DGHUnifML}, with the bound written explicitly. Let $k$ be an algebraically closed field and all varieties are assumed to be defined over $k$. Let $M \ge 1$ be an integer.
\begin{lemma}\label{LemmaNogaAlon}
Let $V$ be a projective irreducible variety with an ample line bundle $L$. Let $C$ be an irreducible curve in $V$ and let $Z$ be a Zariski closed subset of $V^M$. Assume $C^M \not\subseteq Z$. Then if $\Sigma \subseteq C(k)$ has cardinality $> \deg_L(C)^{M(M+1)/2}\deg_L(V)^{M(M-1)/2}\deg_{L^{\boxtimes M}}(Z)$, then $\Sigma^M \not\subseteq Z(k)$.
\end{lemma}
\small
\begin{proof}
We prove this lemma by induction on $M$. The base step $M=1$ follows immediately from B\'{e}zout's Theorem.

Assume the lemma is proved for $1, \ldots, M-1 \ge 1$. Let $q \colon V^M \rightarrow V$ be the projection to the first factor.

B\'{e}zout's Theorem implies $\sum_{Y} \deg_{L^{\boxtimes M}}(Y) \le \deg_L(C)^M\deg_{L^{\boxtimes M}}(Z)$ with $Y$ running over all irreducible components of $C^M \cap Z$. Let $Z'$ be the union of such $Y$'s with $\dim q(Y) \ge 1$, and $Z''$ be the union of the other components. Then $\deg_{L^{\boxtimes M}}(Z'), \deg_{L^{\boxtimes M}}(Z'') \le \sum_{Y} \deg_{L^{\boxtimes M}}(Y)$ and hence 
\begin{equation}\label{EqBezout}
\deg_{L^{\boxtimes M}}(Z') \le \deg_L(C)^M\deg_{L^{\boxtimes M}}(Z), \quad \deg_{L^{\boxtimes M}}(Z'') \le \deg_L(C)^M\deg_{L^{\boxtimes M}}(Z)
\end{equation}

Note that $q(Z') \subseteq q(C^M \cap Z) \subseteq C$. For all $P \in C(k)$, the fiber $q|_{Z'}^{-1}(P) = Z' \cap (\{P\}\times V^{M-1})$ has dimension at most $\dim Z' - 1 \le M-2$. So $\{P\}\times C^{M-1} \not\subseteq Z' \cap (\{P\}\times V^{M-1})$. Write $i \colon \{P\}\times V^{M-1} \cong V^{M-1}$ for the natural isomorphism. Then we can apply the induction hypothesis and conclude: if $\Sigma \subseteq C(k)$ has cardinality $> \deg_L(C)^{M(M-1)/2}\deg_L(V)^{(M-1)(M-2)/2}\deg_{L^{\boxtimes (M-1)}}i(Z' \cap (\{P\}\times V^{M-1}))$, then $\Sigma^{M-1} \not\subseteq i(Z' \cap (\{P\}\times V^{M-1}))(k)$.

But $\deg_{L^{\boxtimes (M-1)}}i(Z' \cap (\{P\}\times V^{M-1})) = \deg_{L^{\boxtimes M}}(Z' \cap (\{P\}\times V^{M-1}))$ (since $\deg_L(P) = 1$) and $\deg_{L^{\boxtimes M}}(Z' \cap (\{P\}\times V^{M-1})) \le \deg_{L^{\boxtimes M}}(Z') \deg_L(V)^{M-1}$ by B\'{e}zout's Theorem. So we can replace $\deg_{L^{\boxtimes (M-1)}}i(Z' \cap (\{P\}\times V^{M-1}))$ in the conclusion of last paragraph by 
$\deg_{L^{\boxtimes M}}(Z') \deg_L(V)^{M-1}$. Thus by \eqref{EqBezout} we get: if $\Sigma \subseteq C(k)$ satisfies
\[
\#\Sigma > \deg_L(C)^{M(M-1)/2}\deg_L(V)^{(M-1)(M-2)/2} \cdot \deg_L(C)^M\deg_{L^{\boxtimes M}}(Z) \deg_L(V)^{M-1},% = \deg_L(C)^{M(M+1)/2}\deg_L(X)^{M(M-1)/2}\deg_{L^{\boxtimes M}}(Z),
\]
then $\{P\} \times \Sigma^{M-1} \not\subseteq Z'(k)$ for all $P \in C(k)$ (and hence $\Sigma^M \not\subseteq Z'(k)$). Notice that the right hand side is precisely $\deg_L(C)^{M(M+1)/2}\deg_L(V)^{M(M-1)/2}\deg_{L^{\boxtimes M}}(Z)$.

Now $\dim q(Z'') = 0$, so $q(Z'')$ is a finite set of cardinality at most the number of irreducible components of $Z''$, which is at most $\deg_{L^{\boxtimes M}}(Z'')$ by definition of the degree. Hence $\#q(Z'') \le \deg_L(C)^M\deg_{L^{\boxtimes M}}(Z)$ by \eqref{EqBezout}. So if $\Sigma \subseteq C(k)$ has cardinality $>\deg_L(C)^M\deg_{L^{\boxtimes M}}(Z)$, then $\Sigma^M \not\subseteq Z''(k)$.

Thus the lemma holds true since $Z = Z' \cup Z''$.
\end{proof}

\normalsize

%% Section 8
\section{Equidistribution on non-degenerate subvarieties and its application}\label{SectionEqdistr}

This section is based on \cite{KuehneUnifMM}. 
Let $S$ be a quasi-projective irreducible variety and let $\pi \colon \cA \rightarrow S$ be an abelian scheme of relative dimension $g$, both over $\IQbar$.

Let $\cL$ be a relatively ample line bundle on $\cA/S$ defined over $\IQbar$ such that $[-1]^*\cL \cong \cL$. Then we have a fiberwise N\'{e}ron--Tate height function $\hat{h}_{\cL} \colon \cA(\IQbar) \rightarrow [0,\infty)$ as in \eqref{EqFiberwiseNTHeight}.%, and a height function $h_{\bar{S},\cM} \colon S(\IQbar) \rightarrow \R$ provided by the Height Machine.

\subsection{The equidistribution result}
Let $\omega$ be the Betti form on $\cA$ as provided by Construction~\ref{ConstrBettiForm}.

%For any irreducible subvariety $X$ of $\cA$, set $X^* = X \setminus X^{\mathrm{deg}}$ with $X^{\mathrm{deg}}$ the Zariski closed subset of $X$ from Theorem~\ref{ThmDegLocusZarClosed}. Then $X^* \not= \emptyset$ if and only if $X$ is non-degenerate.

The following equidistribution result is proved by K\"{u}hne \cite[Thm.1]{KuehneUnifMM}.
\begin{thm}\label{ThmEquidistr}
Let $X$ be a non-degenerate subvariety of $\cA$ defined over $\IQbar$. There exists a constant $k = k(X, \omega) > 0$ such that the following property holds true. For any generic sequence $\{x_n\}_{n \in \mathbb{N}}$ in $X$ (namely $x_n$ converges to the generic point of $X$) satisfying $\hat{h}_{\cL}(x_n) \rightarrow 0$, we have
\begin{equation}\label{EqEquidistr}
\frac{1}{\#O(x_n)} \sum_{y \in O(x_n)} f(y) \rightarrow k \int_{X^{\mathrm{an}}} f (\omega|_X)^{\wedge \dim X}
\end{equation}
for all $f \in \mathscr{C}^0_{\mathrm{c}}(X^{\mathrm{an}})$ (continuous compactly supported in $X(\C)$). Here $O(x_n)$ means the Galois orbit of $x_n$.
\end{thm}

The sequence $\{x_n\}$ in Theorem~\ref{ThmEquidistr} will be called a \textit{generic small sequence} in $X$.

This equidistribution result was proved by DeMarco--Mavraki \cite[Cor.1.2]{DeMarcoMavraki} when $\cA \rightarrow S$ is a fiber product of elliptic surfaces and $X$ is a section.

\medskip

The first step to use equidistribution to study Bogomolov type problems is through the following corollary, which is a minor improvement of \cite[Lem.22]{KuehneUnifMM}. The idea already showed up in the work of Ullmo \cite{Ullmo} and S.~Zhang \cite{ZhangEquidist}. We include the proof in this survey as it is not complicated and because of the importance of the corollary. This proof is almost a literal copy of \cite[Lem.22]{KuehneUnifMM}.

\begin{cor}\label{CorEquidistr}
Let $X$ be a non-degenerate subvariety of $\cA$ defined over $\IQbar$. Set $\mu = k (\omega|_X)^{\wedge \dim X}$ to be the measure on $X(\C)$ with $k=k(X,\omega)>0$ the constant from Theorem~\ref{ThmEquidistr}.

For each function $f \in  \mathscr{C}^0_{\mathrm{c}}(X^{\mathrm{an}})$ and every $\epsilon > 0$, there exist a proper subvariety $Z_{f,\epsilon}$ of $X$ and a constant $\delta_{\epsilon} > 0$ such that each $x \in (X\setminus Z_{f,\epsilon})(\IQbar)$ satisfies the following alternative:
\begin{enumerate}
\item[(i)] Either $\hat{h}_{\cL}(x) \ge \delta_{\epsilon}$;
\item[(ii)]  or $\left| \frac{1}{\#O(x)} \sum_{y \in O(x)} f(y) - \int_{X^{\mathrm{an}}} f \mu \right| < \epsilon$.
\end{enumerate}
\end{cor}
\small
\begin{proof} 
To invoke Theorem~\ref{ThmEquidistr}, we need a generic small sequence in $X$. Let us first explain why we can assume this.

Consider, for each $n \in \mathbb{N}$, the set $X_n := \{x \in X(\IQbar) : \hat{h}_{\cL}(x) < 1/n\}$. Then we have a descending chain $\cdots \supseteq X_n \supseteq X_{n+1} \supseteq \cdots$. Assume that $X_n$ is not Zariski dense in $X$ for some $n$. Then for any $f$ and $\epsilon$, one can take $\delta_{\epsilon} = 1/n$ and $Z_{f,\epsilon} = \bar{X_n}^{\mathrm{Zar}}$. Notice that in this case, part (i) always holds true.

So from now on, we assume that $X_n$ is Zariski dense in $X$ for all $n \gg 1$. There are only countably many proper closed subvarieties of $X$ defined over $\IQbar$, say $\{Z_n\}_{n\in \mathbb{N}}$. For each $n \in \mathbb{N}$, take $x_n \in X_n \setminus Z_n(\IQbar)$. Such an $x_n$ exists because $X_n$ is Zariski dense in $X$ and $X\setminus Z_n$ is Zariski open dense in $X$. Then $\hat{h}_{\cL}(x_n) \rightarrow 0$, and $x_n$ converges to the generic point of $X$. Hence  $\{x_n\}_{n \in \mathbb{N}}$ is a generic small sequence in $X$. Therefore we are in the situation of Theorem~\ref{ThmEquidistr}.

Suppose that the conclusion is false. Then there exist some $f \in \mathscr{C}_{\mathrm{c}}(X^{\mathrm{an}})$ and some $\epsilon > 0$ with the following property. For any $\delta >0$, the set
\[
\mathscr{B}_{\delta}:= \left\{x \in X(\IQbar) : \hat{h}_{\cL}(x) < \delta \text{ and } \left| \frac{1}{\#O(x)} \sum_{y \in O(x)} f(y) - \int_{X^{\mathrm{an}}} f \mu \right| \ge \epsilon \right\}
\]
is Zariski dense in $X$. Then as in the previous paragraph, we can find a generic small sequence $\{x_n\}_{n \in \mathbb{N}}$ in $X$, with each $x_n \in \mathscr{B}_{1/n}$, such that
\[
\left| \frac{1}{\#O(x_n)} \sum_{y \in O(x_n)} f(y) - \int_{X^{\mathrm{an}}} f \mu \right| \ge \epsilon
\]
for all $n$. This contradicts the equidistribution \eqref{EqEquidistr}. Hence we are done.
\end{proof}

\normalsize

\subsection{Application to uniform Bogomolov}
Ullmo \cite{Ullmo} and S.~Zhang \cite{ZhangEquidist} used equidistribution results to prove the Bogomolov conjecture on a single abelian variety over $\IQbar$. A key idea in this approach is to apply the equidistribution result twice and compare the measures on two varieties linked by the \textit{Faltings--Zhang map}. The upshot is that we are \textit{not} in case (ii) of the alternative in the single-abelian-variety version of Corollary~\ref{CorEquidistr}. 

It is natural to expect that the equidistribution result in families (Theorem~\ref{ThmEquidistr}) can be applied to solve some family-version Bogomolov type problems, \textit{provided that there are some non-degenerate subvarieties to start with}.

A useful tool to construct non-degenerate subvarieties is Theorem~\ref{ThmNonDegXA}. Starting from this construction, K\"{u}hne  ran a modified version of Ullmo--Zhang's approach on families of curves in abelian schemes using his family version of the equidistribution (more precisely, Corollary~\ref{CorEquidistr}). 
In the end, with a fiberwise consideration as in the proof of Proposition~\ref{PropNGPLargeCurve}, he proved the following result \cite[Prop.21]{KuehneUnifMM}. %\footnote{} %, sometimes known as the \textit{uniform Bogomolov conjecture} for curves embedded into their Jacobians. 
 We include this beautiful proof in this survey.

%\begin{thm}\label{ThmUnifBogm}
%Use the notation from $\mathsection$\ref{SectionNGP}. %\ref{SubsectionUnivFamily} and \ref{SubsectionHtUnivFamily}. 
% There exist constants $c_2''$ and $c_3''$, depending only on $\mathfrak{C}_g/\mathbb{M}_g$ and the bundle $\cL$ such that the following property holds. For each $s \in \mathbb{M}_g(\IQbar)$ and each $P \in \mathfrak{C}_s(\IQbar)$, we have $\#\{Q \in \mathfrak{C}_s(\IQbar) : \hat{h}_{\cL}(Q-P) \le c_3''\} < c_2''$.
%\end{thm}

%We will apply Theorem~\ref{ThmNonDegXA} to construct non-degenerate subvarieties. Let us explain what the abelian scheme $\cA \rightarrow S$ and the subvariety $X$ are for the current situation.

%Let $S$ be an irreducible subvariety of the universal curve $\mathfrak{C}_g$ defined over $\IQbar$.

%At this stage, Theorem~\ref{ThmUnifBogm} has been turned into the following more abstract statement \cite[Prop.21]{KuehneUnifMM}.

Let $\iota \colon \cA \rightarrow \mathfrak{A}_g$ be the modular map from \eqref{EqModularMap2}. 
\begin{prop}\label{PropUnifBog}
Let $\mathfrak{C} \subseteq \cA$ be an irreducible subvariety satisfying the following properties. Each fiber $\mathfrak{C}_s$ of $\mathfrak{C} \rightarrow S$ is an irreducible curve which generates $\cA_s$ and is not a translate of an elliptic curve, and $\iota|_{\mathfrak{C}\times_S S'}$ is generically finite for all subvarieties $S' \subseteq S$.

Then there exist constants $c_2''$ and $c_3''$ such that for each $s \in S(\IQbar)$, we have
\begin{equation}
\#\{x \in \mathfrak{C}_s(\IQbar): \hat{h}_{\cL}(x) \le c_3''\} < c_2''.
\end{equation}
\end{prop}

Before moving on to the proof, we point out that Proposition~\ref{PropUnifBog} applied to a suitable family yields \cite[Thm.3]{KuehneUnifMM} immediately. See the end of $\mathsection$\ref{SubsectionProofNGP}.

\begin{proof}[Proof of Proposition~\ref{PropUnifBog}] 
We prove this proposition by induction on $\dim S$. The proof for the base step $\dim S =0$ is contained in the induction step.

For readers' convenience, we divide the proof into several steps.

\noindent{\boxed{Step~1}} Construct non-degenerate subvarieties.

%We wish to apply Theorem~\ref{ThmNonDegXA} to $\cA \rightarrow S$ and $X$. 

%We will construct three, instead of two, non-degenerate subvarieties from $\mathfrak{C}$. This is due to technical reasons, which will be seen in Step~2. The tool to construct non-degenerate subvarieties is Theorem~\ref{ThmNonDegXA}.\footnote{But it suffices to apply Theorem~\ref{ThmNonDegXA} twice instead of three times.}% to construct two non-degenerate subvarieties from $\mathfrak{C}$.

Fix $m \ge \dim S$. Consider $\mathfrak{C}^{[m]}:=\mathfrak{C}\times_S\cdots\times_S \mathfrak{C}$ ($m$-copies) and $\cA^{[m]} \rightarrow S$. By generic smoothness, there exists a Zariski open dense subset $S^{\circ}$ of $S$ such that $(\mathfrak{C}^{[m]})^{\mathrm{sm}} \times_S S^{\circ} \rightarrow S^{\circ}$ is a smooth morphism. Moreover up to replacing $S^{\circ}$ by a Zariski open dense subset, we may and do assume $S^{\circ}$ is smooth. Now that each irreducible component of $\overline{S \setminus S^{\circ}}$ has dimension $\le \dim S -1$, by induction hypothesis it suffices to prove the proposition with $S$ replaced by $S^{\circ}$. Hence we may and do assume:
\begin{equation}\label{EqGenericSmoothness}
S\text{ is smooth and }(\mathfrak{C}^{[m]})^{\mathrm{sm}} \rightarrow S \text{ is a smooth morphism}.
\end{equation}

By Theorem~\ref{ThmNonDegXA}(i) applied to $\mathfrak{C} \subseteq \cA \rightarrow S$, we have that $\mathfrak{C}^{[m]}:=\mathfrak{C}\times_S\cdots\times_S \mathfrak{C}$ ($m$-copies) is a non-degenerate subvariety of $\cA^{[m]}$. By Definition~\ref{DefnNonDeg}, for the Betti form $\omega_m$ of $\cA^{[m]}$, there exists a point $\mathbf{x} \in (\mathfrak{C}^{[m]})^{\mathrm{sm}}(\C)$ such that 
\begin{equation}\label{EqNonDegPoint}
(\omega_m|_{\mathfrak{C}^{[m]}}^{\wedge \dim \mathfrak{C}^{[m]}})_\mathbf{x} \not= 0.
\end{equation}

For each $M \gg 1$, 
%Fix $M  \ge \dim S + m = \dim \mathfrak{C}^{[m]}$. R
recall the proper $S$-morphism (for $\cA^{[m]}$ instead of $\cA$) from \eqref{EqFaltingsZhangGeneral}
\begin{equation}\label{EqDDMA}
\mathscr{D}_M^{\cA^{[m]}} \colon (\cA^{[m]})^{[M+1]} \rightarrow (\cA^{[m]})^{[M]}
\end{equation}
fiberwise defined by $(\mathbf{a}_0,\mathbf{a}_1,\ldots,\mathbf{a}_M) \mapsto (\mathbf{a}_1-\mathbf{a}_0,\ldots,\mathbf{a}_M-\mathbf{a}_0)$, with each $\mathbf{a}_i \in \cA^{[m]}(\IQbar)$.

By assumption on $\mathfrak{C}$ (no fiber is a translate of an elliptic curve), it is known that $\mathscr{D}_M^{\cA^{[m]}}|_{(\mathfrak{C}^{[m]})^{[M+1]}}$ is generically finite for $M \gg 1$.

A key point of the classical Ullmo--Zhang approach is to use $\mathscr{D}_M^{\cA^{[m]}}$. A novelty in K\"{u}hne's proof is to consider an extra factor
\begin{equation}\label{EqFZFinal}
\mathscr{D}:= (\mathrm{id}, \mathscr{D}_M^{\cA^{[m]}}) \colon \cA^{[m]} \times_S (\cA^{[m]})^{[M+1]} \rightarrow \cA^{[m]} \times_S (\cA^{[m]})^{[M]},
\end{equation}
which is generically injective. 
In $\cA^{[m]} \times_S (\cA^{[m]})^{[M+1]}$, we have a non-degenerate subvariety $\mathfrak{C}^{[m]} \times_S (\mathfrak{C}^{[m]})^{[M+1]}$.

Let us show that $\mathscr{D}(\mathfrak{C}^{[m]} \times (\mathfrak{C}^{[m]})^{[M+1]})$ is non-degenerate in $\cA^{[m]} \times_S (\cA^{[m]})^{[M]}$. Indeed, 
\[
\mathscr{D}(\mathfrak{C}^{[m]} \times_S (\mathfrak{C}^{[m]})^{[M+1]}) = \mathfrak{C}^{[m]} \times_S \mathscr{D}_M^{\cA^{[m]}}((\mathfrak{C}^{[m]})^{[M+1]}),
\]
and hence is non-degenerate because $\mathfrak{C}^{[m]}$ is non-degenerate; see Lemma~\ref{LemmaNonDegFiberProd}.

%It is not hard to check that all the hypotheses of Theorem~\ref{ThmNonDegXA} are satisfied for $X:=\mathfrak{C}^{[m]} \subseteq \cA^{[m]}\rightarrow S$. Thus applying Theorem~\ref{ThmNonDegXA}(ii) to $X:=\mathfrak{C}^{[m]} \subseteq \cA^{[m]}\rightarrow S$, we get that $\mathscr{D}(\mathfrak{C}^{[m(M+1)]}) = \mathscr{D}((\mathfrak{C}^{[m]})^{[M+1]})$ is non-degenerate in $\cA^{[mM]} = (\cA^{[m]})^{[M]}$. Moreover $(\mathfrak{C}^{[m]})^{[M+1]}$ is non-degenerate in $(\cA^{[m]})^{[M+1]}$ by \eqref{EqNonDegPoint} and Lemma~\ref{LemmaNonDegPointProduct}. 

Now we have obtained the two desired non-degenerate subvarieties $\mathfrak{C}^{[m]} \times_S (\mathfrak{C}^{[m]})^{[M+1]}$ and $\mathscr{D}(\mathfrak{C}^{[m]} \times (\mathfrak{C}^{[m]})^{[M+1]})$. In particular, we are in the situation of Corollary~\ref{CorEquidistr} for both.% $(\mathfrak{C}^{[m]})^{[M+1]}$ and $\mathscr{D}(\mathfrak{C}^{[m(M+1)]})$.

\noindent{\boxed{Step~2}} Choose suitable functions $f_1,f_2$ and constant $\epsilon > 0$ for later applications of Corollary~\ref{CorEquidistr}.

%For simplicity write $\mathscr{D} = \mathscr{D}|_{\mathfrak{C}^{[m(M+1)]}}$. Then $\mathscr{D}$ is a finite dominant morphism.

Let $\mu_1$ be the measure on $\mathfrak{C}^{[m]} \times_S (\mathfrak{C}^{[m]})^{[M+1]}(\C) = \mathfrak{C}^{[m(M+2)]}(\C)$ as in Corollary~\ref{CorEquidistr}, and let $\mu_2$ be the measure on $\mathscr{D}(\mathfrak{C}^{[m]} \times (\mathfrak{C}^{[m]})^{[M+1]})(\C) = \mathscr{D}(\mathfrak{C}^{[m(M+2)]})(\C)$ as in Corollary~\ref{CorEquidistr}. We will prove $\mu_1 \not= \mathscr{D}|_{\mathfrak{C}^{[m(M+2)]}}^*\mu_2$. Assuming this, then there exist a constant $\epsilon > 0$ and a function $f_1 \in \mathscr{C}_{\mathrm{c}}^0(\mathfrak{C}^{[m(M+2)], \mathrm{an}})$ such that
\begin{equation}\label{EqIntegralDiff}
\left| \int_{\mathfrak{C}^{[m(M+2)], \mathrm{an}}} f_1\mu_1 - \int_{\mathfrak{C}^{[m(M+2)], \mathrm{an}}} f_1 \mathscr{D}|_{\mathfrak{C}^{[m(M+2)]}}^*\mu_2  \right| > 2\epsilon.
\end{equation}
Moreover, since $\mathscr{D}|_{\mathfrak{C}^{[m(M+2)]}}$ is generically finite, it is not hard to show that one can choose an $f_1$ satisfying the following property: There exists a unique $f_2 \in \mathscr{C}_{\mathrm{c}}^0( \mathscr{D}(\mathfrak{C}^{[m(M+2)]})^{\mathrm{an}})$ such that $f_1 = f_2 \circ \mathscr{D}$.% We omit the details for this.% Moreover, we can make choices such that $f_1$ is supported in the non-degeneracy locus $\mathfrak{C}^{[m(M+2)],*}(\C)$ and $f_2$ is supported in $\mathscr{D}(\mathfrak{C}^{[m(M+2)]})^*(\C)$; see above Theorem~\ref{ThmEquidistr} for the non-degeneracy locus.

Now let us prove $\mu_1 \not= \mathscr{D}|_{\mathfrak{C}^{[m(M+1)]}}^*\mu_2$.\footnote{It is for this purpose that we need the $\mathfrak{C}^{[m]}$ before constructing the two desired non-degenerate subvarieties linked by the Faltings--Zhang map.}

For the point $\mathbf{x} \in (\mathfrak{C}^{[m]})^{\mathrm{sm}}(\C)$ from \eqref{EqNonDegPoint}, denote by $\Delta_\mathbf{x}$ the point $(\mathbf{x},\ldots,\mathbf{x})$ in $(\mathfrak{C}^{[m]})^{[M+1]}(\C)$. Then $(\mathbf{x},\Delta_\mathbf{x}) \in (\mathfrak{C}^{[m]} \times_S (\mathfrak{C}^{[m]})^{[M+1]})(\C)$, which is furthermore a smooth point by \eqref{EqGenericSmoothness}. We have $(\mu_1)_{(\mathbf{x}, \Delta_\mathbf{x})} \not= 0$ by Lemma~\ref{LemmaNonDegPointProduct}.

On the other hand, %the definition of the Betti rank \eqref{EqBettiRank} yields a trivial bound
%\[
%\mathrm{rank}_{\mathrm{Betti}}\left( \mathscr{D}(\mathfrak{C}^{[m(M+1)]}), \mathscr{D}(\Delta_\mathbf{x}) \right) \le 2 \dim_{\C} T_{\mathscr{D}(\Delta_\mathbf{x})}\mathscr{D}(\mathfrak{C}^{[m(M+1)]}).
%\]
%But 
$\mathscr{D}_M^{\cA^{[m]}}(\Delta_\mathbf{x})$ is the origin of fiber of $(\cA^{[m]})^{[M]} \rightarrow S$ in question (which we call $(\cA_s^m)^M$), so $\mathscr{D}_M^{\cA^{[m]}}|_{\mathfrak{C}^{[m(M+1)]}}^{-1}(\mathscr{D}_M^{\cA^{[m]}}(\Delta_\mathbf{x}))$ contains the diagonal of $\mathfrak{C}_s^m \subseteq \cA_s^m$ in $(\cA_s^m)^M$ (which for the moment we denote by $\Delta_{\mathfrak{C}_s^m}$).

Therefore for the morphism $\mathscr{D} = (\mathrm{id}, \mathscr{D}_M^{\cA^{[m]}}) $ from \eqref{EqFZFinal}, $\mathscr{D}|_{\mathfrak{C}^{[m]} \times_S (\mathfrak{C}^{[m]})^{[M+1]}}^{-1}(\mathbf{x}, \mathscr{D}_M^{\cA^{[m]}}(\Delta_{\mathbf{x}}))$ contains $(\mathbf{x}, \Delta_{\mathfrak{C}_s^m})$. In particular 
Thus $\dim \mathscr{D}|_{\mathfrak{C}^{[m]} \times_S (\mathfrak{C}^{[m]})^{[M+1]}}^{-1}(\mathbf{x}, \mathscr{D}_M^{\cA^{[m]}}(\Delta_{\mathbf{x}})) > 0$, and so the linear map
\[
\mathrm{d}\mathscr{D}|_{\mathfrak{C}^{[m]} \times_S (\mathfrak{C}^{[m]})^{[M+1]}} \colon T_{(\mathbf{x},\Delta_\mathbf{x})}(\mathfrak{C}^{[m]} \times_S \mathfrak{C}^{[m(M+1)]}) \rightarrow T_{(\mathbf{x},\mathscr{D}_M^{\cA^{[m]}}(\Delta_\mathbf{x}))}\mathscr{D}(\mathfrak{C}^{[m]} \times_S \mathfrak{C}^{[m(M+1)]})
\]
has non-trivial kernel. Thus $(\mathscr{D}|_{\mathfrak{C}^{[m(M+1)]}}^*\mu_2)_{(\mathbf{x},\Delta_{\mathbf{x}})} = 0$.

Thus we get $\mu_1 \not= \mathscr{D}|_{\mathfrak{C}^{[m(M+2)]}}^*\mu_2$ by looking at their evaluations at $(\mathbf{x},\Delta_{\mathbf{x}})$. Hence we are done for this step.

\noindent{\boxed{Step~3}} Prove some height lower bounds on $\mathfrak{C}^{[m(M+2)]}$ or $\mathscr{D}(\mathfrak{C}^{[m(M+2)]})$.

We apply the equidistribution result, or more precisely Corollary~\ref{CorEquidistr},  twice.

Apply Corollary~\ref{CorEquidistr} to $\mathfrak{C}^{[m]} \times_S (\mathfrak{C}^{[m]})^{[M+1]}$, $f_1$ and $\epsilon$. We thus obtain a constant $\delta_{\epsilon,1}>0$ and a Zariski closed proper subset $Z_1:=Z_{f_1,\epsilon}$ of $\mathfrak{C}^{[m]} \times_S (\mathfrak{C}^{[m]})^{[M+1]}$.  Apply Corollary~\ref{CorEquidistr} to $\mathscr{D}(\mathfrak{C}^{[m]} \times_S (\mathfrak{C}^{[m]})^{[M+1]})$, $f_2$ and $\epsilon$. We thus obtain a constant $\delta_{\epsilon,2}>0$ and a Zariski closed proper subset $Z_2:=Z_{f_2,\epsilon}$ of $\mathscr{D}(\mathfrak{C}^{[m]} \times_S (\mathfrak{C}^{[m]})^{[M+1]})$.

Let $\delta:= \min\{\delta_{\epsilon,1}, \delta_{\epsilon,2}\} > 0$, and let $Z= Z_1 \bigcup \mathscr{D}|_{\mathfrak{C}^{[m]} \times_S (\mathfrak{C}^{[m]})^{[M+1]}}^{-1}(Z_2) \bigcup Z_3$, where $Z_3$ is the largest Zariski closed subset of $\mathfrak{C}^{[m]} \times_S (\mathfrak{C}^{[m]})^{[M+1]}$ on which $\mathscr{D}$ is not injective. Then $Z$ is Zariski closed in $X := \mathfrak{C}^{[m]} \times_S (\mathfrak{C}^{[m]})^{[M+1]} = (\mathfrak{C}^{[m]})^{[M+2]}$, and is proper because $\mathscr{D}|_{\mathfrak{C}^{[m]} \times_S (\mathfrak{C}^{[m]})^{[M+1]}}$ is generically injective. If a point $\mathbf{x} \in (\mathfrak{C}^{[m(M+2)]}\setminus Z)(\IQbar)$ is such that $\hat{h}_{\cL^{\boxtimes m(M+2)}}(\mathbf{x}) < \delta$ and $\hat{h}_{\cL^{\boxtimes m(M+1)}}(\mathscr{D}(\mathbf{x})) < \delta$, then case (ii) of Corollary~\ref{CorEquidistr} holds true for both $\mathbf{x}, f_1, \mu_1$ and $\mathscr{D}(\mathbf{x}), f_2, \mu_2$. Thus
\[
\left|\int_{\mathfrak{C}^{[m(M+2)], \mathrm{an}}} f_1\mu_1 - \frac{1}{\#O(\mathbf{x})}\sum_{y\in O(\mathbf{x})}f_1(y) \right| < \epsilon \text{ and } \left|\int_{\mathscr{D}(\mathfrak{C}^{[m(M+2)]})^{\mathrm{an}}} f_2 \mu_2 - \frac{1}{\#O(\mathscr{D}(\mathbf{x}))}\sum_{y\in O(\mathscr{D}(\mathbf{x}))}f_2(y) \right| < \epsilon
\]
where $O(\cdot)$ is the Galois orbit. 
But $\frac{1}{\#O(\mathbf{x})}\sum_{y\in O(\mathbf{x})}f_1(y) = \frac{1}{\#O(\mathscr{D}(\mathbf{x}))}\sum_{y\in O(\mathscr{D}(\mathbf{x}))}f_2(y)$ because $f_1 = f_2 \circ \mathscr{D}$ and $\mathscr{D}$ is injective on $\mathfrak{C}^{[m(M+2)]} \setminus Z$. So we have
\[
\left|\int_{\mathfrak{C}^{[m(M+2)], \mathrm{an}}} f_1\mu_1- \int_{\mathscr{D}(\mathfrak{C}^{[m(M+2)]})^{\mathrm{an}}} f_2 \mu_2 \right| \le 2\epsilon.
\]
This contradicts \eqref{EqIntegralDiff} because $f_1 = f_2 \circ \mathscr{D}$.

Hence for each point $\mathbf{x} \in (\mathfrak{C}^{[m(M+2)]}\setminus Z)(\IQbar)$, we are in one of the following alternatives.
\begin{enumerate}
\item[(i)] Either $\hat{h}_{\cL^{\boxtimes m(M+2)}}(\mathbf{x}) \ge \delta$,
\item[(ii)] or $\hat{h}_{\cL^{\boxtimes m(M+1)}}(\mathscr{D}(\mathbf{x})) \ge \delta$.
\end{enumerate}

\noindent{\boxed{Step~4}} Finish the proof with a similar argument to the proof of Proposition~\ref{PropNGPLargeCurve}.

Denote by $\pi \colon \cA^{[m(M+2)]} \rightarrow S$ the structural morphism. As $Z$ is proper Zariski closed in $X$, each irreducible component of $\bar{S \setminus \pi(\mathfrak{C}^{[m(M+2)]}\setminus Z)}$ has dimension $\le \dim S - 1$. Thus by induction hypothesis, it suffices to prove the proposition with $S$ replace by $S \setminus \pi(\mathfrak{C}^{[m(M+2)]}\setminus Z)$. Therefore we may and do assume the following:
\begin{equation}\label{EqLemmaCondHolds}
\text{For each }s \in S(\IQbar) \text{, we have }Z_s \not= \mathfrak{C}_s^{m(M+2)}.
\end{equation}
%Here $Z_P = Z \cap \pi^{-1}(P)$ and $X_P^{m(M+1)} = \mathfrak{C}^{[m(M+1)]} \cap \pi^{-1}(P)$ are the fibers over $P$.

%Let $P \in S(\IQbar) \subseteq \mathfrak{C}_g(\IQbar)$. Then the fiber of $X \rightarrow S$ over $P$ is $\mathfrak{C}_P - P$, with $\mathfrak{C}_P$ being the fiber of $\mathfrak{C}_g\rightarrow \mathbb{M}_g$ in which $P$ lies; see above Proposition~\ref{}. Thus
%$X_P^{m(M+1)} = (\mathfrak{C}_P - P)^{m(M+1)}$. For the fiber of $\mathscr{D}(\mathfrak{C}^{m(M+1)}) \subseteq (\cA^{[m]})^{[M]}$ over $P \in S(\IQbar)$, it is not hard to show that this fiber is precisely $\mathscr{D}(\mathfrak{C}^{m(M+1)}_P) = (\mathfrak{C}_P-\mathfrak{C}_P)^{mM}$ from the definition of $\mathscr{D}$ \eqref{EqDDMA}.

By \eqref{EqLemmaCondHolds} and Lemma~\ref{LemmaNogaAlon}, there exists a constant $c_2''$ such that the following property holds. If a subset $\Sigma \subseteq \mathfrak{C}_s$ has cardinality $\ge c_2''$, then $\Sigma^{m(M+2)} \not\subseteq Z_s$.  This number $c_2''$ depends only on $m(M+2)$, the degree of $\mathfrak{C}_s$, and the degree of $Z_s$. Hence $c_2''$ can be chosen to be independent of $s$.% The degree of $Z_P$ has an upper bound independent of $P \in S(\IQbar)$. The degree of $\mathfrak{C}_P - P$ equals the degree of $\mathfrak{C}_P$, and thus equals the degree of some fiber of $\mathfrak{C}_g\rightarrow \mathbb{M}_g$. But all fibers of $\mathfrak{C}_g\rightarrow \mathbb{M}_g$ have the same degree. So $c_2''$ does not depend on 

Let $c_3'' = \delta / 4m(M+2)$. Set $\Sigma := \{x \in \mathfrak{C}_s(\IQbar): \hat{h}_{\cL}(x) \le c_3'' \}$. It suffices to prove $\#\Sigma < c_2''$. Suppose not. Then there exist $x_1,\ldots,x_{m(M+2)} \in \Sigma$ such that $\mathbf{x}:= (x_1, \ldots, x_{m(M+2)}) \not\in Z_s$. Then $\hat{h}_{\cL^{\boxtimes m(M+2)}}(\mathbf{x}) = \sum_{i=1}^{m(M+2)} \hat{h}_{\cL}(x_i) \le m(M+2)c_3'' < \delta$. On the other hand, each component of $\mathscr{D}(\mathbf{x})$ is of the form $x_k$ or of the form $x_j-x_i$ for some $i$ and $j$, and $\hat{h}_{\cL}(x_j-x_i) \le 2\hat{h}_{\cL}(x_j) + 2\hat{h}_{\cL}(x_i) \le 4c_3''$. So $\hat{h}_{\cL^{\boxtimes m(M+1)}}(\mathscr{D}(\mathbf{x})) \le m(M+1) 4c_3'' < \delta$. Thus we have reached a contradiction to the height bounds at the end of Step~3. Hence we are done.
\end{proof}

%% Section 9
\section{Proof of the New Gap Principle and proof of Uniform Mordell--Lang for curves}\label{SectionNGPProof}

\subsection{Parametrizing space of Abel--Jacobi embeddings}\label{SubsectionParaSpAJ}
Let  $\pi \colon \mathfrak{C}_g \rightarrow \mathbb{M}_g$ be the universal curve of genus $g$.

Each closed point in $\mathfrak{C}_g(\IQbar)$ parametrizes a pair $(C,P)$ with $C$ a smooth curve of genus $g$ defined over $\IQbar$ and $P\in C(\IQbar)$. Each such pair determines an Abel--Jacobi embedding from a curve to its Jacobian $j_P \colon C \rightarrow \mathrm{Jac}(C)$, and all Abel--Jacobi embeddings arise in this way. Thus $\mathfrak{C}_g$ is the parametrizing space of Abel--Jacobi embeddings.

Let us take a closer look at this. Consider the pullback of the relative Jacobian $\mathrm{Jac}(\mathfrak{C}_g/\mathbb{M}_g) \rightarrow \mathbb{M}_g$ along the universal curve $\pi \colon \mathfrak{C}_g \rightarrow \mathbb{M}_g$:%, we obtain the desired abelian scheme% $\mathrm{Jac}(\mathfrak{C}_g/\mathbb{M}_g) \times_{\mathbb{M}_g} \mathfrak{C}_g \rightarrow \mathfrak{C}_g$. This is the desired abelian scheme %In the purpose of proving , Let $S$ be an irreducible subvariety of $\mathfrak{C}_g$, then we obtain the desired abelian scheme
\begin{equation}\label{EqAbSchPullBack}
\mathfrak{J}_{\mathfrak{C}_g} := \mathrm{Jac}(\mathfrak{C}_g/\mathbb{M}_g) \times_{\mathbb{M}_g} \mathfrak{C}_g  \rightarrow \mathfrak{C}_g .
\end{equation}
This is an abelian scheme of relative dimension $g$.

\begin{prop}\label{PropTautoAbelJacobi}
There is a tautological family $\mathfrak{C} \rightarrow \mathfrak{C}_g$, with $\mathfrak{C} \subseteq \mathfrak{I}_{\mathfrak{C}_g}$ a closed $\mathfrak{C}_g$-immersion, satisfying the following property. 
For each $P \in \mathfrak{C}_g(\IQbar)$, 
\begin{equation}\label{EqFiberLargeCurve}
\text{the fiber $\mathfrak{C}_P$ (of $\mathfrak{C} \rightarrow \mathfrak{C}_g$ over $P$) is precisely $\mathfrak{C}_{\pi(P)} - P$},
\end{equation}
 with $\mathfrak{C}_{\pi(P)}$ being the fiber of $\pi \colon \mathfrak{C}_g\rightarrow \mathbb{M}_g$ in which $P$ lies.
 
 Moreover, $(X \subseteq \cA \rightarrow S) := (\mathfrak{C} \subseteq \mathfrak{I}_{\mathfrak{C}_g} \rightarrow \mathfrak{C}_g)$ satisfies all the hypotheses of Theorem~\ref{ThmNonDegXA}, and $\iota|_{\mathfrak{C}\times_S S'}$ is generically finite for the modular map $\iota \colon \mathfrak{J}_{\mathfrak{C}_g}\rightarrow \mathfrak{A}_g$ and for each irreducible subvariety $S' \subseteq S$.
\end{prop}

Before proving Proposition~\ref{PropTautoAbelJacobi}, let us summarize the morphisms and families in the following diagram.
\begin{equation}\label{EqP1}
\xymatrix{
\mathfrak{C} ~ \ar@{^(->}[r] \ar[rd] &  \mathrm{Jac}(\mathfrak{C}_g/\mathbb{M}_g) \times_{\mathbb{M}_g} \mathfrak{C}_g = \mathfrak{J}_{\mathfrak{C}_g}  \ar[d] \ar[r]^-{p_1} \pullbackcorner & \mathrm{Jac}(\mathfrak{C}_g/\mathbb{M}_g) \ar[d] \\
& \mathfrak{C}_g  \ar[r]^-{\pi} & \mathbb{M}_g.
}
\end{equation}

\begin{proof}[Proof of Proposition~\ref{PropTautoAbelJacobi}]
 The projection to the first factor $\mathfrak{C}_g \times_{\mathbb{M}_g} \mathfrak{C}_g \rightarrow \mathfrak{C}_g$ and the morphism $\mathscr{D}_1 \colon \mathfrak{C}_g \times_{\mathbb{M}_g} \mathfrak{C}_g \rightarrow \mathrm{Jac}(\mathfrak{C}_g/\mathbb{M}_g)
$ from \eqref{EqFaltingZhang1} induce an $\mathbb{M}_g$-morphism
\begin{equation}
\lambda \colon \mathfrak{C}_g \times_{\mathbb{M}_g} \mathfrak{C}_g \rightarrow \mathrm{Jac}(\mathfrak{C}_g/\mathbb{M}_g) \times_{\mathbb{M}_g} \mathfrak{C}_g = \mathfrak{J}_{\mathfrak{C}_g}
\end{equation}
which over each point in $\mathbb{M}_g(\IQbar)$ becomes $(P,Q) \mapsto (Q-P, P)$. Set
\[
\mathfrak{C} := \lambda(\mathfrak{C}_g \times_{\mathbb{M}_g} \mathfrak{C}_g) \subseteq \mathfrak{J}_{\mathfrak{C}_g}.
\]
Then $\mathfrak{C}$ is a subvariety of $\mathfrak{J}_{\mathfrak{C}_g}$ which dominates $\mathfrak{C}_g$. The claim \eqref{EqFiberLargeCurve} is not hard to check by definition of $\mathfrak{C}$.  In particular, $\dim \mathfrak{C}= \dim \mathfrak{C}_g+1 = 3g-1$, and the geometric generic fiber of $\mathfrak{C} \rightarrow \mathfrak{C}_g$ is irreducible.%$s(P) \in \mathbb{M}_g(\C)$ is the projection of $P$ under $\mathfrak{C}_g \rightarrow \mathbb{M}_g$.
\
Hypotheses (a)-(c) of Theorem~\ref{ThmNonDegXA} are easy to check for $\mathfrak{C} \subseteq \mathfrak{J}_{\mathfrak{C}_g} \rightarrow  \mathfrak{C}_g$. Thus it remains to check that $\iota|_{\mathfrak{C}\times_{\mathfrak{C}_g}S'}$ is generically finite for the modular $\iota \colon \mathfrak{J}_{\mathfrak{C}_g}  \rightarrow \mathfrak{A}_g$ and for each irreducible subvariety $S' \subseteq \mathfrak{C}_g$. But in this case, $\iota$ is the composite of the natural projection $p_1 \colon \mathfrak{J}_{\mathfrak{C}_g}   \rightarrow \mathrm{Jac}(\mathfrak{C}_g/\mathbb{M}_g)$ in \eqref{EqP1} 
and the quasi-finite morphism $\mathrm{Jac}(\mathfrak{C}_g/\mathbb{M}_g) \rightarrow \mathfrak{A}_g$ from \eqref{EqUnivJac}. Thus it suffices to check that $p_1|_{\mathfrak{C}\times_{\mathfrak{C}_g}S'}$ is generically finite. This is true, because $\dim \mathfrak{C}\times_{\mathfrak{C}_g}S' = \dim S' + 1$, and $p_1(\mathfrak{C}\times_{\mathfrak{C}_g}S') = \mathscr{D}_1(\mathfrak{C}_g \times_{\mathbb{M}_g}S')$ which has dimension $1 + \dim S'$.
\end{proof}

\subsection{Proof of the New Gap Principle}\label{SubsectionProofNGP}
We are ready to prove the New Gap Principle, Theorem~\ref{ThmNGP}.

The proof is by applying the following Proposition~\ref{PropNGPprep} to the $\mathfrak{C} \subseteq \mathfrak{I}_{\mathfrak{C}_g} \rightarrow \mathfrak{C}_g$ constructed in $\mathsection$\ref{SubsectionParaSpAJ}. Proposition~\ref{PropNGPprep} is proved, following the same line of \cite[Prop.2.3]{DGHBog}, by a combination of Proposition~\ref{PropNGPLargeCurve} and Proposition~\ref{PropUnifBog}.\footnote{Alternatively one can also prove Theorem~\ref{ThmNGP}, up to some finite set of uniformly bounded cardinality, by combining (a slight change of) \cite[Prop.7.1]{DGHUnifML} and \cite[Thm.3]{KuehneUnifMM} in the same way. Having this extra finite set does not matter for Theorem~\ref{ConjMazur}. We take the current approach to have a ``cleaner'' statement for the New Gap Principle.}%, and its proof follows .

We retain the notations from the beginning of $\mathsection$\ref{SectionHtIneqEqdistr}. In particular, $S$ is a quasi-projective irreducible variety and $\pi \colon \cA \rightarrow S$ is an abelian scheme of relative dimension $g$;  $\cL$ is a relatively ample line bundle on $\cA/S$ with $[-1]^*\cL \cong \cL$, and $\cM$ is a line bundle over a compactification $\bar{S}$ of $S$. Assume all data are defined over $\IQbar$, and thus we have two height functions $\hat{h}_{\cL} \colon \cA(\IQbar) \rightarrow [0,\infty)$ from \eqref{EqFiberwiseNTHeight} and $h_{\bar{S},\cM} \colon S(\IQbar) \rightarrow \R$ given by the Height Machine \eqref{EqHtMachine}.

\medskip
Let $\iota \colon \cA \rightarrow \mathfrak{A}_g$ be the modular map from \eqref{EqModularMap2}.

\begin{prop}\label{PropNGPprep}
Let $\mathfrak{C} \subseteq \cA$ be an irreducible subvariety satisfying the following properties. Each fiber $\mathfrak{C}_s$ of $\mathfrak{C} \rightarrow S$ is an irreducible curve which generates $\cA_s$ and is not a translate of an elliptic curve, and $\iota|_{\mathfrak{C}\times_S S'}$ is quasi-finite for all subvarieties $S'$ of $S$.

Then there exist constants $c_1$ and $c_2$ such that for each $s \in S(\IQbar)$, we have
\begin{equation}
\#\left\{x \in \mathfrak{C}_s(\IQbar) : \hat{h}_{\cL}(x) \le c_1 \max\{1, h_{\bar{S},\cM}(s)\} \right\} < c_2.
\end{equation}
\end{prop}
\begin{proof}
By Proposition~\ref{PropNGPLargeCurve}, there exist constants $c_1',c_2'$ and $c_3'$ such that for each $s \in S(\IQbar)$, we have
\begin{equation}\label{EqSet1}
\{x \in \mathfrak{C}_s(\IQbar) : \hat{h}_{\cL}(x) \le c_1' \max\{1, h_{\bar{S},\cM}(s)\} - c_3'\}
\end{equation}
has cardinality $< c_2'$.

By Proposition~\ref{PropUnifBog}, there exist constants $c_2''$ and $c_3''$ such that for each $s \in S(\IQbar)$, we have
\begin{equation}\label{EqSet2}
\{x \in \mathfrak{C}_s(\IQbar): \hat{h}_{\cL}(x) \le c_3''\}
\end{equation}
has cardinality $< c_2''$.

Now set
\begin{equation}
c_1 := \min\left\{\frac{c_3''}{\max\{1,2c_3'/c_1'\}}, \frac{c_1'}{2}\right\} \quad \text{ and } \quad c_2:= \max\{c_2',c_2''\}.
\end{equation}
We will prove that these are the desired constants.

To prove this, it suffices to prove the following claim.

\noindent\textbf{Claim:} If $x \in \mathfrak{C}_s(\IQbar)$ satisfies $\hat{h}_{\cL}(x) \le c_1\max\{1, h_{\bar{S},\cM}(s)\}$, then $x$ is in either the set \eqref{EqSet1} or the set \eqref{EqSet2}.

Let us prove this claim. 
Suppose $x \in \mathfrak{C}_s(\IQbar)$ is not in \eqref{EqSet1} or \eqref{EqSet2}, \textit{i.e.}, $\hat{h}_{\cL}(x) > c_1' \max\{1, h_{\bar{S},\cM}(s)\} - c_3'$ and $\hat{h}_{\cL}(x) > c_3''$. We wish to prove $\hat{h}_{\cL}(x) > c_1 \max\{1, h_{\bar{S},\cM}(s)\}$.

We split up to two cases on whether $\max\{1, h_{\bar{S},\cM}(s)\} \le \max\{1,2c_3'/c_1'\}$.

In the first case, \textit{i.e.}, $\max\{1, h_{\bar{S},\cM}(s)\} \le \max\{1,2c_3'/c_1'\}$, we have
\[
\hat{h}_{\cL}(x) > c_3'' \ge c_3'' \frac{\max\{1, h_{\bar{S},\cM}(s)\}}{\max\{1,2c_3'/c_1'\}} = \frac{c_3''}{\max\{1,2c_3'/c_1'\}} \max\{1, h_{\bar{S},\cM}(s)\} \ge c_1 \max\{1, h_{\bar{S},\cM}(s)\} .
\]
In the second case,  \textit{i.e.}, $\max\{1, h_{\bar{S},\cM}(s)\} > \max\{1,2c_3'/c_1'\}$, we have $c_1' \max\{1, h_{\bar{S},\cM}(s)\}  - c_3' \ge (c_1'/2) \max\{1, h_{\bar{S},\cM}(s)\}$ and hence
\[
\hat{h}_{\cL}(x) > \frac{c_1'}{2}\max\{1, h_{\bar{S},\cM}(s)\} \ge c_1 \max\{1, h_{\bar{S},\cM}(s)\}.
\]
Hence we are done.
\end{proof}

Now we are ready to prove the New Gap Principle by applying Proposition~\ref{PropNGPprep} to the family constructed in $\mathsection$\ref{SubsectionParaSpAJ}.
\begin{proof}[Proof of the New Gap Principle (Theorem~\ref{ThmNGP})]
%A specialization argument using Masser's result \cite{masser1989specializations} reduces this theorem to $F = \overline{\Q}$; see \cite[Lem.3.1]{DGHBog}. From now on, we may and do assume $F = \IQbar$.

Let $\mathfrak{C} \subseteq \mathfrak{I}_{\mathfrak{C}_g} \rightarrow \mathfrak{C}_g$ be as in Proposition~\ref{PropTautoAbelJacobi}.  Then we are in the situation of Proposition~\ref{PropNGPprep}, with $(\cA \rightarrow S) = (\mathfrak{I}_{\mathfrak{C}_g} \rightarrow \mathfrak{C}_g)$. % for $\mathfrak{C} \subseteq \cA \rightarrow S = \mathfrak{C}_g$.
 Let us now explain how the line bundles are chosen.

Recall from $\mathsection$\ref{SubsectionHtUnivFamily} that we have fixed a line bundle $\mathfrak{L}$ on $\mathrm{Jac}(\mathfrak{C}_g/\mathbb{M}_g)$ ample over $\mathbb{M}_g$ such that $[-1]^*\mathfrak{L} \cong \mathfrak{L}$, and an ample line bundle $\mathfrak{M}$ over $\bar{\mathbb{M}_g}$, a compactification of $\mathbb{M}_g$. Both line bundles are defined over $\IQbar$.

Use the notations in the diagram \eqref{EqP1}.

Set $\cL:= p_1^*\mathfrak{L}$ for the natural projection $p_1 \colon \mathfrak{J}_{\mathfrak{C}_g} \rightarrow \mathrm{Jac}(\mathfrak{C}_g/\mathbb{M}_g)$ from \eqref{EqP1}. Then $\cL$ is a relatively ample line bundle on $\cA/S$ such that $[-1]^*\cL \cong \cL$.

The morphism $\pi \colon \mathfrak{C}_g \rightarrow \mathbb{M}_g$ extends to a morphism $\bar{\pi} \colon \bar{\mathfrak{C}_g} \rightarrow \bar{\mathbb{M}_g}$ defined over $\IQbar$, with $\bar{\mathfrak{C}_g}$ a suitable compactification of $\mathfrak{C}_g$. Set $\cM:=\bar{\pi}^*\mathfrak{M}$.

Now we are ready to invoke Proposition~\ref{PropNGPprep} and get the following conclusion. 
For each $P \in \mathfrak{C}_g(\IQbar)$, we have
\begin{equation}\label{EqNGPM2}
\#\left\{x \in \mathfrak{C}_P(\IQbar) : \hat{h}_{\cL}(x) \le c_1\max\{1, h_{\bar{\mathfrak{C}_g}, \cM}(P)\} \right\} < c_2.
\end{equation}
Moreover the fiber of $\mathfrak{I}_{\mathfrak{C}_g} \rightarrow \mathfrak{C}_g$ over $P \in \mathfrak{C}_g(\IQbar)$, denoted by $(\mathfrak{I}_{\mathfrak{C}_g})_P$, satisfies that $p_1((\mathfrak{I}_{\mathfrak{C}_g})_P) = \mathrm{Jac}(\mathfrak{C}_g/\mathbb{M}_g)_{\pi(P)}$ (the fiber of $\mathrm{Jac}(\mathfrak{C}_g/\mathbb{M}_g) \rightarrow \mathbb{M}_g$ over $\pi(P)$). As $\mathrm{Jac}(\mathfrak{C}_g/\mathbb{M}_g)_{\pi(P)} = \mathrm{Jac}(\mathfrak{C}_{\pi(P)})$ for $\mathfrak{C}_{\pi(P)}$ defined below \eqref{EqFiberLargeCurve}, we then have $\hat{h}_{\cL}|_{(\mathfrak{I}_{\mathfrak{C}_g})_P} = \hat{h}_{p_1^*\mathfrak{L}}|_{(\mathfrak{I}_{\mathfrak{C}_g})_P} = \hat{h}_{\mathfrak{L}|_{\mathrm{Jac}(\mathfrak{C}_{\pi(P)})}}$.

\medskip

% Moreover, $p_1(\mathfrak{C}_P) = \mathfrak{C}_{s(P)}$ by \eqref{EqFiberLargeCurve}. So $\hat{h}_{\cL}|_{\mathfrak{C}_P} = \hat{h}_{\cL|_{\mathfrak{C}_P}} = \hat{h}_{p_1^*\mathfrak{L}|_{\mathfrak{C}_P}} = \hat{h}_{\mathfrak{L}|_{p_1(\mathfrak{C}_P)}} = \hat{h}_{\mathfrak{L}|_{\mathfrak{C}_{s(P)}}}$.

For each $s \in \mathbb{M}_g(\IQbar)$ and each $P \in \mathfrak{C}_s(\IQbar)$, we have $P \in \mathfrak{C}_g(\IQbar)$ with $\pi(P) = s$. By \eqref{EqFiberLargeCurve}, we have $\mathfrak{C}_P = \mathfrak{C}_s - P$. So each $x \in \mathfrak{C}_P(\IQbar)$ is $Q-P$ with some $Q \in \mathfrak{C}_s(\IQbar)$. We have seen $\hat{h}_{\cL}|_{(\mathfrak{I}_{\mathfrak{C}_g})_P} = \hat{h}_{\mathfrak{L}|_{\mathrm{Jac}(\mathfrak{C}_{s})}}$ from the last paragraph. Moreover $h_{\bar{\mathfrak{C}_g}, \cM}(P) = h_{\bar{\mathfrak{C}_g}, \bar{\pi}^*\mathfrak{M}}(P) = h_{\bar{\mathbb{M}_g}, \mathfrak{M}}(s)$. Thus \eqref{EqNGPM2} becomes $\#\left\{Q-P \in \mathfrak{C}_s(\IQbar) - P : \hat{h}_{\mathfrak{L}}(Q-P) \le c_1 \max\{1, h_{\bar{\mathbb{M}_g},\mathfrak{M}}(s)\} \right\} < c_2$. So
\[
\#\left\{Q \in \mathfrak{C}_s(\IQbar) : \hat{h}_{\mathfrak{L}}(Q-P) \le c_1 \max\{1, h_{\bar{\mathbb{M}_g},\mathfrak{M}}(s)\} \right\} < c_2,
\]
which is precisely the desired cardinality bound.
\end{proof}

One can recover the weaker statements \cite[Prop.7.1]{DGHUnifML} and \cite[Thm.3]{KuehneUnifMM} using the same argument: instead of Proposition~\ref{PropNGPprep}, it suffices to apply the weaker Proposition~\ref{PropNGPLargeCurve} (resp. Proposition~\ref{PropUnifBog}) to the family constructed in $\mathsection$\ref{SubsectionParaSpAJ} in order to get \cite[Prop.7.1]{DGHUnifML} (resp. \cite[Thm.3]{KuehneUnifMM}).
%\begin{rmk}\label{RmkNGPWeakVersions}
%\end{rmk}

\subsection{Proof of Uniform Mordell--Lang}\label{SubsectionProofUML}
We are now ready to prove Theorem~\ref{ConjMazur} by a packing argument using Theorem~\ref{ThmLargePointsFamily} and Theorem~\ref{ThmNGP}.

A specialization argument using Masser's result \cite{masser1989specializations} reduces this theorem to $F = \overline{\Q}$; see \cite[Lem.3.1]{DGHBog}. From now on, we may and do assume $F = \IQbar$.

Let $C$ be a smooth curve of genus $g$ defined over $\IQbar$, $P_0 \in C(\IQbar)$ and $\Gamma$ a subgroup of $\mathrm{Jac}(C)(\IQbar)$ of rank $\rho$. Then there exists $s \in \mathbb{M}_g(\IQbar)$ which parametrizes the curve $C$. Thus the fiber of $\pi \colon \mathfrak{C}_g \rightarrow  \mathbb{M}_g$ over $s$, $\mathfrak{C}_s$, is isomorphic to $C$ over $\IQbar$. We thus view $P_0 \in \mathfrak{C}_s(\IQbar) \subseteq \mathfrak{C}_g(\IQbar)$, and $\Gamma$ a subgroup of $\mathrm{Jac}(\mathfrak{C}_s)(\IQbar)$ of rank $\rho$. Notice that $\pi(P_0) = s$.

There exists a surjective quasi-finite \'{e}tale morphism $S \rightarrow \mathbb{M}_g$ such that $\mathfrak{C}_g \times_{\mathbb{M}_g} S \rightarrow S$ admits a section. This induces a morphism $\sigma \colon S \rightarrow \mathfrak{C}_g$. Thus 
we can construct the following morphism, which should be seen as the Abel--Jacobi embedding in family, $\mathfrak{C}_g \times_{\mathbb{M}_g}S \xrightarrow{(\mathrm{id},\sigma)} \mathfrak{C}_g \times_{\mathbb{M}_g} \mathfrak{C}_g \xrightarrow{\mathscr{D}_1} \mathrm{Jac}(\mathfrak{C}_g/\mathbb{M}_g)$. For each $s \in \mathbb{M}_g(\IQbar)$, an irreducible component of the image (which we call $\mathfrak{C}$) is $\mathfrak{C}_s - P_s$ for some $P_s \in \mathfrak{C}_s(\IQbar)$.%quasi-section of $\mathfrak{C}_g \rightarrow \mathbb{M}_g$ by \cite[]{}, \textit{i.e.} there is a morphism $S \rightarrow \mathfrak{C}_g$, with $S$ an affine scheme, which factors through a . 

%Let $\cL = p_1^*\mathfrak{L}$ and $\cM = \bar{\pi}^*\mathfrak{M}$ be as in the proof of Theorem~\ref{ThmNGP} above. 
Apply Theorem~\ref{ThmLargePointsFamily} to $(\cA \rightarrow S) = (\mathrm{Jac}(\mathfrak{C}_g/\mathbb{M}_g) \rightarrow \mathbb{M}_g)$, $\mathfrak{C}$, $\mathfrak{L}$ and $\mathfrak{M}$. Then we have
\[
\#\left\{P - P_s \in (\mathfrak{C}_s-P_s)(\IQbar) \cap \Gamma  : \hat{h}_{\mathfrak{L}}(P - P_s) > c\max\{1, h_{\bar{\mathbb{M}_g}, \mathfrak{M}}(s)\} \right\} \le  c^{\rho}
\]
 for some constant $c$ depending only on the family and the line bundles. %Thus
% \[
% \#\left\{P \in \mathfrak{C}_s(\IQbar)  : P-P_s \in \Gamma,~ \hat{h}_{\mathfrak{L}}(P - P_s) > c\max\{1, h_{\bar{\mathbb{M}_g}, \mathfrak{M}}(s)\} \right\} \le  c^{\rho}
% \]

 Set  $ R:=  (c \max\{1, h_{\bar{\mathbb{M}_g}, \mathfrak{M}}(s)\})^{1/2}$.
 
 %Then the bound above becomes, as in the last paragraph of the proof of Theorem~\ref{ThmNGP} above,
 %\[
 %\#\left\{P \in (\mathfrak{C}_s-P_0)(\IQbar) \cap \Gamma  : \hat{h}_{\mathfrak{L}}(P) >R^2 \right\} \le  c^{\rho}.
 %\]
 \medskip
 
 We start by the case where $P_0 = P_s$. Then it remains to prove
 \begin{equation}\label{EqSubsetCenterBalls}
\#\left\{P - P_s \in (\mathfrak{C}_s-P_s)(\IQbar) \cap \Gamma : \hat{h}_{\mathfrak{L}}(P - P_s) \le c\max\{1, h_{\bar{\mathbb{M}_g}, \mathfrak{M}}(s)\} \right\} \le c^{1+\rho}
\end{equation}
 up to increasing $c$.
 
Let $c_1$ and $c_2$ be as in Theorem~\ref{ThmNGP}. Set $r = (c_1\max\{1, h_{\bar{\mathbb{M}_g}, \mathfrak{M}}(s)\})^{1/2} / 2$. 
Consider the real vector space $\Gamma\otimes \R$ endowed with the Euclidean norm $| \cdot | = \hat{h}_{\mathfrak{L}}^{1/2}$. 
By an elementary ball packing argument, any subset of $\Gamma \otimes\R$ contained in a closed ball of radius $R$ centered at $0$ is covered by at most $(1+2R/r)^{\rho}$ closed balls of radius $r$ centered at the elements $P - P_s$ of the given subset \eqref{EqSubsetCenterBalls}; see \cite[Lem.6.1]{Remond:Decompte}. Thus the number of balls in the covering is at most $(1+4\sqrt{c c_1^{-1}})^{\rho}$. % after increasing $c$.
 But each closed ball of radius $r$ centered at some $P - P_s$ in \eqref{EqSubsetCenterBalls} contains at most $c_2$ elements by Theorem~\ref{ThmNGP}. So \eqref{EqSubsetCenterBalls} contains at most $c_2 (1+4\sqrt{c c_1^{-1}})^{\rho} \le c^{1+\rho}$ elements for a suitable $c$. So we are done for this case.

\medskip

For arbitrary $P_0$, let $\Gamma'$ be the subgroup of $\mathrm{Jac}(\mathfrak{C}_s)(\IQbar)$ generated by $\Gamma$ and $P_0-P_s$. Then $\mathrm{rk}\Gamma' \le \rho + 1$. For any $P \in C(\IQbar) - P_0$, we have $P + P_0-P_s \in \mathfrak{C}_s(\IQbar) - P_s$. So $\#(\mathfrak{C}_s - P_0)(\IQbar) \cap \Gamma \le \# (\mathfrak{C}_s - P_s)(\IQbar) \cap \Gamma'$, which is $\le c^{2+\rho} \le (c^2)^{1+\rho}$ by the previous case. So we are done by replacing $c$ with $c^2$.

%% Section 10
%\section{End of proof of the uniform Mordell--Lang conjecture}

%% Section 11
\section{Further aspects}\label{SectionOtherAspects}
\subsection{Relative Bogomolov Conjecture}\label{SubsectionRelBog}
In this subsection, we state the \textit{Relative Bogomolov Conjecture} and explain how it induces \cite[Thm.3]{KuehneUnifMM}, known as the \textit{Uniform Bogomolov Conjecture} for curves embedded into Jacobians.

The Relative Bogomolov Conjecture is a folklore conjecture. The formulation we state here is taken from \cite[Conj.1.1]{DGHBog}.

Let $S$ be an irreducible quasi-projective variety. Let $\pi \colon \cA \rightarrow S$ be an abelian scheme
of relative dimension $g \ge 1$. Let $\cL$ be a relatively
ample line bundle on $\cA/S$ such that $[-1]^*\cL\cong \cL$. Assume that $S$, $\pi$ and $\cL$ are all defined over $\IQbar$. We thus have a fiberwise N\'{e}ron--Tate height $\hat{h}_{\cL} \colon \cA(\IQbar) \rightarrow [0,\infty)$ as defined in \eqref{EqFiberwiseNTHeight}.

We will use the following notation. For any subvariety $X$ of
$\cA$ that
dominates $S$, denote by $X_{\overline{\eta}}$ the geometric generic
fiber of $X$.
In particular, $\cA_{\overline{\eta}}$ is an abelian variety  over an
algebraically closed field.

\begin{conj}[Relative Bogomolov Conjecture]
  \label{ConjRelBog}
  Let $X$ be a subvariety of $\cA$ defined over
  $\IQbar$ that dominates $S$. Assume that
  $X_{\overline{\eta}}$ is irreducible
  % each irreducible
  % component of
  and not contained in any proper
  algebraic subgroup of $\cA_{\overline{\eta}}$. If $\codim_{\cA} X >
  \dim S$, then there exists $\epsilon > 0$ such that
  \[ X(\epsilon; \cL) := \{ x \in X(\IQbar) : \hat{h}_{\cL}(x) \le
    \epsilon \}
  \]
  is not Zariski dense in $X$.
\end{conj}

The name \textit{Relative Bogomolov Conjecture} is reasonable: the
same statement with $\epsilon = 0$ is precisely the relative
Manin--Mumford conjecture proposed by Pink~\cite[Conj.6.2]{Pink}
and  Zannier~\cite{ZannierBook}, which is proved when
% $\dim S = 1$
$\dim X=1$ in a series of papers \cite{MasserZannierTorsionPointOnSqEC, MASSER2014116, MasserZannierRelMMSimpleSur, CorvajaMasserZannier2018, MasserZannierRMMoverCurve}. The Betti map is heavily used in these works.

The classical Bogomolov conjecture, proved by Ullmo \cite{Ullmo} and
S.~Zhang \cite{ZhangEquidist}, is precisely
Conjecture~\ref{ConjRelBog} for $\dim S = 0$. When $\dim S = 1$ and
$X$ is the image of a section, Conjecture~\ref{ConjRelBog} is
equivalent to  S.~Zhang's conjecture in his 1998 ICM note
\cite[$\mathsection$4]{zhang1998small} if $\cA_{\overline{\eta}}$ is
simple and is proved by DeMarco--Mavraki
\cite[Thm.1.4]{DeMarcoMavraki} if $\cA \rightarrow S$ is isogenous
to a fiber product of elliptic surfaces. The latter proof 
was simplified and strengthened by DeMarco--Mavraki in \cite{DeMarcoMavraki21}: in \cite{DeMarcoMavraki} the authors reduced their Theorem~1.4 to the case of torsion points treated by \cite{MASSER2014116}, whereas in \cite{DeMarcoMavraki21} the authors proved this result (among other generalizations \cite[Thm.1.5]{DeMarcoMavraki21}) directly. 

K\"{u}hne \cite{KuehneRBC} recently proved Conjecture~\ref{ConjRelBog} if $\cA \rightarrow S$ is isogenous
to a fiber product of elliptic surfaces. In general Conjecture~\ref{ConjRelBog} is still open. Notice that the proof of Proposition~\ref{PropUnifBog} can be adapted to show that Conjecture~\ref{ConjRelBog} holds true for $\mathfrak{C}^{[m]}$ for some suitable $m \gg 1$; see the conclusion of Step~3.

\medskip
Using the proof pattern of Proposition~\ref{PropNGPLargeCurve}, it is not hard to show that the Relative Bogomolov Conjecture implies the Uniform Bogomolov Conjecture for curves embedded into Jacobians \cite[Thm.3]{KuehneUnifMM}.

\begin{prop}\label{LemmaRelBogUnifBog}
Conjecture~\ref{ConjRelBog} implies Proposition~\ref{PropUnifBog}%. In particular, Conjecture~\ref{ConjRelBog} implies \cite[Thm.3]{KuehneUnifMM}.
, and hence \cite[Thm.3]{KuehneUnifMM}
\footnote{See the end of $\mathsection$\ref{SubsectionProofNGP}.}
.
\end{prop}
\begin{proof}
We prove this proposition by induction on $\dim S$. The proof of the base step $\dim S = 0$ is contained in the induction step.

Let $\mathfrak{C} \subseteq \cA \rightarrow S$ and $\cL$ be from Proposition~\ref{PropUnifBog}. Consider the fibered  powers $\mathfrak{C}^{[M]}$, $\cA^{[M]}$ and $\cL^{\boxtimes M}$ over $S$. As $\mathfrak{C} \not= \cA$, we have
\[
\codim_{\cA^{[M]}} \mathfrak{C}^{[M]} = M(g-1) > \dim S
\]
for some $M\gg 1$. Thus we can apply Conjecture~\ref{ConjRelBog} to $\mathfrak{C}^{[M]} \subseteq \cA^{[M]} \rightarrow S$ and $\cL^{\boxtimes M}$ to conclude that
\[
\mathfrak{C}^{[M]}(\epsilon; \cL^{\boxtimes M}) := \{ \mathbf{x} \in \mathfrak{C}^{[M]}(\IQbar) : \hat{h}_{\cL^{\boxtimes M}}(\mathbf{x}) \le
    \epsilon \}
\]
is not Zariski dense in $\mathfrak{C}^{[M]}$, for some $\epsilon > 0$. 

Set $Z$ to be the Zariski closure of $\mathfrak{C}^{[M]}(\epsilon; \cL^{\boxtimes M})$. Then each irreducible component of $\overline{ S \setminus \pi(\mathfrak{C}^{[M]} \setminus Z) }$ has dimension $\le \dim S - 1$. Thus by induction hypothesis, it suffices to prove the lemma with $S$ replaced by $ S \setminus \pi(\mathfrak{C}^{[M]} \setminus Z)$. Thus we may and do assume the following:
\begin{equation}\label{EqRelBogUnifBog}
\text{For each }s \in S(\IQbar) \text{, we have }Z_s \not= \mathfrak{C}_s^{M}.
\end{equation}

By \eqref{EqRelBogUnifBog} and Lemma~\ref{LemmaNogaAlon}, there exists a constant $c_2''$ such that the following property holds. If a subset $\Sigma \subseteq \mathfrak{C}_s(\IQbar)$ has cardinality $\ge c_2''$, then $\Sigma^M \not\subseteq Z_s$.  This number $c_2''$ depends only on $M$, the degree of $\mathfrak{C}_s$, and the degree of $Z_s$. Hence $c_2''$ can be chosen to be independent of $s$.% The degree of $Z_P$ has an upper bound independent of $P \in S(\IQbar)$. The degree of $\mathfrak{C}_P - P$ equals the degree of $\mathfrak{C}_P$, and thus equals the degree of some fiber of $\mathfrak{C}_g\rightarrow \mathbb{M}_g$. But all fibers of $\mathfrak{C}_g\rightarrow \mathbb{M}_g$ have the same degree. So $c_2''$ does not depend on 

Let $c_3'' := \epsilon /M$, and $\Sigma = \{ x \in \mathfrak{C}_s(\IQbar) : \hat{h}_{\cL}(x) \le c_3''\}$. It suffices to prove $\#\Sigma < c_2''$. Suppose not. Then there exist $x_1, \ldots, x_M \in \Sigma$ such that $\mathbf{x} := (x_1,\ldots,x_M) \not\in Z_s$. Then $\hat{h}_{\cL^{\boxtimes M}}(\mathbf{x}) = \sum \hat{h}_{\cL}(x_i) \le Mc_3'' = \epsilon$. This contradicts the definition of $Z$. Hence we are done.
\end{proof}

\subsection{High dimensional subvarieties}\label{SubsectionHighDimUnifML}
Let $F$ be a field of characteristic $0$ with $F = \overline{F}$. In this subsection, all varieties and line bundles are assumed to be defined over $F$.

Let $A$ be an abelian variety of dimension $g$, and let $L$ be an ample line bundle on $A$. By a \textit{coset} in $A$ we mean the translate of an abelian subvariety of $A$ by a point in $A(F)$.

Let $X$ be a closed irreducible subvariety. Faltings \cite{Faltings:DAAV} and Hindry \cite{Hindry:Lang} proved the following \textit{Mordell--Lang Conjecture}. 
If $\Gamma$ is a finite rank subgroup of $A(F)$, then there exist finitely many $x_1,\ldots,x_n \in X(F) \cap \Gamma$ and $B_1,\ldots,B_n$ abelian subvarieties of $A$, with $x_i + B_i \subseteq X$ and $(x_i+B_i)(F) \cap \Gamma$ not a finite set for each $i$, such that
\begin{equation}\label{EqHighDimML}
X(F) \cap \Gamma = \bigcup_{i=1}^n (x_i+B_i)(F) \cap \Gamma \coprod S
\end{equation}
for a finite set $S$. 
In particular, each $B_i$ satisfies $\dim B_i > 0$.
%The abelian subvarieties $B_i$'s are not necessarily distinct. But the $x_i+B_i$'s can be chosen to satisfy the following properties: they are pairwise distinct, and $(x_i+B_i)(F) \cap \Gamma$ is not a finite set for each $i$.% The union in \eqref{EqHighDimML} can be divided into two parts: the part with $\dim B_i > 0$ and the \textit{discrete} part with $\dim B_i = 0$.%If $B_i = \{0\}$ is the origin (trivial abelian variety), then $x_i$ is said to be in the discrete part of $X(F) \cap \Gamma$.

\medskip

\begin{conj}\label{ConjHighDimML}
$\#S \le c(g,\deg_L X , \deg_L A)^{\mathrm{rk}\Gamma + 1}$.
\end{conj}

This conjecture is a natural generalization of Theorem~\ref{ConjMazur}. Indeed, let $C$ be a curve of genus $g \ge 2$ and $P_0 \in C(F)$ as in Theorem~\ref{ConjMazur}. Let $\mathrm{Jac}(C)$ be the Jacobian of $C$ and view $C-P_0$ as a curve in $\mathrm{Jac}(C)$ via the Abel--Jacobi embedding based at $P_0$. As $g\ge 2$, $C-P_0$ does not contain any positive dimensional coset in $\mathrm{Jac}(C)$. Thus for $X= C-P_0$ and  $A=\mathrm{Jac}(C)$, \eqref{EqHighDimML}  becomes $(C-P_0)(F) \cap \Gamma = S$. It is a classical result that there exists a line bundle $L$ on $\mathrm{Jac}(C)$ with $\deg_L \mathrm{Jac}(C) = g!$ and $\deg_L (C-P_0) = \deg_L C = g$.\footnote{In fact, here we do not need the explicit functions in $g$. So it suffices to use the existence of the universal curve $\mathfrak{C}_g \rightarrow \mathbb{M}_g$ to conclude that both $\deg_L\mathrm{Jac}(C)$ and $\deg_L C$ can be assumed to depend only on $g$.} Hence Conjecture~\ref{ConjHighDimML} implies Theorem~\ref{ConjMazur}.

\medskip

We will see that Conjecture~\ref{ConjHighDimML} self improves to the following stronger conjecture proposed by David--Philippon \cite[Conj.1.8]{DaPh:07}.%, which replaces $\#S$ by $n+\#S$ on the left hand side of the inequality in Conjecture~\ref{ConjHighDimML}. 

%But we have more. Dan Abramovich suggested to the author another conjecture, which implies Conjecture~\ref{ConjHighDimML} and \cite[Conj.1.8]{DaPh:07}. To state the conjecture we need some preparation.

\begin{conjbis}{ConjHighDimML}\label{ConjHighDimMLStr}
There exists a partition \eqref{EqHighDimML} such that $n+\#S \le c(g,\deg_L X, \deg_L A)^{\mathrm{rk}\Gamma+1}$.
%$\deg_L (X(F) \cap \Gamma)^{\mathrm{Zar}}  \le c(g,\deg_L X, \deg_L A)^{\mathrm{rk}\Gamma+1}$.
\end{conjbis}

%If $F = \bar{F}$, then from \eqref{EqHighDimML} we have $(X(F) \cap \Gamma)^{\mathrm{Zar}} = \bigcup_{i=1}^n (x_i+B_i) \coprod S$ for well-chosen $B_i$.
%Thus the left hand side of the inequality in Conjecture~\ref{ConjHighDimMLStr} is $\sum_{i=1}^n \deg_L B_i + \#S$. Thus Conjecture~\ref{ConjHighDimMLStr} immediately implies Conjecture~\ref{ConjHighDimML} and \cite[Conj.1.8]{DaPh:07} if $F = \bar{F}$. For arbitrary $F$, we can reduce to the case $F = \bar{F}$ by noticing $X(F) \cap \Gamma = X(\bar{F}) \cap (A(F)\cap\Gamma)$ and $\mathrm{rk}(A(F) \cap \Gamma) \le \mathrm{rk}\Gamma$. 

\medskip
%Without loss of generality we may and do assume $F = \bar{F}$ in the rest of this subsection; otherwise it suffices to replace $\Gamma$ by $A(F) \cap \Gamma$. 
Let us show that Conjecture~\ref{ConjHighDimML} self improves to Conjecture~\ref{ConjHighDimMLStr}. To do this, we recall the \textit{Ueno locus} or the \textit{Kawamata locus} defined as follows. Consider the union $\bigcup_{x+B \subseteq X}(x+B)$, where $x$ runs over $A(F)$ and $B$ runs over abelian subvarieties of $A$ with $\dim B > 0$. Bogomolov \cite[Thm.1]{BogomolovUeno} proved that this union is a closed subset of $X$. %\footnote{There are only finitely many such $B$'s with $x+B \subseteq X$ for some $x \in X(\IQbar)$, maximal for this property. Thus the Ueno locus is $\bigcup_{i=1}^n (X_i+B_i)$ for some abelian subvarieties $B_i$ of $A$ with $\dim B_i > 0$.}
 Denote by $X^{\circ}$ its complement in $X$. It is not hard to check that the $S$ from \eqref{EqHighDimML} is $X^{\circ}(F) \cap \Gamma$.

Set $\Sigma(X;A)$ to be the set of abelian subvarieties $B \subseteq A$ with $\dim B > 0$ satisfying: $x+B \subseteq X$ for some $x \in A(F)$, and $B$ is maximal for this property. Then Bogomolov \cite[Thm.1]{BogomolovUeno} says that $\Sigma(X;A)$ is a finite set.

\begin{lemma}\label{LemmaConjHighDimMLStrWeakEqui}
If Conjecture~\ref{ConjHighDimML} holds true for all $X$  (in addition to $\Gamma$, $A$ and $L$), then Conjecture~\ref{ConjHighDimMLStr} also holds true.
\end{lemma}
\begin{proof}
%As, we only need to handle the other part $\bigcup_{i=1}^n (x_i+B_i)(F) \cap \Gamma$ if Conjecture~\ref{ConjHighDimML} holds true.
%We prove the lemma by induction on $\dim X$. The base step $\dim X = 0$ is trivial.
For arbitrary $X$. 
By Bogomolov \cite[Thm.1]{BogomolovUeno}, each $B \in \Sigma(X;A)$ satisfies $\deg_L B \le c_3$ for some constant $c_3 = c_3(g,\deg_L X) > 0$. Thus $\# \Sigma(X;A) \le c_4 = c_4(g,\deg_L X, \deg_L A)$ by \cite[Prop.4.1]{Remond:Decompte}.

The Ueno locus of $X$ defined above is a finite union $\bigcup_{B \in \Sigma(X;A)}(X_B+B)$, with $X_B$ constructed as follows. 
%an irreducible subvariety of $A$ such that $\mathrm{Stab}_A(X_B)$ is finite and that $X_B\cap B$ is finite. The subvariety $X_B$ can be constructed as follows. Notice that $X_B+B$ is an irreducible component of $\bigcap_{b \in B(F)}(X-b)$. 
Let $B^\perp$ be a complement of $B$, \textit{i.e.} $B \cap B^\perp$ is finite and $B+B^\perp = A$. It is possible to choose such a $B^\perp$ with $\deg_L B^\perp \le c_5'(g, \deg_L A, \deg_L B)$; see \cite{MW:Complement}. Set $X_B:=\bigcap_{b \in B(F)}(X-b) \bigcap B^\perp$. This intersection must be a finite intersection (of at most $\dim X \le g$ members) by dimension reasons. Recall that $\deg_L B \le c_3(g,\deg_L X)$. So $\deg_L X_B \le c_5(g,\deg_LA, \deg_L X)$ by B\'{e}zout's Theorem. In particular $X_B$ has $\le c_5$ irreducible components $X_{B,1},\ldots,X_{B,m_B}$.

% I think you might want to take X_B as an irreducible component of X^1_B \cap C, where C \subset A is such that the restriction of the addition morphism induces an isogeny B \times C \to A. It is possible to find such a C with deg_L C \leq c(g,deg_L A,deg_L B), e.g. by Bertrand, "Duality on tori and multiplicative dependence relations", Theorem 3.

As the $B_i$'s in \eqref{EqHighDimML} satisfies $x_i+B_i \subseteq X$ and $\dim B_i > 0$, we may and do assume $B_i \in \Sigma(X;A)$ by definition of the Ueno locus. It is not hard to check that  the finite set $S$ from \eqref{EqHighDimML} is $X^{\circ}(F) \cap \Gamma$. So \eqref{EqHighDimML} becomes
\begin{equation}\label{EqConjHighDimMLStrOrg}
X(F) \cap \Gamma = \bigcup_{B\in \Sigma(X;A)} \bigcup_{j=1}^{n_B} (x_{B,j}+B)(F) \cap \Gamma  \coprod S.
\end{equation}
Moreover, each $x_{B,j}$ can be chosen to be in $X_B^{\circ}(F) \cap \Gamma$, where $X_B^{\circ} = \bigcup_{k=1}^{m_B}X_{B,k}^{\circ}$. See \cite[Lem.4.6]{Remond:Decompte}; notice that $p|_{X_B}$ is finite for the quotient $p \colon A \rightarrow A/B$. In particular, $n_B \le \#X_B^{\circ}(F) \cap \Gamma$.% Indeed, we have $x_{B,j} \in Y(F) \cap \Gamma$. Let $p \colon A \rightarrow A/B$ be the quotient. Then  $p(x_{B,j}) \in p(Y)^{\circ}(F)$ by \cite[Lem.4.6]{Remond:Decompte}. As $p|_Y$ is a finite morphism, we then have $x_{B,j} \in Y^{\circ}(F)$.

We need to take a closer look at the union in \eqref{EqConjHighDimMLStrOrg}. First, we have seen $\# \Sigma(X;A) \le c_4(g,\deg_L X, \deg_L A)$ above.

Next we bound $n_B$ for each $B \in \Sigma(X;A)$. Let $B \in \Sigma(X;A)$. %; in particular, $\dim B > 0$. If $\dim X_B = \dim X$, then $X_B = X$ and hence $B = \{0\}$, which contradicts $\dim B > 0$. Thus $\dim X_B < \dim X$. So we can apply induction hypothesis to conclude that 
Conjecture~\ref{ConjHighDimML} applied to each irreducible component $X_{B,k}$ of $X_B$ says that $\#X_{B,k}^{\circ}(F) \cap \Gamma \le c^{\mathrm{rk}\Gamma+1}$ for some $c = c(g,\deg_L X_{B,k}, \deg_L A) > 0$. But we have seen that $X_B$ has $\le c_5(g,\deg_L A, \deg_L X)$ components and that $\deg_L X_{B,k} \le \deg_L X_B \le c_5(g,\deg_L X)$ before. So $\#X_B^{\circ}(F) \cap \Gamma \le c_6(g,\deg_L X, \deg_L A)^{\mathrm{rk}\Gamma+1}$. In particular, $n_B \le c_6(g,\deg_L X, \deg_L A)^{\mathrm{rk}\Gamma+1}$ for each $B \in \Sigma(X;A)$.

%Finally, let us look at each $x_{B,j}+B$ appearing in the union from \eqref{EqConjHighDimMLStrOrg}. If $x_{B,j}+B \subsetneq X$, then we can apply the induction hypothesis to $x_{B,j}+B$ to conclude that $\deg_L ((x_{B,j}+B)(F) \cap \Gamma)^{\mathrm{Zar}} \le c_7(g,\deg_L (x_{B,j}+B), \deg_L A)^{\mathrm{rk}\Gamma+1}$. If $x_{B,j}+B = X$, then 

%In particular, $n_B \le \#(Y\cap B) \cdot \#(\bar{Y}^{\circ}(F) \cap \bar{\Gamma})$. Hence $n_B \le \deg_LY \deg_L B \cdot \#(\bar{Y}^{\circ}(F) \cap \bar{\Gamma})$ by B\'{e}zout's Theorem, and thus $n_B \le c_6(g,\deg_LX)\cdot \#(\bar{Y}^{\circ}(F) \cap \bar{\Gamma})$ by the previous two paragraphs.

%For each $B \in \Sigma(X;A)$, we can apply induction hypothesis because $\dim A/B < \dim A$ to bound $\#(\bar{Y}^{\circ}(F) \cap \bar{\Gamma})$. So from the previous paragraph, we have  $n_B \le c_7(g, \deg_L X, \deg_L A)^{\mathrm{rk}\Gamma+1}$.

By \eqref{EqConjHighDimMLStrOrg}, Conjecture~\ref{ConjHighDimMLStr} is equivalent to
\begin{equation}\label{EqConjHighDimMLStr}
\sum_{B \in \Sigma(X;A)} n_B  + \#S \le c(g,\deg_L X, \deg_L A)^{\mathrm{rk}\Gamma+1}.
\end{equation}
We have bounded $\#\Sigma(X;A)$ and $n_B$ in terms of $g,\deg_L X, \deg_L A$ and $\mathrm{rk}\Gamma$ as desired.  It remains to bound $\#S$. But this is exactly what Conjecture~\ref{ConjHighDimML} claims. Hence we are done.
%If Conjecture~\ref{ConjHighDimML} holds true, then \eqref{EqConjHighDimMLStr} is reduced to
%\begin{equation}\label{EqConjHighDimMLStrCont}
%\sum_{B \in \Sigma(X;A)} n_B \deg_LB \le c(g,\deg_L X, \deg_L A)^{\mathrm{rk}\Gamma+1}
%\end{equation}
%up to modifying the constant $c$.
\end{proof}

\small
A natural question is whether the left hand side of Conjecture~\ref{ConjHighDimMLStr} can be replaced by $\deg_L (X(F)\cap\Gamma)^{\mathrm{Zar}}$, which is $\sum_{i=1}^n \deg_L B_i + \#S$ in view of \eqref{EqHighDimML} for some well-chosen $B_i$. Unfortunately this is not possible in general, because in the proof of Lemma~\ref{LemmaConjHighDimMLStrWeakEqui} $(B(F)\cap\Gamma)^{\mathrm{Zar}}$ could be any abelian subvariety of $B$ and hence we cannot expect a bound for its degree. For example, let $X = A = E^2$ be the square of an elliptic curve defined over $\IQbar$. The graph $E_N \subseteq E^2$ of $[N]\colon E \rightarrow E$ then has degree $N^2$. Take a subgroup $\Gamma$ of $E_N(\IQbar)$ of rank $1$, then $\deg (X(\IQbar) \cap \Gamma)^{\mathrm{Zar}} = \deg E_N = N^2$. This provides a counterexample. 

However, the proof of Lemma~\ref{LemmaConjHighDimMLStrWeakEqui} suggests that this is the only obstacle. In fact, as $\deg_L B \le c_3(g,\deg_L X)$ for each $B \in \Sigma(X;A)$, in the proof \eqref{EqConjHighDimMLStr} can be improved to $\sum_{B \in \Sigma(X;A)} n_B \deg_L B  + \#S \le c(g,\deg_L X, \deg_L A)^{\mathrm{rk}\Gamma+1}$. Thus if Conjecture~\ref{ConjHighDimML} holds true for all $X$ and $\Gamma$ (in addition to $A$ and $L$), then the following conjecture holds true.\footnote{Conjecture~\ref{ConjHighDimMLStrStr} is suggested to me by Dan Abramovich.}
\begin{conjbisbis}{ConjHighDimML}\label{ConjHighDimMLStrStr}
If $\Gamma$ is saturate for each $B \in \Sigma(X;A)$, \textit{i.e.} $(B(F)\cap\Gamma)^{\mathrm{Zar}} = B$ for each $B \in \Sigma(X;A)$, then $\deg_L (X(F) \cap \Gamma)^{\mathrm{Zar}}  \le c(g,\deg_L X, \deg_L A)^{\mathrm{rk}\Gamma+1}$.
\end{conjbisbis}

On the other hand, Conjecture~\ref{ConjHighDimMLStrStr} implies both Conjecture~\ref{ConjHighDimML} and Conjecture~\ref{ConjHighDimMLStr}. Indeed by dimension reasons and the assumption $F=\bar{F}$, for any finite rank subgroup $\Gamma$ of $A(F)$ and any abelian subvariety $B$ of $A$, there exists a subgroup $\Gamma_B \supseteq \Gamma$ of rank $\le \mathrm{rk}\Gamma + \dim B \le \mathrm{rk}\Gamma + g$ such that $\Gamma_B$ is saturate for $B$. Applying this successively to each $B \in \Sigma(X;A)$, we get a subgroup $\Gamma_X \supseteq \Gamma$ of rank $\le \mathrm{rk}\Gamma + g \#\Sigma(X;A) \le \mathrm{rk}\Gamma + g \cdot c_4(g,\deg_LX,\deg_LA)$ which is saturate for all $B\in \Sigma(X;A)$. \textit{Assume Conjecture~\ref{ConjHighDimMLStrStr}}. Then $\sum_{B \in \Sigma(X;A)} n_B \deg_L B + \#S \le c(g,\deg_L X, \deg_L A)^{\mathrm{rk}\Gamma_X+1} \le c^{\mathrm{rk}\Gamma+gc_4+1} \le (c^{gc_4+1})^{\mathrm{rk}\Gamma+1}$. Thus $n+\#S = \sum_{B \in \Sigma(X;A)} n_B + \#S  \le (c^{gc_4+1})^{\mathrm{rk}\Gamma+1}$. Hence Conjecture~\ref{ConjHighDimMLStr} and Conjecture~\ref{ConjHighDimML} both hold true with $c$ replaced by $c^{gc_4+1}$.

\normalsize

\medskip

As in the case of curves, to prove Conjecture~\ref{ConjHighDimML} it suffices to work with $F = \IQbar$ by a standard specialization argument using Masser's result \cite{masser1989specializations}. So from now on we assume $F = \IQbar$. We also assume that $L$ is symmetric; this can be achieved by replacing $L$ by $L \otimes [-1]^*L$ (and $\deg_{L \otimes [-1]^*L}(X) = 2^{\dim X} \deg_L(X)$).

R\'{e}mond proved the generalized Vojta's Inequality \cite[Thm.1.2]{Remond:Vojtasup} for points in  $X^{\circ}(\IQbar)$ and the generalized Mumford's Inequality \cite[Thm.3.2]{Remond:Decompte} for points in $X^{\circ}(\IQbar) \cap \Gamma$. As in the case for curves, these two generalized inequalities also yield the desired bound (the one in Conjecture~\ref{ConjHighDimML}) for the number of \textit{large points} in $X^{\circ}(\IQbar) \cap \Gamma$. A modified version of these results then reduces Conjecture~\ref{ConjHighDimML} to studying the \textit{small points}, \textit{i.e.} to prove a bound in the form of 
\begin{equation}\label{EqSmallPointsBoundHighDim1}
\left\{P \in X^{\circ}(\IQbar) \cap \Gamma : \hat{h}_L(P) \le c \max\{1,h_{\mathrm{Fal}}(A)\} \right\} \le c^{\mathrm{rk}\Gamma+1}
\end{equation}
for some $c = c(g,\deg_L X, \deg_L A) > 0$. We refer to \cite[Thm.6.8]{DaPh:07} for this reduction.\footnote{The constants $c_{\mathrm{NT}}$ and $h_1$ in \cite{DaPh:07} are bounded by $ \max\{1,h_{\mathrm{Fal}}(A)\}$ by an argument similar to \cite[(8.4) and (8.7)]{DGHUnifML}.}

But one can and should do one more step. Let $A'$ be the abelian subvariety of $A$ generated by $X-X$. Then $X \subseteq A' + Q$ for some $Q \in A(\IQbar)$. The subgroup $\Gamma'$ of $A(\IQbar)$ generated by $\Gamma$ and $Q$ has rank $\le \mathrm{rk}\Gamma + 1$. We have $(X-Q)^{\circ} = X^{\circ} - Q$ by definition of the Ueno locus, $(X^{\circ}(\IQbar)-Q) \cap \Gamma  \subseteq  (X^{\circ}(\IQbar) - Q) \cap \Gamma' = X^{\circ}(\IQbar) \cap \Gamma'$ and $\deg_L(X-Q) = \deg_L X$. So we may replace $X$ by $X-Q$, $A$ by $A'$ and $\Gamma$ by $\Gamma' \cap A'(\IQbar)$ and the constant $c$ in the conclusion by $c^2$. Thus Conjecture~\ref{ConjHighDimML} is reduced to the following bound: \textit{Assume $X$ generates $A$}, then
\begin{equation}\label{EqSmallPointsBoundHighDim}
\left\{P \in X^{\circ}(\IQbar) \cap \Gamma : \hat{h}_L(P) \le c \max\{1,h_{\mathrm{Fal}}(A)\} \right\} \le c^{\mathrm{rk}\Gamma+1}
\end{equation}
for some $c = c(g,\deg_L X, \deg_L A) > 0$. 
%; see \cite[Thm.3.5]{Remond:Decompte}.% Thus, following the same argument in $\mathsection$\ref{SubsectionProofUML}, 

\medskip
The following conjecture is a natural generalization of the New Gap Principle to high dimensional cases. Recall $X^{\circ}$ defined above Lemma~\ref{LemmaConjHighDimMLStrWeakEqui}.

\begin{conj}\label{ConjNGPHighDim}
Assume that $X$ generates $A$. %, \textit{i.e.} $X$ is not contained in any proper subgroup of $A$. 
There exist constants $c_1 = c_1(g,\deg_L X , \deg_L A) > 0$ and $c_2 = c_2(g,\deg_L X , \deg_L A) > 0$ satisfying the following property. For each $P_0 \in X^{\circ}(\IQbar)$, the set
\begin{equation}\label{EqNGPHighDim}
\left\{P \in X^{\circ}(\IQbar): \hat{h}_L(P-P_0) \le c_1 \max\{1,h_{\mathrm{Fal}}(A)\} \right\} 
\end{equation}
is contained in a proper Zariski closed subset $X' \subsetneq X$ with $\deg_L X'< c_2$.
\end{conj}

This conjecture is equivalent to the following conjecture, because $(X-P_0)^{\circ} = X^{\circ} - P_0$ and $\deg_L(X-P_0) = \deg_L X$.
\begin{conjbis}{ConjNGPHighDim}\label{ConjNGPHighDim2}
Assume that $X$ generates $A$. %, \textit{i.e.} $X$ is not contained in any proper subgroup of $A$. 
There exist constants $c_1 = c_1(g,\deg_L X , \deg_L A) > 0$ and $c_2 = c_2(g,\deg_L X , \deg_L A) > 0$ satisfying the following property. The set
\begin{equation}\label{EqNGPHighDim2}
\left\{P \in X^{\circ}(\IQbar): \hat{h}_L(P) \le c_1 \max\{1,h_{\mathrm{Fal}}(A)\} \right\}
\end{equation}
is contained in a proper Zariski closed subset $X' \subsetneq X$ with $\deg_L X'< c_2$.
\end{conjbis}

If Conjecture~\ref{ConjNGPHighDim2} holds true for all $A$, $X$ and $L$, then one can also handle points on the Ueno locus by induction.

It is not hard to prove that Conjecture~\ref{ConjNGPHighDim} implies \eqref{EqSmallPointsBoundHighDim} by induction on $\dim X$ and the standard packing argument as presented in $\mathsection$\ref{SubsectionProofUML}. Thus we have%expect the following statement: If Conjecture~\ref{ConjNGPHighDim} (or Conjecture~\ref{ConjNGPHighDim2}) holds true, then Conjecture~\ref{ConjHighDimML} holds true.
\begin{prop}
If Conjecture~\ref{ConjNGPHighDim} (or Conjecture~\ref{ConjNGPHighDim2}) holds true, then Conjecture~\ref{ConjHighDimML} holds true.
\end{prop}

\medskip

Let us briefly explain why the assumption ``$X$ generates $A$'' is added in Conjecture~\ref{ConjNGPHighDim} and Conjecture~\ref{ConjNGPHighDim2}. Suppose $X$ is contained in a proper abelian subvariety $A'$ of $A$, and $A = A' \times A''$. Then $h_{\mathrm{Fal}}(A) = h_{\mathrm{Fal}}(A') + h_{\mathrm{Fal}}(A'')$. We are free to create examples with $h_{\mathrm{Fal}}(A'')$ arbitrarily large, and \eqref{EqNGPHighDim2} ultimately says that all points in $X^{\circ}(\IQbar)$ are actually contained in a proper Zariski closed subset of $X$. This is impossible.% This is beyond the scope of Uniform Mordell--Lang and is a question in the nature of effective Mordell--Lang.

Next let us briefly explain why we do not directly conjecture the set \eqref{EqNGPHighDim2} to have cardinality $< c_2$. Suppose $A = B \times J$ a product of two abelian varieties and $X = Y \times C$, with $Y \subseteq B$ and $C \subseteq J$ the Abel--Jacobi embedding of a curve of genus $\ge 2$ via some point; in particular $0_J \in C(\IQbar)$. It is possible to choose an appropriate ample line bundle $L := L_B\boxtimes L_J$ such that $\deg_{L_J}J = g!$ and $\deg_{L_J}C = g$. Then for each $y \in Y^{\circ}(\IQbar)$, we have $(y,0_J) \in X^{\circ}(\IQbar)$. It is possible to choose $C$ and $J$ with $h_{\mathrm{Fal}}(J)$ arbitrarily large. If the set \eqref{EqNGPHighDim2}  has cardinality $< c_2$, then this yields $\#Y^{\circ}(\IQbar) < \infty$, and this is not true in general. Notice that in this example, the statement of Conjecture~\ref{ConjNGPHighDim2} can be related to the New Gap Principle for curves embedded into Jacobians (Theorem~\ref{ThmNGP}).

Finally, we remark that the problems revealed by both examples above do not occur \textit{if} we only consider the setup for Uniform Bogomolov, \textit{i.e.} replace $c_1 \max\{1,h_{\mathrm{Fal}}(A)\}$ from \eqref{EqNGPHighDim2} by a constant $c_3$. Indeed, in both examples above, eventually what prevents us to get a more general statement for Conjecture~\ref{ConjNGPHighDim2} is the fact $h_{\mathrm{Fal}}(A)$ can be arbitrarily large.

\bibliographystyle{alpha}
\bibliography{bibliographie}

%\begin{center}
%  \today
%\end{center}

\end{document}